%% file: arxiv_my_template.tex
\documentclass{article}

\input{prologue.tex}

\usepackage{graphicx,wrapfig,lipsum}
\usepackage[
singlelinecheck=false 
]{caption}
\usepackage{makecell}

\begin{document}

\title{Acceleration of Frank-Wolfe Algorithms with Open-Loop Step-Sizes}

\author{\name Elias Wirth \email          \texttt{\href{mailto:wirth@math.tu-berlin.de}{wirth@math.tu-berlin.de}}\\
      \addr Institute of Mathematics \\
      Berlin Institute of Technology \\
      Strasse des 17. Juni 135, Berlin, Germany
      \AND
      \name Thomas Kerdreux \email \texttt{\href{thomaskerdreux@gmail.com}{thomaskerdreux@gmail.com}}\\
      \addr Geolabe LLC \\
      1615 Central Avenue, Los Alamos, New Mexico, USA
      \AND
      \name Sebastian Pokutta \email \texttt{\href{mailto:pokutta@zib.de}{pokutta@zib.de}} \\
      \addr Institute of Mathematics \& AI in Society, Science, and Technology\\
      Berlin Institute of Technology \& Zuse Institute Berlin\\
      Strasse des 17. Juni 135, Berlin, Germany}
\maketitle

\begin{abstract}
Frank-Wolfe algorithms (FW) are popular first-order methods for solving constrained convex optimization problems that rely on a linear minimization oracle instead of potentially expensive projection-like oracles. Many works have identified accelerated convergence rates under various structural assumptions on the optimization problem and for specific FW variants when using line-search or short-step, requiring feedback from the objective function. Little is known about accelerated convergence regimes when utilizing open-loop step-size rules, a.k.a. FW with pre-determined step-sizes, which are algorithmically extremely simple and stable. Not only is FW with open-loop step-size rules not always subject to the same convergence rate lower bounds as FW with line-search or short-step, but in some specific cases, such as kernel herding in infinite dimensions, it has been empirically observed that FW with open-loop step-size rules enjoys to faster convergence rates than FW with line-search or short-step. We propose a partial answer to this unexplained phenomenon in kernel herding, characterize a general setting for which FW with open-loop step-size rules converges non-asymptotically faster than with line-search or short-step, and derive several accelerated convergence results for FW with open-loop step-size rules. Finally, we demonstrate that FW with open-loop step-sizes can compete with momentum-based open-loop FW variants.
\end{abstract}


\section{{Introduction}}
In this paper, we address the constrained convex optimization problem
\begin{equation}\label{eq:opt}\tag{OPT}
    \min_{x\in\cC}f(x),
\end{equation}
where $\cC\subseteq\R^d$ is a compact convex set and $f\colon \cC \to \R$ is a convex and $L$-smooth function. Let $x^* \in \argmin_{x\in \cC} f(x)$ be the constrained optimal solution.
A classical approach to addressing \eqref{eq:opt} is to apply \emph{projected gradient descent}. When the geometry of $\cC$ is too complex, the projection step can become computationally too expensive. In these situations, the \emph{Frank-Wolfe algorithm} (FW) \citep{frank1956algorithm}, a.k.a. the conditional gradients algorithm \citep{levitin1966constrained}, described in Algorithm~\ref{algo:fw}, is an efficient alternative, as it only requires first-order access to the objective $f$ and access to a linear minimization oracle (LMO) for the feasible region, that is, given a vector $c \in \R^d$, the LMO outputs $\argmin_{x \in \cC} \langle c, x\rangle$.
At each iteration, the algorithm calls the LMO, $p_t \in \argmin_{p\in \cC} \langle \nabla f (x_t), p-x_t\rangle$, and takes a step in the direction of the vertex $p_t$ to obtain the next iterate $x_{t+1}= (1-\eta_t) x_t + \eta_t p_t$. As a convex combination of elements of $\cC$, $x_t$ remains in the feasible region $\cC$ throughout the algorithm's execution. 
Various options exist for the choice of $\eta_t$, such as the \textit{open-loop step-size}\footnote{Open-loop is a term from control theory and here implies that there is no feedback from the objective function to the step-size.}, a.k.a. \emph{agnostic step-size}, rules $\eta_t = \frac{\ell}{t + \ell}$ for $\ell \in \N_{\geq 1}$ \citep{dunn1978conditional} or line-search $\eta_t \in \argmin_{\eta  \in [0,1]} f((1-\eta) x_t + \eta p_t)$.
Another classical approach, the \emph{short-step} step-size $\eta_t = \min\{ \frac{ \langle \nabla f(x_t), x_t - p_t\rangle}{L\|x_t - p_t\|_2^2},1\}$, henceforth referred to as short-step, is determined by minimizing a quadratic upper bound on the $L$-smooth objective function.
There also exist variants that adaptively estimate local $L$-smoothness parameters \citep{pedregosa2018step}.
\begin{algorithm}[t]
\SetKwInput{Input}{Input}
\SetKwInput{Output}{Output}
\caption{Frank-Wolfe algorithm (FW) \citep{frank1956algorithm}}\label{algo:fw}
  \Input{$x_0\in \cC$, step-sizes $\eta_t\in [0, 1]$ for $t\in\{0,\ldots, T-1\}$.}
  \hrulealg
  \For{$t= 0, \ldots, T-1 $}{
        $p_{t} \in \argmin_{p \in \cC} \langle\nabla f(x_{t}), p- x_{t}\rangle$\label{line:p_t_det}\\
        $x_{t+1} \gets (1 - \eta_{t}) x_t + \eta_t p_{t}$}
\end{algorithm}
\subsection{{Related work}}\label{sec:related_work}
Frank-Wolfe algorithms (FW) are first-order methods that enjoy various appealing properties \citep{jaggi2013revisiting}. They are easy to implement, projection-free, affine invariant \citep{lacoste2013affine,lan2013complexity,kerdreux2021affine,pena2021affine}, and iterates are sparse convex combinations of extreme points of the feasible region. These properties make FW an attractive algorithm for practitioners who work at scale, and FW appears in a variety of scenarios in machine learning, such as deep learning, optimal transport, structured prediction, and video co-localization \citep{ravi2018constrained,courty2016optimal,giesen2012optimizing,joulin2014efficient}. See \citet{braun2022conditional}, for a survey.
For several settings, FW with line-search or short-step admits accelerated convergence rates in primal gap $h_t = f(x_t) - f(x^*)$, where $x^*\in\argmin_{x\in\cC}f(x)$ is the minimizer of $f$:
Specifically, when the objective is strongly convex and the optimal solution lies in the relative interior of the feasible region, FW with line-search or short-step converges linearly \citep{guelat1986some}. Moreover, when the feasible region is strongly convex and the norm of the gradient of the objective is bounded from below by a nonnegative constant, FW with line-search or short-step converges linearly \citep{levitin1966constrained, demianov1970approximate, dunn1979rates}. Finally, when the feasible region and objective are strongly convex, FW with line-search or short-step converges at a rate of order $\cO(1/t^2)$, see also Table~\ref{table:references_to_results}.
However, the drawback of FW is its slow convergence rate
when the feasible region $\cC$ is a polytope and the optimal solution lies in the relative interior of an at least one-dimensional face $\cC^*$ of $\cC$. In this setting, for any $\epsilon > 0$, FW with line-search or short-step converges at a rate of order $\Omega (1/t^{1+\epsilon})$ \citep{wolfe1970convergence, canon1968tight}.
To achieve linear convergence rates in this setting, algorithmic modifications of FW are necessary \citep{lacoste2015global, garber2016linear,  braun2019blended, combettes2020boosting, garber2020revisiting}.

FW with open-loop step-size rules, on the other hand, has a convergence rate that is not governed by the lower bound of \citet{wolfe1970convergence}. Indeed, \citet{bach2021effectiveness} proved an asymptotic convergence rate of order $\cO(1/t^2)$ for FW with open-loop step-sizes in the setting of \citet{wolfe1970convergence}. However, proving that the latter result holds non-asymptotically remains an open problem.
Other disadvantages of line-search and short-step are that the former can be difficult to compute and the latter requires knowledge of the smoothness constant of the objective $f$. On the other hand, open-loop step-size rules are problem-agnostic and, thus, easy to compute.
Nevertheless, little is known about the settings in which FW with open-loop step-size rules admits acceleration, except for two momentum-exploiting variants that achieve convergence rates of order up to $\cO(1/t^2)$:
The \emph{primal-averaging Frank-Wolfe algorithm} (PAFW), presented in Algorithm~\ref{algo:pafw}, was first proposed by \citet{lan2013complexity} and later analyzed by \citet{kerdreux2021local}. PAFW employs the open-loop step-size $\eta_t = \frac{2}{t+2}$ and momentum to achieve convergence rates of order up to $\cO(1/t^2)$ when the feasible region is uniformly convex and the gradient norm of the objective is bounded from below by a nonnegative constant.
For the same setting, the \emph{momentum-guided Frank-Wolfe algorithm} (MFW) \citep{li2021momentum}, presented in Algorithm~\ref{algo:mfw}, employs the open-loop step-size $\eta_t = \frac{2}{t+2}$, and also incorporates momentum to achieve similar convergence rates as PAFW. In addition, MFW converges at a rate of order $\cO(1/t^2)$ when the feasible region is a polytope, the objective is strongly convex, the optimal solution lies in the relative interior of an at least one-dimensional face of $\cC$, and strict complementarity holds.
Finally, note that FW with open-loop step-size $\eta_t = \frac{1}{t+1}$ is equivalent to the kernel-herding algorithm \citep{bach2012equivalence}. For a specific infinite-dimensional kernel-herding setting, empirical observations in \citet[Figure 3, right]{bach2012equivalence} have shown that FW with open-loop step-size $\eta_t = \frac{1}{t+1}$ converges at the optimal rate of order $\cO(1/t^2)$, whereas FW with line-search or short-step converges at a rate of essentially $\Omega(1/t)$. Currently, both phenomena lack a theoretical explanation.

\subsection{{Contributions}}
\begin{table*}[t]
\footnotesize
\centering
\begin{tabular}{|c|c|c|c|c|c|c|}
    \hline
       References & Region $\cC$ & Objective $f$ & Location of $x^*$ &  Rate & Step-size rule \\
     \hline
     \rowcolor{LightCyan}
     \citep{jaggi2013revisiting}&-& - & unrestricted & $\cO(1/t)$  & any\\
     \hline
      \citep{guelat1986some}& -& str. con. & interior & $\cO(e^{-t})$ &  line-search, short-step \\
     \hline
     \bf{Theorem}~\ref{thm:interior} & -& str. con. & interior & $\cO(1/t^2)$ &  open-loop $\eta_t = \frac{4}{t+4}$\\
     \hline
     \rowcolor{LightCyan}
     \makecell{\citep{levitin1966constrained}\\
     \citep{demianov1970approximate}\\
     \citep{dunn1979rates}} & str. con. &\makecell{$\|\nabla f(x)\|_2 \geq\lambda > 0$\\
     for all $x\in\cC$} & unrestricted & $\cO(e^{-t})$ &  line-search, short-step\\
     \hline
     \rowcolor{LightCyan}
     \bf{Theorem}~\ref{thm:exterior} & str. con. & \makecell{$\|\nabla f(x)\|_2 \geq\lambda > 0$\\
     for all $x\in\cC$} & unrestricted & $\cO(1/t^2)$ &  open-loop $\eta_t = \frac{4}{t+4}$\\
      \hline
     \rowcolor{LightCyan}
     \bf{Remark}~\ref{rem:ol_linear} & str. con. & \makecell{$\|\nabla f(x)\|_2 \geq\lambda > 0$\\
     for all $x\in\cC$} & unrestricted & $\cO(1/t^{\ell/2})$ &  \Gape[0pt][2pt]{\makecell{open loop $\eta_t = \frac{\ell}{t+\ell}$\\ for $\ell\in\N_{\geq 4}$}}\\
     \hline
     \rowcolor{LightCyan}
     \bf{Remark}~\ref{rem:ol_linear} & str. con. & \makecell{$\|\nabla f(x)\|_2 \geq\lambda > 0$\\
     for all $x\in\cC$} & unrestricted & $\cO(e^{-t})$ &  constant\\
     \hline
    \citep{garber2015faster} & str. con. & str. con. & unrestricted & $\cO(1/t^2)$ &  line-search, short-step\\
     \hline
     \bf{Theorem}~\ref{thm:unrestricted} & str. con. & str. con. & unrestricted & $\cO(1/t^2)$ &  open-loop $\eta_t = \frac{4}{t+4}$\\
     \hline
     \rowcolor{LightCyan}
     \citep{wolfe1970convergence} & polytope & str. con. & interior of face & $\Omega(1/t^{1 +\eps})^*$  &  line-search, short-step\\
     \hline
     \rowcolor{LightCyan}
     \citep{bach2021effectiveness} & polytope & str. con. & interior of face & $\cO(1/t^2)^*$  &  open-loop $\eta_t = \frac{2}{t+2}$\\
     \hline
     \rowcolor{LightCyan}
     \bf{Theorem}~\ref{thm:polytope} & polytope & str. con. & interior of face & $\cO(1/t^2)$  &  open-loop $\eta_t = \frac{4}{t+4}$\\
     \hline
\end{tabular}
\normalsize
\caption{Comparison of convergence rates of FW for various settings. We denote the optimal solution by
$x^*\in\argmin_{x\in\cC}f(x)$. Convexity of $\cC$ and convexity and smoothness of $f$ are always assumed.
The big-O notation $\cO(\cdot)^*$ indicates that a result only holds asymptotically, "str. con." is an abbreviation for strongly convex, and "any" refers to line-search, short-step, and open-loop step-size $\eta_t = \frac{2}{t+2}$. Shading is used to group related results and our results are denoted in bold.}
\label{table:references_to_results}
\end{table*}
In this paper, we develop our understanding of settings for which FW with open-loop step-sizes admits acceleration. In particular, our contributions are five-fold:

First, we prove accelerated convergence rates of FW with open-loop step-size rules in settings for which FW with line-search or short-step enjoys accelerated convergence rates. Details are presented in Table~\ref{table:references_to_results}. Most importantly, when the feasible region $\cC$ is strongly convex and the norm of the gradient of the objective $f$  is bounded from below by a nonnegative constant for all $x\in\cC$, the latter of which is, for example, implied by the assumption that the unconstrained optimal solution $\argmin_{x\in\R^d}f(x)$ lies in the exterior of $\cC$, we prove convergence rates of order $\cO(1/t^{\ell/2})$ for FW with open-loop step-sizes $\eta_t=\frac{\ell}{t+\ell}$, where $\ell\in\N_{\geq 1}$. 

Second, under the assumption of strict complementarity, we prove that FW with open-loop step-sizes admits a convergence rate of order $\cO(1/t^2)$ in the setting of the lower bound due to \citet{wolfe1970convergence}, that is, we prove the non-asymptotic version of the result due to \citet{bach2021effectiveness}. We thus characterize a setting for which FW with open-loop step-sizes is non-asymptotically faster than FW with line-search or short-step, see the last three rows of Table~\ref{table:references_to_results} for details.

Third, we return again to the setting of the lower bound due to \citet{wolfe1970convergence}, for which both FW and MFW with open-loop step-sizes admit convergence rates of order $\cO(1/t^2)$, assuming strict complementarity. We demonstrate that the \emph{decomposition-invariant pairwise Frank-Wolfe algorithm} (DIFW) \citep{garber2016linear} and the \emph{away-step Frank-Wolfe algorithm} (AFW) \citep{guelat1986some, lacoste2015global} with open-loop step-sizes converge at rates of order $\cO(1/t^2)$ without the assumption of strict complementarity.

Fourth, we compare FW with open-loop step-sizes to PAFW and MFW for the problems of logistic regression and collaborative filtering. The results indicate that FW with open-loop step-sizes converges at comparable rates as or better rates than PAFW and MFW. 
This implies that faster convergence rates can not only be achieved by studying algorithmic variants of FW but can also be obtained via deeper understanding of vanilla FW and its various step-size rules.

Finally, we provide a theoretical analysis of the accelerated convergence rate of FW with open-loop step-sizes in the kernel herding setting of \citet[Figure 3, right]{bach2012equivalence}.

\subsection{{Outline}}
Preliminaries are introduced in Section~\ref{sec:preliminaries}.
In Section~\ref{sec:accelerated}, we present a proof blueprint for obtaining accelerated convergence rates for FW with open-loop step-sizes.
In Section~\ref{sec:ol_faster_than_ls_ss}, for the setting of the lower bound of \citet{wolfe1970convergence} and assuming strict complementarity, we prove that FW with open-loop step-sizes converges faster than FW with line-search or short-step.
In Section~\ref{sec:fw_variants}, we introduce two algorithmic variants of FW with open-loop step-sizes that admit accelerated convergence rates in the problem setting of the lower bound of \citet{wolfe1970convergence} without relying on strict complementarity.
In Section~\ref{sec:kernel_herding}, we prove accelerated convergence rates for FW with open-loop step-sizes in the infinite-dimensional kernel-herding setting of \citet[Figure 3, right]{bach2012equivalence}. Section~\ref{sec:numerical_experiments_main} contains the numerical experiments. Finally, we discuss our results in Section~\ref{sec:discussion}.

\section{{Preliminaries}}\label{sec:preliminaries}

Throughout, let $d\in \N$. Let $\zeroterm\in\R^d$ denote the all-zeros vector, let $\oneterm \in \R^d$ denote the all-ones vector, and let $\bar{\oneterm}\in \R^d$ be a vector such that $\bar{\oneterm}_i=0$ for all $i \in \{1, \ldots, \lceil d/2\rceil\}$ and $\bar{\oneterm}_{i}=1$ for all $i \in \{\lceil d/2\rceil + 1, \ldots, d\}$. For $i\in\{1,\ldots, d\}$, let $e^{(i)}\in \R^d$ be the $i$th unit vector such that $e^{(i)}_i = 1$ and $e^{(i)}_j = 0$ for all $j \in \{1,\ldots, d\} \setminus \{i\}$. Given a vector $x\in\R^d$, define its support as $\supp(x) = \{i \in \{1,\ldots, d\} \mid x_i \neq 0\}$. Let $I\in \R^{d\times d}$ denote the identity matrix.
Given a set $\cC \subseteq \R^d$, let $\aff(\cC)$,  $\conv(\cC)$, $\mathspan(\cC)$, and $\vertices(\cC)$ denote the affine hull, the convex hull, the span, and the set of vertices of $\cC$, respectively.
For $z\in \R^d$ and $\beta > 0$, the ball of radius $\beta$ around $z$ is defined as
$
    B_\beta(z):= \{x\in \R^d \mid \|x - z\|_2 \leq \beta\}.
$
For the iterates of Algorithm~\ref{algo:fw}, we denote the \emph{primal gap} at iteration $t\in \{0, \ldots, T\}$ by $h_t := f(x_t) - f(x^*)$, where $x^*\in\argmin_{x\in\cC}f(x)$. Finally, for $x\in\R$, let $[x]:= x - \lfloor x\rfloor$. We introduce several definitions.
\begin{definition}[Uniformly convex set]\label{def:unif_cvx_C}
Let $\cC \subseteq \R^d$ be a compact convex set, $\alpha_\cC >0$, and $q>0$. We say that $\cC$ is \emph{$(\alpha_\cC, q)$-uniformly convex} with respect to $\|\cdot\|_2$ if for all $x,y \in \cC$, $\gamma \in [0,1]$, and $z \in \R^d$ such that $\|z\|_2=1$, it holds that
$
    \gamma x + ( 1- \gamma) y + \gamma (1 - \gamma) \alpha_\cC \|x-y\|_2^q z \in \cC.
$
We refer to $(\alpha_\cC, 2)$-uniformly convex sets as \emph{$\alpha_\cC$-strongly convex sets}. 
\end{definition}
\begin{definition}[Smooth function]\label{def:smooth_f}
Let $\cC \subseteq \R^d$ be a compact convex set, let $f\colon \cC \to \R$ be differentiable in an open set containing $\cC$, and let $L > 0$. We say that $f$ is \emph{$L$-smooth} over $\cC$ with respect to $\|\cdot\|_2$ if for all $x,y\in \cC$, it holds that
$
    f(y) \leq f(x) + \langle \nabla f(x), y - x\rangle + \frac{L}{2}\|x-y\|_2^2.
$
\end{definition}
\begin{definition}[Hölderian error bound]\label{def:heb}
Let $\cC \subseteq \R^d$ be a compact convex set, let $f\colon \cC \to \R$ be convex,
let $\mu> 0$, and let $\theta \in [0, 1/2]$. We say that $f$ satisfies a \emph{$(\mu, \theta)$-Hölderian error bound} if for all $x\in \cC$ and $x^*\in\argmin_{x\in\cC}f(x)$, it holds that
\begin{align}\label{eq:heb_original}
    \mu(f(x)-f(x^*))^\theta \geq \min_{y\in\argmin_{z\in\cC}f(z)} \|x-y\|_2 .
\end{align}
\end{definition}
Throughout, for ease of notation, we assume that $x^*\in\argmin_{x\in\cC}f(x)$ is unique. This follows, for example, from the assumption that $f$ is strictly convex. When $x^*\in \argmin_{x\in \cC}f(x)$ is unique, \eqref{eq:heb_original} becomes
\begin{align}\tag{HEB}\label{eq:heb}
    \mu(f(x)-f(x^*))^\theta \geq \|x - x^*\|_2.
\end{align}
An important family of functions satisfying \eqref{eq:heb} is the family of uniformly convex functions, which interpolate between convex functions ($\theta = 0$) and strongly convex functions ($\theta = 1/2$).
\begin{definition}[Uniformly convex function]\label{def:unif_cvx_f}
Let $\cC \subseteq \R^d$ be a compact convex set, let $f\colon \cC \to \R$ be differentiable in an open set containing $\cC$, let $\alpha_f >0$, and let $r\geq 2$. We say that $f$ is \emph{$(\alpha_f, r)$-uniformly convex} over $\cC$ with respect to $\|\cdot\|_2$ if for all $x,y\in \cC$, it holds that
$
    f(y) \geq f(x) + \langle \nabla f(x), y-x\rangle + \frac{\alpha_f}{r}\|x-y\|_2^r.
$
We refer to $(\alpha_f, 2)$-uniformly convex functions as \emph{$\alpha_f$-strongly convex}.
\end{definition}
Note that $(\alpha_f, r)$-uniformly convex functions satisfy a $((r/{\alpha_f})^{1/r},1/r)$-\eqref{eq:heb}:
$
    f(x) - f(x^*)  \geq  \langle \nabla f (x^*), x- x^*\rangle + \frac{\alpha_f}{r}\|x - x^*\|^r_2\geq \frac{\alpha_f}{r}\|x - x^*\|^r_2.
$

\section{{Accelerated convergence rates for FW with open-loop step-sizes}}\label{sec:accelerated}
FW with open-loop step-size rules was already studied by \citet{dunn1978conditional} and currently, two open-loop step-sizes are prevalent, $\eta_t = \frac{1}{t+1}$, for which the best known convergence rate is $\cO\left(\log (t)/t \right)$, and $\eta_t = \frac{2}{t+2}$, for which a faster convergence rate of order $\cO(1/t)$ holds, see, for example, \citet{dunn1978conditional} and \citet{jaggi2013revisiting}, respectively.
In this section, we derive convergence rates for FW with open-loop step-size $\eta_t = \frac{4}{t+4}$. Convergence results for FW with $\eta_t = \frac{\ell}{t+\ell}$ for $\ell\in \N_{\geq 1}$ presented throughout this paper, except for those in Section~\ref{sec:kernel_herding}, can always be generalized (up to a constant) to $\eta_{t} = \frac{j}{t+j}$ for $j\in\N_{\geq \ell}$. 

This section is structured as follows. First, we derive a baseline convergence rate of order $\cO(1/t)$ in Section~\ref{sec:baseline}. Then, in Section~\ref{sec:blueprint}, we present the proof blueprint used throughout most parts of the paper to derive accelerated convergence rates and directly apply our approach to the setting when the objective satisfies \eqref{eq:heb} and the optimal solution $x^*\in\argmin_{x\in\cC} f(x)$ lies in the relative interior of the feasible region. In Section~\ref{sec:exterior}, we prove accelerated rates when the feasible region is uniformly convex and the norm of the gradient of the objective is bounded from below by a nonnegative constant. Finally, in Section~\ref{sec:unconstrained}, we prove accelerated rates when the feasible region is uniformly convex and the objective satisfies \eqref{eq:heb}.
\subsection{Convergence rate of order $\cO(1/t)$}\label{sec:baseline}
We begin the analysis of FW with open-loop step-size rules by first recalling the, to the best of our knowledge, best general convergence rate of the algorithm. 
Consider the setting when $\cC \subseteq \R^d$ is a compact convex set and $f\colon \cC \to \R$ is a convex and $L$-smooth function with unique minimizer $x^*\in\argmin_{x\in\cC}f(x)$. Then, the iterates of Algorithm~\ref{algo:fw} with any step-size $\eta_t \in [0, 1]$ satisfy
\begin{align}\tag{Progress-Bound}\label{eq:start_progress_bound}
    h_{t+1} & \leq h_t - \eta_t \langle\nabla f(x_t), x_t - p_t\rangle + \eta_t^2\frac{\ L \|x_t-p_t\|^2_2}{2},
\end{align}
which follows from the smoothness of $f$. 
With \eqref{eq:start_progress_bound}, it is possible to derive a baseline convergence rate for FW with open-loop step-size $\eta_t = \frac{4}{t+4}$ similar to the one derived by \citet{jaggi2013revisiting} for FW with $\eta_t = \frac{2}{t+2}$.
\begin{proposition}[Convergence rate of order $\cO(1/t)$]\label{prop:generalization_jaggi}
Let $\cC \subseteq \R^d$ be a compact convex set of diameter $\delta > 0$, let $f\colon \cC \to \R$ be a convex and $L$-smooth function with unique minimizer $x^*\in\argmin_{x\in\cC}f(x)$. Let $T\in\N$ and $\eta_t = \frac{4}{t+4}$ for all $t\in\Z$.
Then, for the iterates of Algorithm~\ref{algo:fw} with step-size $\eta_t$, it holds that
$
    h_t \leq \frac{ 8L \delta^2}{t+3} = \eta_{t-1}2L\delta^2
$
for all $t\in\{1,\ldots, T\}$.
\end{proposition}
\begin{proof}
In the literature, the proof is usually done by induction \citep{jaggi2013revisiting}. Here, for convenience and as a brief introduction for things to come, we proceed with a direct approach.
Since $\eta_0 = 1$, by $L$-smoothness, we have
    $h_1 \leq \frac{ L \delta^2}{2}.$
Let $t\in\{1, \ldots, T-1\}$. By optimality of $p_t$ and convexity of $f$, 
$\langle\nabla f(x_t), x_t - p_t\rangle \geq \langle\nabla f(x_t), x_t - x^*\rangle \geq h_t$.
Plugging this bound into \eqref{eq:start_progress_bound} and with $\|x_t - p_t\|_2\leq \delta$, it holds that
\begin{align}
    h_{t+1}&\leq (1-\eta_t)h_t + \eta_t^2\frac{ L \|x_t - p_t\|_2^2}{2} \label{eq:always_combine_with_this}\\
    &  \leq \prod_{i=1}^t(1-\eta_i)h_1 + \frac{L\delta^2}{2}\sum_{i=1}^t \eta_i^2\prod_{j = i +1}^t(1-\eta_j)\nonumber\\
    & \leq \frac{ L \delta^2}{2}(\frac{4!}{(t+1) \cdots (t+4)} + \sum_{i=1}^t \frac{4^2}{(i+4)^2}\frac{(i+1) \cdots (i+4)}{(t+1) \cdots (t+4)}) \nonumber\\
    & \leq 8L \delta^2 ( \frac{1}{(t+4-1)(t+4)} + \frac{t}{(t+4-1)(t+4 )}) \nonumber\\
    & \leq \frac{8L \delta^2}{t+4},\nonumber
\end{align}
where we used that
$\prod_{j = i+1}^t (1 - \eta_j) =  \frac{(i+1) (i+2) \cdots t}{(i+5) (i+6) \cdots (t+4)}  = \frac{(i+1)(i+2)(i+3)(i+4)}{(t+1)(t+2)(t+3)(t+4)}$.
\end{proof}
To prove accelerated convergence rates for FW with open-loop step-sizes, we require bounds on the \emph{Frank-Wolfe gap} (FW gap) $\max_{p\in\cC} \langle\nabla f(x_t), x_t - p \rangle$, which appears in the middle term in \eqref{eq:start_progress_bound}.

\subsection{{Optimal solution in the relative interior -- a blueprint for acceleration}}\label{sec:blueprint}

Traditionally, to prove accelerated convergence rates for FW with line-search or short-step, the geometry of the feasible region, curvature assumptions on the objective function, and information on the location of the optimal solution are exploited \citep{levitin1966constrained, demianov1970approximate, guelat1986some, garber2015faster}. A similar approach leads to acceleration results for FW with open-loop step-sizes, however, requiring a different proof technique as FW with open-loop step-sizes is not monotonous in primal gap. Here, we introduce the proof blueprint used to derive most of the accelerated rates in this paper via the setting when the objective $f$ satisfies \eqref{eq:heb} and the minimizer of $f$ is in the relative interior of the feasible region $\cC$.

Our goal is to bound the FW gap to counteract the error accumulated from the right-hand term in \eqref{eq:start_progress_bound}.
More formally, we prove the existence of $\phi > 0$, such that there exists an iteration $\fwt \in \N$ such that for all iterations $t\geq \fwt$ of FW, it holds that 
\begin{align}\tag{Scaling}\label{eq:scaling}
    \frac{\langle \nabla f(x_t) , x_t - p_t \rangle}{\|x_t - p_t\|_2} \geq \phi  \frac{\langle \nabla f(x_t) , x_t - x^* \rangle}{\|x_t - x^*\|_2}.
\end{align}
Inequalities that bound \eqref{eq:scaling} from either side are referred to as \emph{scaling inequalities}.
Intuitively speaking, scaling inequalities relate the \emph{FW direction} $\frac{p_t - x_t}{\|p_t-x_t\|_2}$ with the \emph{optimal descent direction} $\frac{x^*-x_t}{\|x^*-x_t\|_2}$. Scaling inequalities stem from the geometry of the feasible region, properties of the objective function, or information on the location of the optimal solution.
The scaling inequality below exploits the latter property.
\begin{lemma}[\citealp{guelat1986some}]\label{lemma:GM}
    Let $\cC \subseteq \R^d$ be a compact convex set of diameter $\delta > 0$, let $f\colon \cC \to \R$ be a convex and $L$-smooth function with unique minimizer $x^*\in\argmin_{x\in\cC}f(x)$, and suppose that there exists $\beta>0$ such that
$\aff (\cC) \cap B_\beta(x^*)\subseteq \cC$.
Then, for all $x\in \cC\cap B_\beta(x^*)$, it holds that
\begin{equation}\tag{Scaling-INT}\label{eq:scaling_int}
    \frac{\langle \nabla f(x), x - p\rangle}{\|x - p\|_2} \geq \frac{\beta}{\delta} \|\nabla f(x)\|_2,
\end{equation}
where $p \in \argmin_{v\in \cC} \langle \nabla f(x), v \rangle$.
    \end{lemma}
    
Below, we prove that there exists $\fwt \in \N$ such that for all $t\geq \fwt$, $x_t \in B_\beta(x^*)$ and \eqref{eq:scaling_int} is satisfied.
\begin{lemma}\label{lemma:dist_to_opt}
Let $\cC \subseteq \R^d$ be a compact convex set of diameter $\delta > 0$, let $f\colon \cC \to \R$ be a convex and $L$-smooth function satisfying a $(\mu, \theta)$-\eqref{eq:heb} for some $\mu > 0 $ and $\theta \in ]0, 1/2]$ with unique minimizer $x^*\in\argmin_{x\in\cC}f(x)$, and let $\beta > 0$.
Let $\fwt = \lceil 8L \delta^2\left(\mu / \beta\right)^{1/\theta} \rceil$,
$T\in\N$, and $\eta_t = \frac{4}{t+4}$ for all $t\in\Z$.
Then, for the iterates of Algorithm~\ref{algo:fw} with step-size $\eta_t$, it holds that $\|x_t -x^*\|_2 \leq \beta$ for all $t \in\{\fwt,\ldots, T\}$.
\end{lemma}
\begin{proof}
By \eqref{eq:heb} and Proposition~\ref{prop:generalization_jaggi},
$\|x_t - x^*\|_2 \leq \mu h_t^\theta \leq \mu (\frac{8 L \delta^2 }{ 8L \delta^2(\mu/\beta)^{1/\theta} })^\theta \leq \beta$ for all $t\in\{\fwt,\ldots, T\}$.
\end{proof} 

The second scaling inequality follows from the objective satisfying \eqref{eq:heb}.
\begin{lemma}\label{lemma:heb_to_grad}
Let $\cC\subseteq \R^d$ be a compact convex set and let $f\colon \cC \to \R$ be a convex function satisfying a $(\mu, \theta)$-\eqref{eq:heb} for some $\mu > 0 $ and $\theta \in [0, 1/2]$ with unique minimizer $x^*\in\argmin_{x\in\cC}f(x)$. Then, for all $x\in \cC$, it holds that
\begin{align}\label{eq:scaling_heb}
    \|\nabla f (x)\|_2 &  \geq \frac{\langle \nabla f(x), x - x^*\rangle}{\|x - x^*\|_2}\geq \frac{1}{\mu}(f(x) - f(x^*))^{{1-\theta}}.\tag{Scaling-HEB}
\end{align}
\end{lemma}
\begin{proof}
The statement holds for $x=x^*$. For $x\in\cC\setminus \{x^*\}$, by convexity and \eqref{eq:heb},
$f(x) - f(x^*)  \leq
    \frac{\langle \nabla f (x), x-x^*\rangle}{\|x-x^*\|_2} \|x-x^*\|_2  \leq \frac{\langle \nabla f (x), x-x^*\rangle}{\|x-x^*\|_2} \mu (f(x) - f(x^*))^{\theta}$.
Dividing by $\mu(f(x)-f(x^*))^\theta$ yields \eqref{eq:scaling_heb}.
\end{proof}
For $t\in\{\fwt,\ldots, T-1\}$, where 
$\fwt = \lceil 8L \delta^2\left(2\mu / \beta\right)^{1/\theta} \rceil$,
we plug \eqref{eq:scaling_int} and \eqref{eq:scaling_heb} into \eqref{eq:start_progress_bound} to obtain 
$h_{t+1}  \leq h_t - \eta_t \frac{\beta^2}{2\mu \delta} h_t^{1-\theta} + \eta_t^2\frac{ L \delta^2}{2}$.
Combined with \eqref{eq:always_combine_with_this}, we have
\begin{align}\label{eq:int_sequence}
    h_{t+1} & \leq (1 -\frac{\eta_t}{2}) h_t - \eta_t \frac{\beta^2}{4\mu \delta} h_t^{1-\theta} + \eta_t^2\frac{ L \delta^2}{2}
\end{align}
for all $t\in\{\fwt,\ldots, T-1\}$. If the primal gaps of FW with open-loop step-sizes satisfy an inequality of this type, the lemma below implies accelerated convergence rates.
\begin{lemma}\label{lemma:sequences}
Let $\psi \in [0, 1/2]$, $\fwt, T \in \N_{\geq 1}$, and $\eta_t = \frac{4}{t+4}$ for all $t\in\Z$. Suppose that there exist constants $A, B, C > 0$, a nonnegative sequence $\{C_t\}_{t=\fwt}^{T-1}$ such that $C \geq C_t \geq 0$ for all $t\in\{\fwt,\ldots,T-1\}$, and a nonnegative sequence $\{h_t\}_{t = \fwt}^{T}$ such that
\begin{align}\label{eq:gotta_derive_this}
    h_{t+1} &  \leq (1 - \frac{\eta_t}{2})h_t - \eta_t AC_t h_t^{1-\psi} + \eta_t^2 BC_t
\end{align}
for all $t \in\{\fwt,  \ldots, T-1\}$. Then, 
\begin{align}\label{eq:cd_simple}
    h_t  \leq \max \left\{  \left(\frac{\eta_{t-2}}{\eta_{\fwt-1}}\right)^{1/(1-\psi)}h_\fwt, \left(\frac{\eta_{t-2} B}{A}\right)^{1/(1-\psi)} + \eta_{t-2}^2 BC\right\}
    \end{align}
for all $t \in\{\fwt, \ldots, T\}$.
\end{lemma}
\begin{proof}
For all $t\in\{\fwt, \ldots, T\}$, we first prove that 
\begin{align}\label{eq:cd}
    h_t & \leq \max \left\{ \left(\frac{\eta_{t-2}\eta_{t-1}}{\eta_{\fwt-2}\eta_{\fwt-1}}\right)^{1/(2(1-\psi))}h_\fwt, \left(\frac{\eta_{t-2}\eta_{t-1} B^2}{A^2}\right)^{1/(2(1-\psi))} + \eta_{t-2 }\eta_{t-1} BC\right\} ,
\end{align}
which then implies \eqref{eq:cd_simple}.
The proof is a straightforward modification of Footnote $3$ in the proof of Proposition $2.2$ in \citet{bach2021effectiveness} and is by induction. The base case of \eqref{eq:cd} with $t = \fwt$ is immediate, even if $\fwt= 1$, as $\eta_{-1}\geq \eta_{0} = 1$.
Suppose that \eqref{eq:cd} is correct for a specific iteration $t\in\{\fwt,  \ldots, T-1\}$. We distinguish between two cases.
First, suppose that 
$h_t \leq (\frac{\eta_t B}{A})^{1/(1-\psi)}$.
Plugging this bound into \eqref{eq:gotta_derive_this}, we obtain
$h_{t+1} \leq (1-\frac{\eta_t}{2}) h_t - 0 + \eta_t^2 BC_t\leq  (\frac{\eta_t B}{A})^{1/(1-\psi)} + \eta_t^2 BC \leq (\frac{\eta_{t-1}\eta_t B^2}{A^2})^{1/(2(1-\psi))} + \eta_{t-1}\eta_t BC$.
Next, suppose that
$h_t \geq  (\frac{\eta_t B}{A})^{1/(1-\psi)}$ instead.
Plugging this bound on $h_t$ into \eqref{eq:gotta_derive_this} and using the induction assumption \eqref{eq:cd} at iteration $t$ yields
\begin{align*}
    h_{t+1} &\leq  \left(1 - \frac{\eta_t}{2}\right)h_t  -\eta_t A C_t \frac{\eta_t B}{A} + \eta_t^2 B C_t\\
    &= \frac{t+2}{t+4} h_t \\
    & = \frac{\eta_{t}}{\eta_{t-2}} h_t\\
    &\leq  \frac{\eta_{t}}{\eta_{t-2}} \max \left\{ \left(\frac{\eta_{t-2}\eta_{t-1}}{\eta_{\fwt-2}\eta_{\fwt-1}}\right)^{1/(2(1-\psi))}h_\fwt, \left(\frac{\eta_{t-2}\eta_{t-1} B^2}{A^2}\right)^{1/(2(1-\psi))} + \eta_{t-2 }\eta_{t-1} BC\right\}\\
    & \leq \max \left\{ \left(\frac{\eta_{t-1}\eta_{t}}{\eta_{\fwt-2}\eta_{\fwt-1}}\right)^{1/(2(1-\psi))}h_\fwt, \left(\frac{\eta_{t-1}\eta_{t} B^2}{A^2}\right)^{1/(2(1-\psi))} + \eta_{t-1 }\eta_{t} BC\right\},
\end{align*}
where the last inequality holds due to $\frac{\eta_t}{\eta_{t-2}}(\eta_{t-2}\eta_{t-1})^{1/(2(1-\psi))} \leq (\eta_{t-1}\eta_{t})^{1/(2(1-\psi))}$ for $\frac{\eta_t}{\eta_{t-2}}\in [0,1]$ and $1/(2(1-\psi)) \in [1/2,1]$.
In either case, \eqref{eq:cd} is satisfied for $t+1$. By induction, the lemma follows.
\end{proof}
We conclude the presentation of our proof blueprint by stating the first accelerated convergence rate for FW with open-loop step-size $\eta_t = \frac{4}{t+4}$ when the the objective function $f$ satisfies \eqref{eq:heb} and the minimizer lies in the relative interior of the feasible region $\cC$. For this setting,
FW with line-search or short-step converges linearly if the objective function is strongly convex \citep{guelat1986some,garber2015faster}.
Further, FW with open-loop step-size $\eta_t = \frac{1}{t+1}$ converges at a rate of order $\cO(1/t^2)$ when the objective is of the form $f(x) = \frac{1}{2}\|x-b\|_2^2$ for some $b\in \cC$ \citep{chen2012super}.

\begin{theorem}[Optimal solution in the relative interior of $\cC$]\label{thm:interior}
Let $\cC \subseteq \R^d$ be a compact convex set of diameter $\delta > 0$, let $f\colon \cC \to \R$ be a convex and $L$-smooth function satisfying a $(\mu, \theta)$-\eqref{eq:heb} for some $\mu > 0 $ and $\theta \in ]0, 1/2]$ with unique minimizer $x^*\in\argmin_{x\in\cC}f(x)$, and suppose that there exists $\beta>0$ such that
$\aff (\cC)  \cap B_\beta(x^*) \subseteq \cC$.
Let
$\fwt = \lceil 8L \delta^2\left(2\mu / \beta\right)^{1/\theta} \rceil$,
$T\in\N$, and $\eta_t = \frac{4}{t+4}$ for all $t\in\Z$.
Then, for the iterates of Algorithm~\ref{algo:fw} with step-size $\eta_t$, it holds that
 \begin{align}\label{eq:interior_sol}
     h_t \leq \max \Bigg\{ & \left(\frac{\eta_{t-2}}{\eta_{\fwt-1}}\right)^{1/(1-\theta)} h_\fwt, \left(\frac{\eta_{t-2}2 \mu L\delta^3}{\beta^2}\right)^{1/(1-\theta)} + \eta_{t-2}^2 \frac{L\delta^2}{2}\Bigg\}
 \end{align}
 for all $t\in\{\fwt, \ldots, T\}$.
\end{theorem}
\begin{proof}
Let $t\in\{\fwt, \ldots, T-1\}$. By Lemma~\ref{lemma:dist_to_opt}, $\|x_t -x^*\|_2 \leq \beta / 2$  and, by triangle inequality, we have $\|x_t - p_t\|_2 \geq \beta / 2$. Thus,
for all $t\in\{\fwt, \ldots, T\}$, it follows that \eqref{eq:int_sequence} holds.
We apply Lemma~\ref{lemma:sequences} with $A = \frac{\beta^2}{4 \mu \delta}$, $B = \frac{L\delta^2}{2}$, $C= 1$, $C_t = 1$ for all $t\in\{\fwt, \ldots, T-1\}$, and $\psi = \theta$, resulting in \eqref{eq:interior_sol} holding
for all $t\in\{\fwt, \ldots, T\}$. 
\end{proof}
We complement Theorem~\ref{thm:interior} with a discussion on the lower bound of the convergence rate of FW when the optimal solution is in the relative interior of the probability simplex.
\begin{lemma}[\citealp{jaggi2013revisiting}]\label{lemma:lb_jaggi}
Let $\cC\subseteq \R^d$ be the probability simplex, $f(x) = \|x\|_2^2$, and $t\in\{1,\ldots, d\}$. It holds that
$\min_{\substack{x \in \cC \\ |\supp (x)| \leq t}} f(x) = \frac{1}{t}$,
where $|\supp(x)|$ denotes the number of non-zero entries of $x$.
\end{lemma}
\begin{remark}[{Compatibility with lower bound from \citet{jaggi2013revisiting}}]\label{rem:jaggi_interior}
In Lemma~\ref{lemma:lb_jaggi}, the optimal solution $x^* = \frac{1}{d}\oneterm\in \R^d$ lies in the relative interior of $\cC$ and $\min_{x\in \cC}f(x) = 1/d$. When $\cC$ is the probability simplex, all of its vertices are of the form $e^{(i)} = (0, \ldots, 0 , 1, 0 , \ldots, 0)^\intercal \in \R^d$, $i\in\{1, \ldots, d\}$.
Thus, any iteration of FW can modify at most one entry of iterate $x_t$ and the primal gap is at best $h_t = 1/t -1/d$ for $t\in\{1, \ldots, d\}$.
Applying Theorem~\ref{thm:interior}
to the setting of Lemma~\ref{lemma:lb_jaggi}, we observe that $\beta =1/d$ and acceleration starts only after $\fwt = \Omega(d^{1/\theta}) \geq  \Omega(d)$ iterations. Thus, Theorem~\ref{thm:interior} does not contradict Lemma~\ref{lemma:lb_jaggi}.
\end{remark}
\subsection{Unconstrained minimizer in the exterior -- lower-bounded gradient norm}\label{sec:exterior}
In this section, we apply the proof blueprint from the previous section to the setting when the feasible region $\cC$ is uniformly convex and the norm of the gradient of $f$ is bounded from below by a nonnegative constant.

For this setting, FW with line-search or short-step converges linearly when the feasible region is also strongly convex \citep{levitin1966constrained, demianov1970approximate, garber2015faster}. When the feasible region is only uniformly convex, rates interpolating between $\cO(1/t)$ and linear convergence are known \citep{kerdreux2021projection}.
Two FW variants employ open-loop step-sizes and enjoy accelerated convergence rates of order up to $\cO(1/t^2)$ when the feasible region $\cC$ is uniformly convex and the norm of the gradient of $f$ is bounded from below by a nonnegative constant: the primal-averaging Frank-Wolfe algorithm (PAFW) \citep{lan2013complexity, kerdreux2021local}, presented in Algorithm~\ref{algo:pafw}, and the momentum-guided FW algorithm (MFW) \citep{li2021momentum}, presented in Algorithm~\ref{algo:mfw}.
Below, for the same setting, we prove that FW with open-loop step-size $\eta_t = \frac{4}{t+4}$ also admits accelerated convergence rates of order up to $\cO(1/t^2)$ depending on the uniform convexity of the feasible region. Furthermore, when the feasible region is strongly convex, we prove that FW with open-loop step-size $\eta_t=\frac{\ell}{t+\ell}$, where $\ell\in\N_{\geq 2}$, converges at a rate of order $\cO(1/t^{\ell/2})$, which is faster than the convergence rates known for PAFW and MFW.
To prove these results, we require two new scaling inequalities, the first of which follows directly from the assumption that the norm of the gradient of $f$ is bounded from below by a nonnegative constant. More formally, let $\cC\subseteq \R^d$ be a compact convex set and let $f\colon \cC \to \R$ be a convex and $L$-smooth function such that there exists $\lambda > 0$ such that for all $x\in \cC$, 
\begin{equation}\tag{Scaling-EXT}\label{eq:scaling_ext}
    \|\nabla f(x)\|_2 \geq \lambda.
\end{equation}
In case $f$ is well-defined, convex, and differentiable on $\R^d$, \eqref{eq:scaling_ext} is, for example, implied by the convexity of $f$ and the assumption that the unconstrained minimizer of $f$, that is, $\argmin_{x\in\R^d} f(x)$, lies in the exterior of $\cC$.
The second scaling inequality follows from the uniform convexity of the feasible region and is proved in the proof of \citet[Theorem 2.2]{kerdreux2021projection} in FW gap. The result stated below is then obtained by bounding the FW gap from below with the primal gap.

\begin{figure*}[t!]
\begin{minipage}[t]{0.46\textwidth}
\begin{algorithm}[H]
\SetKwInput{Input}{Input}
\SetKwInput{Output}{Output}
\caption{Primal-averaging Frank-Wolfe algorithm (PAFW) \citep{lan2013complexity}}\label{algo:pafw}
  \Input{$x_0\in \cC$, step-sizes $\eta_t = \frac{\ell}{t+\ell}$, where $\ell\in\N_{\geq 1}$, for $t\in\{0,\ldots, T-1\}$.}
  \hrulealg
  $v_0 \gets x_0$\\
  \For{$t= 0, \ldots, T-1 $}{
    $y_t \gets (1-\eta_t) x_t + \eta_t v_t$\\
    $w_{t+1} \gets \nabla f(y_t)$\label{line:pafw_w}\\
    $v_{t+1}\in\argmin_{v\in\cC} \langle w_{t+1}, v\rangle$\\
    $x_{t+1}\gets(1-\eta_t) x_t + \eta_t v_{t+1}$}
\end{algorithm}
\end{minipage}
\hfil
\begin{minipage}[t]{0.46\textwidth}
\begin{algorithm}[H]
\SetKwInput{Input}{Input}
\SetKwInput{Output}{Output}
\caption{Momentum-guided Frank-Wolfe algorithm (MFW) \citep{li2021momentum}}\label{algo:mfw}
  \Input{$x_0\in \cC$, step-sizes $\eta_t = \frac{\ell}{t+\ell}$, where $\ell\in\N_{\geq 1}$, for $t\in\{0,\ldots, T-1\}$.}
  \hrulealg
  $v_0 \gets x_0$; $w_0 \gets \zeroterm$\\
  \For{$t= 0, \ldots, T-1 $}{
    $y_t \gets (1-\eta_t) x_t + \eta_t v_t$\\
    $w_{t+1} \gets (1-\eta_t)w_t + \eta_t \nabla f(y_t)$\label{line:mfw_w}\\
    $v_{t+1}\in\argmin_{v\in\cC} \langle w_{t+1}, v\rangle$\\
    $x_{t+1}\gets(1-\eta_t) x_t + \eta_t v_{t+1}$}
\end{algorithm}
\end{minipage}
\end{figure*}

\begin{lemma}[\citealp{kerdreux2021projection}]\label{lemma:unif_convexity}
For $\alpha >0$ and $q\geq 2$, let $\cC \subseteq \R^d$ be a compact $(\alpha,q)$-uniformly convex set and let $f\colon \cC \to \R$ be a convex function that is differentiable in an open set containing $\cC$ with unique minimizer $x^*\in\argmin_{x\in\cC}f(x)$. Then, for all $x\in \cC$, it holds that
    \begin{equation}\tag{Scaling-UNIF}\label{eq:scaling_unif}
           \frac{\langle \nabla f(x), x-p\rangle}{\|x-p\|_2^2} \geq  \left(\frac{\alpha}{2}\|\nabla f(x)\|_2\right)^{2/q} (f(x)-f(x^*))^{1-2/q},
    \end{equation}
where $p \in \argmin_{v\in \cC} \langle \nabla f(x), v \rangle$.
\end{lemma}
Combining \eqref{eq:scaling_ext} and \eqref{eq:scaling_unif}, we derive the following accelerated convergence result.
\begin{theorem}[Norm of the gradient of $f$ is bounded from below by a nonnegative constant]\label{thm:exterior}
For $\alpha >0$ and $q \geq 2$, let $\cC \subseteq \R^d$ be a compact $(\alpha,q)$-uniformly convex set of diameter $\delta > 0$, let $f\colon \cC \to \R$ be a convex and $L$-smooth function with lower-bounded gradients, that is, $\|\nabla f(x)\|_2 \geq \lambda$ for all $x\in \cC$ for some $\lambda > 0$, with unique minimizer $x^*\in\argmin_{x\in\cC}f(x)$. 
Let $T\in\N$ and $\eta_t = \frac{4}{t+4}$ for all $t\in\Z$.
Then, for the iterates of Algorithm~\ref{algo:fw} with step-size $\eta_t$, when $q \geq 4$, it holds that
\begin{align}\label{eq:ext_q_greater_4}
    h_t & \leq \max \left\{ \eta_{t-2}^{1/(1-2/q)}\frac{L\delta^2}{2}, \left(\eta_{t-2} L \left( \frac{2}{\alpha \lambda}\right)^{2/q}\right)^{1/(1-2/q)} + \eta_{t-2}^2 \frac{L\delta^2}{2}\right\}
\end{align}
for all $t\in\{1, \ldots, T\}$, 
and letting ${\fwt} = \lceil 8 L \delta^2 \rceil$, when $q \in [2, 4[$, it holds that
\begin{align}\label{eq:ext_sol}
    h_t & \leq 
    \max \left\{ \left(\frac{\eta_{t-2}}{\eta_{{\fwt}-1}}\right)^{2}h_{\fwt}, \left(\eta_{t-2} L \left( \frac{2}{\alpha \lambda}\right)^{2/q}\right)^{2} + \eta_{t-2}^2 \frac{L\delta^2}{2}\right\}
\end{align}
for all $t\in\{\fwt,\ldots, T\}$.
\end{theorem}
\begin{proof}
Let $t\in\{1,\ldots, T-1\}$. Combining \eqref{eq:scaling_unif} and \eqref{eq:scaling_ext}, it holds that
$\langle \nabla f(x_t), x_t - p_t\rangle \geq \| x_t - p_t \|_2^2 \left(\frac{\alpha\lambda}{2}\right)^{2/q} h_t^{1-2/q}$.
Then, using \eqref{eq:start_progress_bound}, we obtain
$h_{t+1}  \leq   h_t - \eta_t\|x_t-p_t\|_2^2( \frac{\alpha \lambda}{2})^{2/q} h_t^{1-2/q} +\eta_t^2\frac{L\|x_t-p_t\|_2^2}{2}$.
Combined with \eqref{eq:always_combine_with_this}, we obtain
\begin{align}\label{eq:ext}
   h_{t+1}  \leq &  \left(1-\frac{\eta_t}{2}\right)h_t +\frac{\eta_t\|x_t-p_t\|_2^2}{2}  \left( \eta_t L-\left( \frac{\alpha \lambda}{2}\right)^{2/q} h_t^{1-2/q} \right).
\end{align}
Suppose that $q \geq  4$. Then, \eqref{eq:ext} allows us to apply Lemma~\ref{lemma:sequences} with $A =( \frac{\alpha \lambda}{2})^{2/q}$, $B=L$, $C= \frac{\delta^2}{2}$, $C_t = \frac{\|x_t-p_t\|_2^2}{2}$ for all $t\in\{1,\ldots, T-1\}$, and $\psi = 2/q\in [0,1/2]$, resulting in \eqref{eq:ext_q_greater_4} holding for all $t\in\{1,\ldots, T\}$, since
$h_1 \leq \frac{L\delta^2}{2}$, and $\eta_{-1}\geq \eta_{0} = 1$.
Next, suppose that $q\in[2, 4[$ and note that $2/q > 1/2$. Thus, Lemma~\ref{lemma:sequences} can be applied after a burn-in phase of slower convergence. Let $t\in\{\fwt,\ldots,T-1\}$.
By Proposition~\ref{prop:generalization_jaggi}, $ h_t \leq h_\fwt \leq 1 $.
Since $1-2/q\leq 1/2$, we have $h_t^{1-2/q} \geq h_t^{1/2} = h_t^{1-1/2}$. Combined with \eqref{eq:ext}, it holds that
$h_{t+1}  \leq   (1-\frac{\eta_t}{2})h_t +\frac{\eta_t\|x_t-p_t\|_2^2}{2}  ( \eta_t L-( \frac{\alpha \lambda}{2})^{2/q} h_t^{1 - 1/2} )$.
We then apply 
Lemma~\ref{lemma:sequences} with $A =( \frac{\alpha \lambda}{2})^{2/q}$, $B=L$, $C= \frac{\delta^2}{2}$, $C_t = \frac{\|x_t-p_t\|_2^2}{2}$ for all $t\in\{\fwt,\ldots,T-1\}$, and $\psi = 1/2$, resulting in \eqref{eq:ext_sol} holding
for all $t\in\{\fwt,\ldots, T\}$. Note that the lemma holds even if $\fwt = 1$ since $\eta_{-1}\geq \eta_0 = 1$.
\end{proof}

As we discuss below, in the setting of Theorem~\ref{thm:exterior}, when $q=2$, FW with open-loop step-sizes $\eta_t= \frac{\ell}{t+\ell}$, where $\ell\in \N_{\geq 2}$, converges at a rate of order $\cO(1/t^{\ell/2})$.

\begin{remark}[Acceleration beyond rates of order $\cO(1/t^2)$]\label{rem:ol_linear}
Under the assumptions of Theorem~\ref{thm:exterior}, analogously to Proposition~\ref{prop:generalization_jaggi}, one can prove convergence rates of order $\cO(1/t)$ for FW with step-sizes $\eta_t = \frac{\ell}{t+\ell}$, where $\ell\in \N_{\geq 2}$, depending on $L, \delta$, and $\ell$. Thus, for $q=2$, there exists $\fwt \in \N$ depending only on $L, \alpha, \delta, \lambda,\ell$, such that for all $t\in\{\fwt,\ldots, T-1\}$, it holds that 
\begin{align*}
        \frac{\eta_t\|x_t-p_t\|_2^2}{2}   ( \eta_t L-\frac{\alpha \lambda}{2} ) \leq 0. 
\end{align*}
Thus, \eqref{eq:ext} becomes $h_{t+1}\leq (1-\frac{\eta_t}{2})h_t$ for all $t\in\{\fwt, \ldots, T-1\}$.
Then, by induction, for even $\ell\in\N_{\geq 2}$, it holds that
$h_t \leq \frac{h_{\fwt} ({\fwt}+\ell/2) ({\fwt}+\ell/2 +1) \cdots ({\fwt}+\ell-1)}{(t+\ell/2) (t+\ell/2 +1) \cdots (t+\ell-1)}$
for all $t\in\{\fwt,\ldots, T-1\}$, resulting in a convergence rate of order $\cO(1/t^{\ell/2})$. For $\ell\in\N_{\geq 6}$, this convergence rate is better than the convergence rates of order $\cO(1/t^2)$ known for PAFW and MFW.
Using similar arguments, one can prove that FW with the constant open-loop step-size 
$\eta_t = \frac{\alpha\lambda}{2L}$
converges linearly, that is, $h_t \leq (1 - \frac{\alpha\lambda}{4L})^th_0$ for all $t\in\{0, \ldots, T\}$.
\end{remark}

\begin{figure}[t]
\captionsetup[subfigure]{justification=centering}
\begin{tabular}{c c c}
    \begin{subfigure}{.3\textwidth}
    \centering
        \includegraphics[width=1\textwidth]{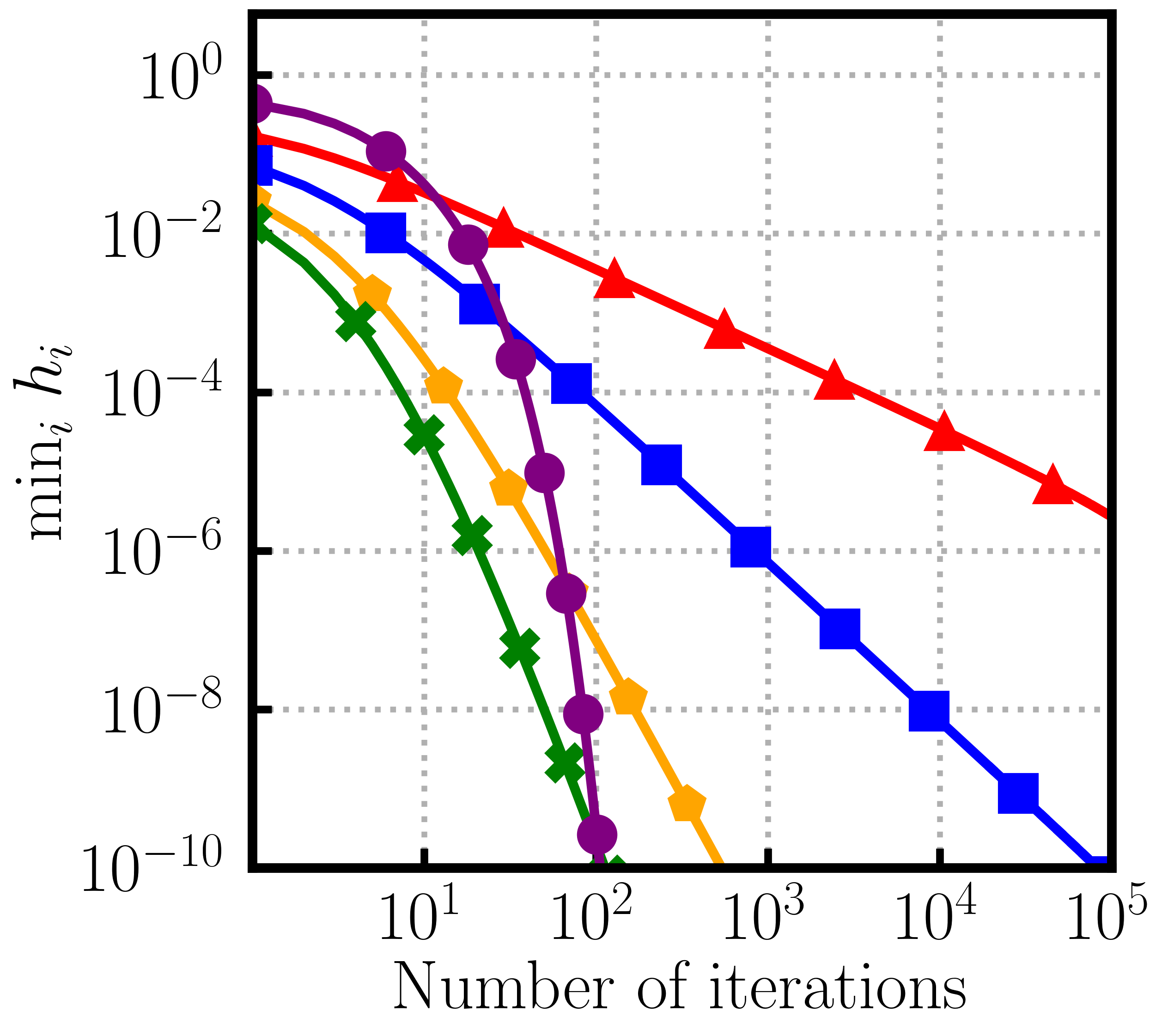}
        \caption{$\ell_2$-ball.}\label{fig:exterior_2}
    \end{subfigure}& 
    \begin{subfigure}{.3\textwidth}
    \centering
        \includegraphics[width=1\textwidth]{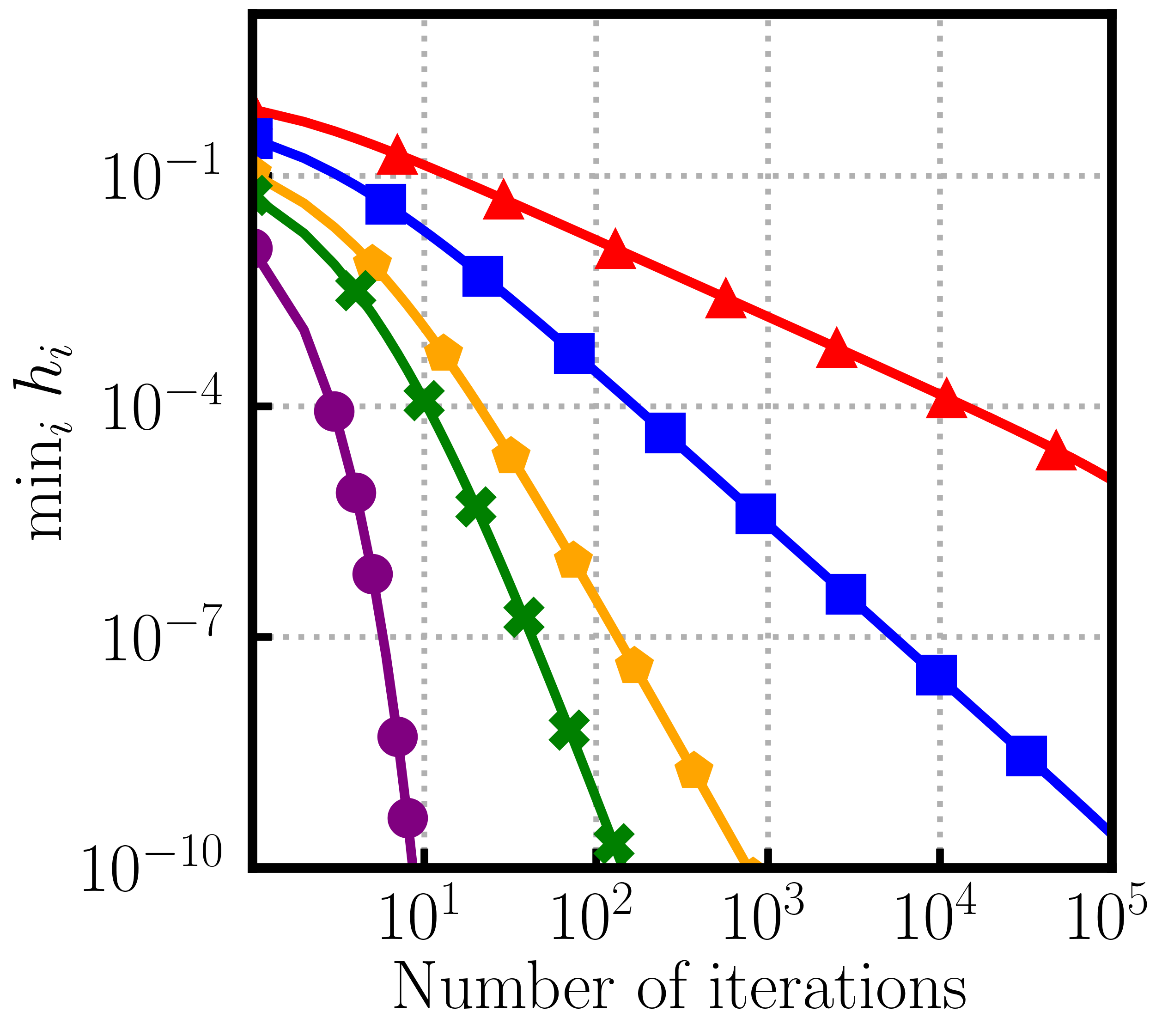}
        \caption{$\ell_3$-ball.}\label{fig:exterior_3}
    \end{subfigure} & 
    \begin{subfigure}{.3\textwidth}
    \centering
        \includegraphics[width=1\textwidth]{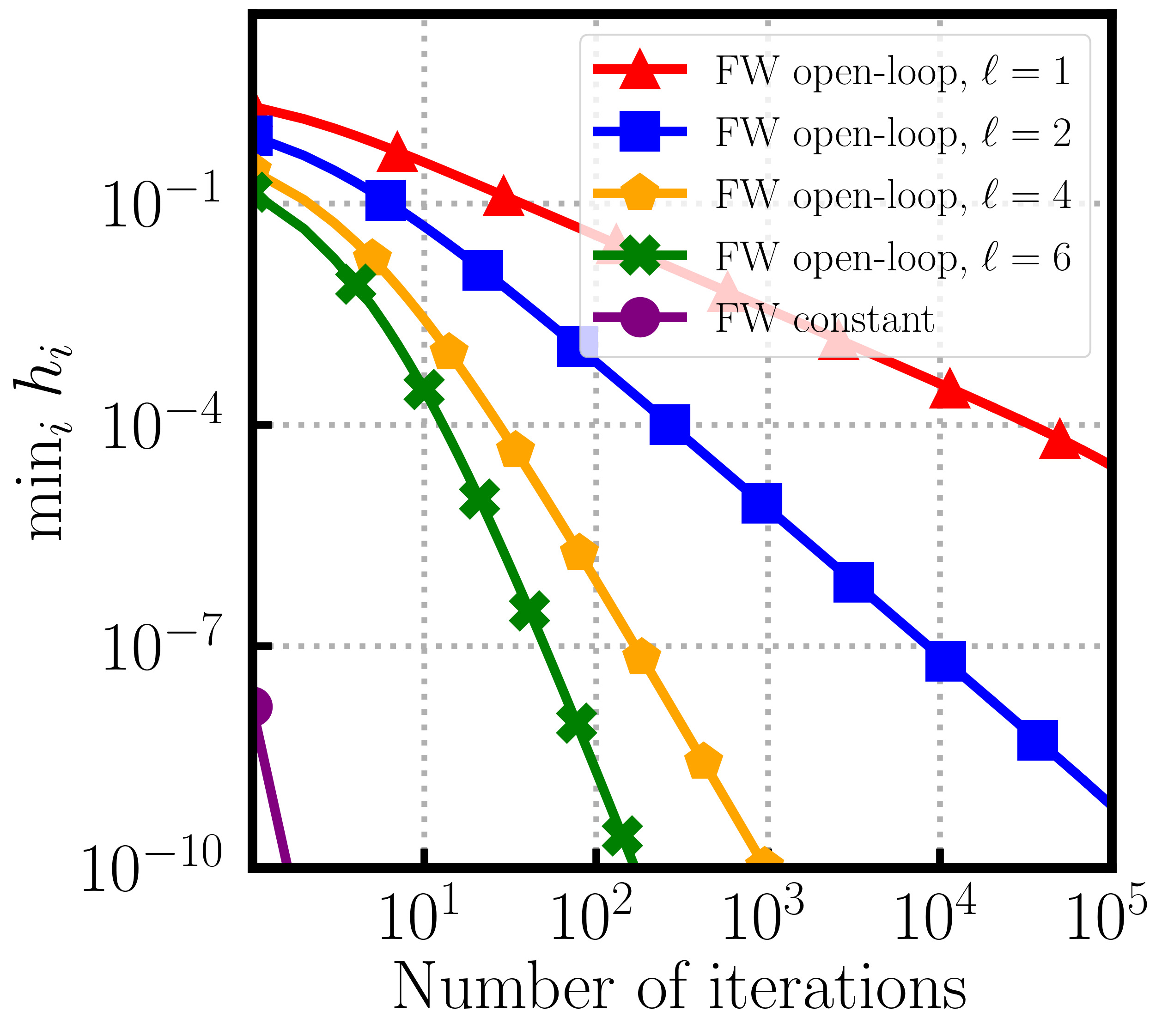}
        \caption{$\ell_5$-ball.}\label{fig:exterior_5}
    \end{subfigure}\\
\end{tabular}
\caption{
Comparison of FW with different step-sizes when the feasible region $\cC\subseteq\R^{100}$ is an $\ell_p$-ball, the objective $f$ is not strongly convex, and the unconstrained optimal solution $\argmin_{x\in\R^d}f(x)$ lies in the exterior of $\cC$, implying that $\|\nabla f(x)\|_2 \geq \lambda > 0$ for all $x\in\cC$ for some $\lambda > 0$. The $y$-axis represents the minimum primal gap.
FW with open-loop step-sizes $\eta_t = \frac{\ell}{t+\ell}$, where $\ell\in\N_{\geq 1}$, converges at a rate of order $\cO(1/t^\ell)$ and FW with constant step-size converges linearly.
}\label{fig:exterior}
\end{figure}

The results in Figure~\ref{fig:exterior}, see Section~\ref{sec:experiment_exterior} for details, show that in the setting of Theorem~\ref{thm:exterior} and Remark~\ref{rem:ol_linear}, FW with open-loop step-sizes $\eta_t=\frac{\ell}{t+\ell}$, where $\ell\in\N_{\geq 1}$, converges at a rate of order $\cO(1/t^\ell)$ and FW with constant step-size $\eta_t = \frac{\alpha\lambda}{2L}$ converges linearly in Figure~\ref{fig:exterior_2}. The convergence rates for FW with $\eta_t=\frac{\ell}{t+\ell}$ are better than predicted by Remark~\ref{rem:ol_linear} and indicate a gap between theory and practice. Note that we observe acceleration beyond $\cO(1/t^2)$ even when the feasible region is only uniformly convex, a behaviour which our current theory does not explain.

\subsection{{No assumptions on the location of the optimal solution}}\label{sec:unconstrained}
In this section, we address the setting when the feasible region $\cC$ is uniformly convex, the objective function $f$ satisfies \eqref{eq:heb}, and no assumptions are made on the location of the optimal solution $x^*\in\argmin_{x\in\cC}f(x)$.

\citet{garber2015faster} showed that strong convexity of the feasible region and the objective function are enough to modify \eqref{eq:start_progress_bound} to prove a convergence rate of order $\cO(1/t^2)$ for FW with line-search or short-step. 
\citet{kerdreux2021projection} relaxed these assumptions and proved convergence rates for FW with line-search or short-step interpolating between $\cO(1/t)$ and $\cO(1/t^2)$.
Below, for the same setting, we prove that FW with open-loop step-sizes also admits rates interpolating between $\cO(1/t)$ and $\cO(1/t^2)$.

\begin{theorem}[No assumptions on the location of the optimal solution]\label{thm:unrestricted}
For $\alpha>0$ and $q\geq 2$, let $\cC \subseteq \R^d$ be a compact $(\alpha,q)$-uniformly convex set of diameter $\delta > 0$, let $f\colon \cC \to \R$ be a convex and $L$-smooth function satisfying a $(\mu, \theta)$-\eqref{eq:heb} for some $\mu > 0 $ and $\theta \in [0, 1/2]$ with unique minimizer $x^*\in\argmin_{x\in\cC}f(x)$. 
Let $T\in\N$ and $\eta_t = \frac{4}{t+4}$ for all $t\in\Z$.
Then, for the iterates of Algorithm~\ref{algo:fw} with step-size $\eta_t$, it holds that
\begin{align}\label{eq:unrestricted_rate}
    h_t & \leq \max \left\{ \eta_{t-2}^{1/(1-2\theta/q)}\frac{L\delta^2}{2}, \left(\eta_{t-2} L \left( \frac{2\mu}{\alpha}\right)^{2/q}\right)^{1/(1-2\theta/q)} + \eta_{t-2}^2 \frac{L\delta^2}{2}\right\}
\end{align}
for all $t \in\{1, \ldots, T\}$.
\end{theorem}
\begin{proof}
Let $t\in\{1,\ldots, T-1\}$. Combining \eqref{eq:scaling_unif} and \eqref{eq:scaling_heb}, we obtain
$\langle \nabla f(x_t) ,x_t-p_t \rangle \geq \|x_t-p_t\|_2^2 (\frac{\alpha}{2\mu})^{2/q} h_t^{1-2\theta/q}$.
Then, using \eqref{eq:start_progress_bound}, we obtain
$h_{t+1} \leq  h_t - \eta_t \|x_t-p_t\|_2^2 (\frac{\alpha}{2\mu})^{2/q} h_t^{1-2\theta/q} + \eta_t^2\frac{L\|x_t - p_t\|_2^2}{2}$.
Combined with \eqref{eq:always_combine_with_this}, we have
$h_{t+1} \leq  (1-\frac{\eta_t}{2})h_t + \frac{\eta_t\|x_t-p_t\|_2^2}{2} (\eta_t L - (\frac{\alpha}{2\mu})^{2/q} h_t^{1-2\theta/q})$.
We apply Lemma~\ref{lemma:sequences} with $A =( \frac{\alpha}{2\mu})^{2/q}$, $B=L$, $C= \frac{\delta^2}{2}$, $C_t = \frac{\|x_t-p_t\|_2^2}{2}$ for all $t\in\{\fwt, \ldots, T-1\}$, and $\psi = 2\theta/q \leq 1/2$, resulting in \eqref{eq:unrestricted_rate} holding
for all $t\in\{\fwt,\ldots, T\}$, since $h_1 \leq \frac{L\delta^2}{2}$, and $\eta_{-1}\geq \eta_0 = 1$.
\end{proof}
\section{Optimal solution in the relative interior of a face of $\cC$}\label{sec:ol_faster_than_ls_ss}
In this section, we consider the setting when the feasible region is a polytope, the objective function is strongly convex, and the optimal solution lies in the relative interior of an at least one-dimensional face $\cC^*$ of $\cC$. Then, under mild assumptions, FW with line-search or short-step converges at a rate of order $\Omega(1/t^{1+\eps})$ for any $\eps  > 0$ \citep{wolfe1970convergence}. Due to this lower bound, several FW variants with line-search or short-step were developed that converge linearly in the described setting, see Section~\ref{sec:related_work}

For this setting, following our earlier blueprint from Section~\ref{sec:blueprint}, we prove that FW with open-loop step-sizes converges at a rate of order $\cO(1/t^2)$, which is non-asymptotically faster than FW with line-search or short-step. Our result can be thought of as the non-asymptotic version of Proposition~2.2 in \citet{bach2021effectiveness}.
Contrary to the result of \citet{bach2012equivalence}, our result is in primal gap, we do not require bounds on the third-order derivatives of the objective, and we do not invoke affine invariance of FW to obtain acceleration.
To prove our result, we require two assumptions. The first assumption stems from \emph{active set identification}, that is, the concept of identifying the face $\cC^*\subseteq\cC$ containing the optimal solution $x^*\in\argmin_{x\in\cC}f(x)$ to then apply faster methods whose convergence rates then often only depend on the dimension of the optimal face \citep{hager2006new, bomze2019first, bomze2020active}.
Here, it is possible to determine the number of iterations necessary for FW with open-loop step-sizes to identify the optimal face when the following regularity assumption, already used in, for example, \citet{garber2020revisiting, li2021momentum}, is satisfied.
\begin{assumption}[Strict complementarity]\label{ass:strict_comp}
Let $\cC \subseteq \R^d$ be a polytope and let $f\colon \cC \to \R$ be differentiable in an open set containing $\cC$. Suppose that $x^*\in\argmin_{x\in\cC}f(x)$ is unique and contained in an at least one-dimensional
face $\cC^*$ of $\cC$ and that there exists $\kappa > 0$ such that if $p\in \vertices \left(\cC\right)\setminus \cC^*$, then $\langle \nabla f(x^*), p-x^*\rangle \geq \kappa$; otherwise, if $p\in \vertices \left(\cC^*\right)$, then $\langle \nabla f(x^*), p-x^*\rangle = 0$.
\end{assumption}

In the proof of Theorem~$5$ in \citet{garber2020revisiting}, the authors showed that there exists an iterate $\fwt \in \N$ such that for all $t\geq \fwt$, the FW vertices $p_t$ lie in the optimal face, assuming that the objective function is strongly convex. Below, we generalize their result to convex functions satisfying \eqref{eq:heb}.

\begin{lemma}[Active set identification]\label{lemma:active_face_identification}
Let $\cC \subseteq \R^d$ be a polytope of diameter $\delta > 0$, let $f\colon \cC \to \R$ be a convex and $L$-smooth function satisfying a $(\mu, \theta)$-\eqref{eq:heb} for some $\mu > 0 $ and $\theta \in ]0, 1/2]$ with unique minimizer $x^*\in\argmin_{x\in\cC}f(x)$, and suppose that there exists $\kappa > 0$ such that Assumption~\ref{ass:strict_comp} is satisfied.
Let
$\fwt =\lceil  8 L\delta^2 \left({2\mu L \delta}/{\kappa}\right)^{1/\theta}\rceil$,
$T\in\N$, and $\eta_t = \frac{4}{t+4}$ for all $t\in\Z$.
Then, for the iterates of Algorithm~\ref{algo:fw} with step-size $\eta_t$, it holds that $p_t\in \vertices \left(\cC^*\right)$ for all $t\in\{\fwt, \ldots, T-1\}$.
\end{lemma}
\begin{proof}
Let $t\in\{\fwt, \ldots, T-1\}$.
Note that in Line~\ref{line:p_t_det} of Algorithm~\ref{algo:fw}, $p_t \in \argmin_{p\in \cC} \langle \nabla f(x_t), p - x_t\rangle $ can always be chosen such that $p_t \in \argmin_{p \in \vertices(\cC)} \langle \nabla f(x_t), p - x_t\rangle$. For $p \in \vertices (\cC)$, it holds that
\begin{align}\label{eq:any_vertex}
    \langle\nabla f(x_t), p - x_t\rangle & = \langle\nabla f(x_t) -\nabla f(x^*) + \nabla f(x^*), p - x^* + x^* - x_t\rangle \nonumber \\
    & = \langle\nabla f(x_t) -\nabla f(x^*), p - x_t\rangle + \langle \nabla f(x^*), p - x^* \rangle + \langle \nabla f(x^*), x^* - x_t\rangle.
\end{align}
We distinguish between vertices $p\in \vertices \left(\cC\right) \setminus \cC^*$ and vertices $p\in \vertices \left(\cC^*\right)$. First, suppose that $p\in \vertices \left(\cC\right) \setminus \cC^*$. Using strict complementarity, Cauchy-Schwarz, $L$-smoothness, and \eqref{eq:heb} to bound \eqref{eq:any_vertex} yields
\begin{align*}
    \langle\nabla f(x_t), p - x_t\rangle & \geq - \|\nabla f(x_t) - \nabla f(x^*)\|_2 \|p - x_t\|_2 + \kappa + \langle \nabla f(x^*), x^* - x_t\rangle\nonumber \\
    &\geq \kappa -L\delta \|x_t-x^*\|_2 + \langle \nabla f(x^*), x^* - x_t\rangle \nonumber \\
    & \geq \kappa - \mu L\delta h_t^\theta + \langle \nabla f(x^*), x^* - x_t\rangle.
\end{align*}
Next, suppose that $p\in \vertices \left(\cC^*\right)$. Using strict complementarity, Cauchy-Schwarz, $L$-smoothness, and \eqref{eq:heb} to bound \eqref{eq:any_vertex} yields
\begin{align*}
    \langle\nabla f(x_t), p - x_t\rangle & \leq \|\nabla f(x_t) - \nabla f(x^*)\|_2 \|p - x_t\|_2 + \langle \nabla f(x^*), x^* - x_t\rangle \\
    &\leq L \delta \|x_t-x^*\|_2 + \langle \nabla f(x^*), x^* - x_t\rangle \\
    & \leq \mu L\delta h_t^\theta + \langle \nabla f(x^*), x^* - x_t\rangle.
\end{align*}
By Proposition~\ref{prop:generalization_jaggi}, 
$\mu L \delta h_t^\theta  \leq \mu L \delta h_\fwt^\theta  \leq \mu L \delta \left(\frac{8L\delta^2}{8 L \delta^2\left({2\mu L \delta}/{\kappa}\right)^{1/\theta} +3 }\right)^\theta  < \frac{\kappa}{2}$.
Hence, for $t\in\{\fwt, \ldots, T-1\}$,
\begin{equation*}
    \langle \nabla f(x_t), p-x_t \rangle =  \begin{cases}
         > \frac{\kappa}{2} + \langle \nabla f(x^*), x^* - x_t\rangle, & \text{if} \ p \in \vertices \left(\cC\right) \setminus \cC^* \\
         < \frac{\kappa }{2} + \langle \nabla f(x^*), x^* - x_t\rangle, & \text{if} \ p \in \vertices \left(\cC^*\right).
\end{cases}
\end{equation*}
Then, by optimality of $p_t$, for all iterations $t\in\{\fwt, \ldots, T-1\}$ of Algorithm~\ref{algo:fw}, it holds that $p_t \in \vertices \left(\cC^*\right)$.
\end{proof}

In addition, we assume the optimal solution $x^*\in\argmin_{x\in\cC}f(x)$ to be in the relative interior of an at least one-dimensional face $\cC^*$ of  $\cC$.

\begin{assumption}[Optimal solution in the relative interior of a face of $\cC$]\label{ass:opt_in_face}
Let $\cC \subseteq \R^d$ be a polytope and let $f\colon \cC \to \R$. Suppose that $x^*\in\argmin_{x\in\cC}f(x)$ is unique and contained in the relative interior of an at least one-dimensional face $\cC^*$ of $\cC$, that is, there exists $\beta > 0 $ such that $\emptyset \neq B_\beta (x^*) \cap \aff(\cC^*) \subseteq \cC$.
\end{assumption}

Using Assumption~\ref{ass:opt_in_face}, \citet{bach2021effectiveness} derived the following scaling inequality, a variation of \eqref{eq:scaling_int}.

\begin{lemma}[\citealp{bach2021effectiveness}]\label{lemma:scaling_bach}
    Let $\cC \subseteq \R^d$ be a polytope, let $f\colon \cC \to \R$ be a convex and $L$-smooth function with unique minimizer $x^*\in\argmin_{x\in\cC}f(x)$, and suppose that there exists $\beta > 0$ such that Assumption~\ref{ass:opt_in_face} is satisfied. 
    Then, for all $x\in \cC$ such that $p \in \argmin_{v\in \cC} \langle \nabla f(x), v \rangle \subseteq \cC^* $, it holds that 
    \begin{align}\tag{Scaling-BOR}\label{eq:scaling_bor}
        \langle \nabla f(x), x - p \rangle & \geq \beta \|\Pi \nabla f(x)\|_2,
    \end{align}
    where $\Pi x$ denotes the orthogonal projection of $x\in\R^d$ onto the span of $\{x^* - p \mid p \in \cC^* \}$.
    \end{lemma}
\begin{proof}
Suppose that $x \in \cC$ such that $p \in \argmin_{v\in \cC} \langle \nabla f(x), v \rangle \subseteq \cC^*$. Then,
\begin{align*}
    \langle \nabla f(x), x - p \rangle  & = \max_{v\in \cC^* } \langle \nabla f(x), x - v \rangle \\
    &\geq \langle\nabla f(x), x - x^* \rangle + \langle \nabla f(x), \beta \frac{\Pi \nabla f(x) }{\|\Pi \nabla f(x) \|_2} \rangle \\
    & = \langle\nabla f(x), x - x^* \rangle + \langle \Pi \nabla f(x) + (\Iota - \Pi) \nabla f(x), \beta \frac{\Pi \nabla f(x) }{\|\Pi \nabla f(x) \|_2}\rangle \\
    & = \langle\nabla f(x), x - x^* \rangle + \beta \|\Pi \nabla f(x)\|_2\\
    & \geq \beta \|\Pi \nabla f(x)\|_2,
\end{align*}
where the first equality follows from the construction of $p \in \argmin_{v\in \cC} \langle \nabla f(x), v \rangle$, the first inequality follows from the fact that the maximum is at least as large as the maximum attained on $B_\beta (x^*) \cap \cC^*$, the second equality follows from the definition of the orthogonal projection, the third equality follows from the fact that $\Pi x$ and $(\Iota - \Pi) x$ are orthogonal for any $x\in \R^d$, and the second inequality follows from the convexity of $f$.
\end{proof}

To derive the final scaling inequality, we next bound the distance between $x_t$ and the optimal face $\cC^*$. 

\begin{lemma}[Distance to optimal face]\label{lemma:distance_to_optimal_face}
Let $\cC \subseteq \R^d$ be a polytope of diameter $\delta > 0$, let $f\colon \cC \to \R$ be a convex and $L$-smooth function satisfying a $(\mu, \theta)$-\eqref{eq:heb} for some $\mu > 0 $ and $\theta \in ]0, 1/2]$ with unique minimizer $x^*\in\argmin_{x\in\cC}f(x)$, and suppose that there exist $\beta, \kappa > 0$ such that Assumptions~\ref{ass:strict_comp} and~\ref{ass:opt_in_face} are satisfied.
Let $\fwt =  \max\{ \lceil 8L \delta^2\left({\mu}/{\beta}\right)^{1/\theta} \rceil, 
\lceil  8 L\delta^2 \left({2\mu L \delta}/{\kappa}\right)^{1/\theta} \rceil \}$,
$T\in\N$, and $\eta_t = \frac{4}{t+4}$ for all $t\in\Z$.
 Then, for the iterates of Algorithm~\ref{algo:fw} with step-size $\eta_t$, it holds that
\begin{align}\label{eq:statement_1}
    \|(I-\Pi) (x_t - x^*)\|_2 & \leq \frac{\eta_t^4}{\eta_{\fwt}^4} \beta
\end{align}
for all $t\in\{\fwt, \ldots, T-1\}$,
where $\Pi x$ denotes the orthogonal projection of $x\in\R^d$ onto the span of $\{x^* - p \mid p \in \cC^* \}$.
\end{lemma}
\begin{proof}
Let $t\in\{\fwt, \ldots, T-1\}$.
By Lemma~\ref{lemma:active_face_identification}, $p_t \in \vertices (\cC^*)$. Thus, $(\Iota -\Pi) (p_t -x^*) = \zeroterm$,
\begin{align*}
    (\Iota - \Pi) (x_{t+1} - x^*) & = (1- \eta_t) (\Iota - \Pi) (x_t - x^*) + \eta_t (\Iota - \Pi) (p_t -x^*) \\
    & = (1- \eta_t) (\Iota - \Pi) (x_t - x^*)\\
    & = \prod_{i = \fwt}^t (1-\eta_i) (\Iota - \Pi) (x_\fwt - x^*) \\
    & = \frac{\fwt (\fwt+1)(\fwt+2) (\fwt+3)}{(t+1)(t+2)(t+3)(t+4)} (\Iota - \Pi) (x_\fwt - x^*),
\end{align*}
and
$\|(I-\Pi) (x_{t+1} - x^*)\|_2 \leq \frac{\eta_{t+1}^4}{\eta_\fwt^4} \|(I-\Pi) (x_\fwt - x^*)\|_2  \leq \frac{\eta_{t+1}^4}{\eta_\fwt^4} \beta$,
where the last inequality follows from Lemma~\ref{lemma:dist_to_opt}.
\end{proof}

We derive the second scaling inequality below.

\begin{lemma}\label{lemma:scaling_boundary}
Let $\cC \subseteq \R^d$ be a polytope of diameter $\delta > 0$, let $f\colon \cC \to \R$ be an $\alpha_f$-strongly convex and $L$-smooth function with unique minimizer $x^*\in\argmin_{x\in\cC}f(x)$, and suppose that there exist $\beta, \kappa > 0$ such that Assumptions~\ref{ass:strict_comp} and~\ref{ass:opt_in_face} are satisfied.
Let $M = \max_{x\in \cC}\|\nabla f(x)\|_2$, $\fwt = \max\{ \lceil {16L \delta^2}/{\alpha_f\beta^2}\rceil, \lceil  {64 L^3\delta^4}/{\alpha_f\kappa^2} \rceil \}$,
$T\in\N$, and $\eta_t = \frac{4}{t+4}$ for all $t\in\Z$.
Then, for the iterates of Algorithm~\ref{algo:fw} with step-size $\eta_t$ and $t\in\{\fwt, \ldots, T-1\}$, it holds that $h_t \leq \frac{\eta_t^4}{\eta_{\fwt}^4} \beta M$ or
\begin{align}\label{eq:scaling_cvx}\tag{Scaling-CVX}
    \|\Pi \nabla f(x_t)\|_2 \geq \sqrt{\frac{\alpha_f}{2}} \sqrt{h_t} -  \frac{\eta_t^2}{\eta_{\fwt}^2}\sqrt{\frac{\alpha_f\beta M}{2}}-\frac{\eta_t^4}{\eta_\fwt^4}L \beta,
\end{align}
where $\Pi x$ denotes the orthogonal projection of $x\in\R^d$ onto the span of $\{x^* - p \mid p \in \cC^* \}$.
\end{lemma}
\begin{proof}
Given a vector $x\in\R^d$, let $\Pi_{\aff(\cC^*)}x$ denote the projection of $x$ onto $\aff(\cC^*)$, that is, $\Pi_{\aff(\cC^*)}x\in \argmin_{y\in\aff(\cC^*)}\|y-x\|_2$.
We first demonstrate how to express $\Pi_{\aff(\cC^*)}$ using $\Pi$.
Since
$\aff(\cC^*) = x^* + \mathspan(\{x^*-p \mid p\in\cC^*\})$, there has to exist some $y\in\R^d$ such that 
$\Pi_{\aff(\cC^*)}x = (I-\Pi)x^*  + \Pi x + \Pi y$.
By orthogonality of $\Pi$, we have
$\|\Pi_{\aff(\cC^*)}x - x\|_2 = \|(I-\Pi)x^*- (I-\Pi) x+\Pi y \|_2 = \|(I-\Pi)x^*- (I-\Pi) x \|_2 +\|\Pi y\|_2$.
The right-hand side is minimized when $\Pi y = \zeroterm$. Thus,
$\Pi_{\aff(\cC^*)}x = (I-\Pi)x^* + \Pi x \in  \argmin_{y\in\aff(\cC^*)}\|y-x\|_2$.
Let $t\in\{\fwt, \ldots, T-1\}$.
By Lemma~\ref{lemma:dist_to_opt}, $\|x_t -x^*\|_2 \leq \beta$ and, thus, by Assumption~\ref{ass:opt_in_face}, $\Pi_{\aff(\cC^*)}x_t\in \cC^*$.
By $L$-smoothness of $f$, it holds that
$\|\nabla f(x_t) - \nabla f(\Pi_{\aff(\cC^*)}x_t)\|_2 \leq L \|x_t - \Pi_{\aff(\cC^*)}x_t\|_2 = L\|(I-\Pi) (x_t-x^*)\|_2$.
By Lemma~\ref{lemma:distance_to_optimal_face}, it then holds that
\begin{align}\label{eq:pre_proj_grad}
    \|\nabla f(x_t) - \nabla f(\Pi_{\aff(\cC^*)}x_t)\|_2 \leq \frac{\eta_t^4}{\eta_{\fwt}^4}L\beta.
\end{align}
Since for any $x\in \R^d$, we have that 
$\|\Pi x\|_2 \leq \|\Pi x\|_2 + \|(I-\Pi) x\|_2 = \|x\|_2$,
Inequality \eqref{eq:pre_proj_grad} implies that $\|\Pi \nabla f(x_t) - \Pi \nabla f(\Pi_{\aff(\cC^*)}x_t)\|_2 \leq \frac{\eta_t^4}{\eta_{\fwt}^4}L\beta$.
Combined with the triangle inequality,
$\|\Pi \nabla f(\Pi_{\aff(\cC^*)}x_t)\|_2  \leq \|\Pi \nabla f(x_t)\|_2 + \|\Pi \nabla f(x_t) - \Pi \nabla f(\Pi_{\aff(\cC^*)}x_t)\|_2 \leq  \|\Pi \nabla f(x_t)\|_2 + \frac{\eta_t^4}{\eta_{\fwt}^4}L\beta$,
which we rearrange to
\begin{align}\label{eq:proj_grad_bound}
    \|\Pi \nabla f(\Pi_{\aff(\cC^*)}x_t)\|_2 - \frac{\eta_t^4}{\eta_{\fwt}^4}L\beta\leq \|\Pi \nabla f(x_t)\|_2.
\end{align}
For the remainder of the proof, we bound $\|\Pi \nabla f(\Pi_{\aff(\cC^*)}x_t)\|_2$ from below. 
To do so, define the function $g\colon \cC\cap B_\beta(x^*) \to \R$ via
$g(x) := f(\Pi_{\aff(\cC^*)}x) = f((I-\Pi)x^* + \Pi x)$.
The gradient of $g$ at $x\in \cC\cap B_\beta(x^*)$ is
$\nabla g(x) = \Pi \nabla f(\Pi_{\aff(\cC^*)}x)=\Pi \nabla f((I-\Pi)x^*+\Pi x)$.
Since $f$ is $\alpha_f$-strongly convex in $\cC$ and $g(x) = f(x)$ for all $x\in\aff(\cC^*) \cap B_\beta(x^*)$, $g$ is $\alpha_f$-strongly convex in $\aff(\cC^*) \cap B_\beta(x^*)$.
Since the projection onto $\aff(\cC^*)$ is idempotent,
$\Pi_{\aff(\cC^*)}x_t \in \aff(\cC^*) \cap B_\beta(x^*)$, and $g$ is $\alpha_f$-strongly convex in $\aff(\cC^*) \cap B_\beta(x^*)$, it holds that
$
\|\Pi \nabla f(\Pi_{\aff(\cC^*)}x_t)\|_2  = \|\Pi \nabla f(\Pi_{\aff(\cC^*)}^2x_t)\|_2 = \|\nabla g(\Pi_{\aff(\cC^*)}x_t)\|_2 \geq \sqrt{\frac{\alpha_f}{2}} \sqrt{g(\Pi_{\aff(\cC^*)}x_t) - g(x^*)} = \sqrt{\frac{\alpha_f}{2}} \sqrt{f(\Pi_{\aff(\cC^*)}x_t) - f(x^*)}
$.
Suppose that $h_t \geq \frac{\eta_t^4}{\eta_{\fwt}^4} \beta M$. Then, by Lemma~\ref{lemma:distance_to_optimal_face} and Cauchy-Schwarz, we obtain
$h_t - \langle \nabla f(x_t), (I-\Pi)(x_t-x^*)\rangle \geq h_t -\frac{\eta_t^4}{\eta_{\fwt}^4} \beta M \geq 0$. Combined with convexity of $f$, it holds that
\begin{align*}
    \|\Pi \nabla f(\Pi_{\aff(\cC^*)}x_t)\|_2 & \geq \sqrt{\frac{\alpha_f}{2}} \sqrt{f(x_t) + \langle \nabla f(x_t), \Pi_{\aff(\cC^*)}x_t - x_t\rangle  - f(x^*)}\\
    &=  \sqrt{\frac{\alpha_f}{2}} \sqrt{h_t - \langle \nabla f(x_t), (I-\Pi)(x_t -x^*) \rangle }\\
    &\geq \sqrt{\frac{\alpha_f}{2}} \sqrt{h_t -\frac{\eta_t^4}{\eta_{\fwt}^4} \beta M}.
\end{align*}
Since for $a, b \in \R$ with $a \geq b \geq 0$, we have $\sqrt{a -b} \geq  \sqrt{a} - \sqrt{b}$, we obtain
$\|\Pi \nabla f(\Pi_{\aff(\cC^*)}x_t)\|_2  \geq \sqrt{\frac{\alpha_f}{2}} (\sqrt{h_t} -\sqrt{\frac{\eta_t^4}{\eta_{\fwt}^4} \beta M }) =  \sqrt{\frac{\alpha_f}{2}} (\sqrt{h_t} - \frac{\eta_t^2}{\eta_{\fwt}^2}\sqrt{\beta M})$.
Combined with \eqref{eq:proj_grad_bound}, we obtain \eqref{eq:scaling_cvx}.
\end{proof}

Finally, we prove that when the feasible region $\cC$ is a polytope, the objective function $f$ is strongly convex, and the unique minimizer $x^* \in \argmin_{x\in \cC} f(x)$ lies in the relative interior of an at least one-dimensional face $\cC^*$ of $\cC$, FW with the open-loop step-size $\eta_t = \frac{4}{t+4}$ converges at a rate of order $\cO(1/t)$ for iterations $ t\leq \fwt$ and at a non-asymptotic rate of order $\cO(1/t^2)$ for iterations $t\geq \fwt$, where $\fwt$ is defined as in Lemma~\ref{lemma:scaling_boundary}.

\begin{theorem}[Optimal solution in the relative interior of a face of $\cC$]\label{thm:polytope}
Let $\cC \subseteq \R^d$ be a polytope of diameter $\delta > 0$, let $f\colon \cC \to \R$ be an $\alpha_f$-strongly convex and $L$-smooth function with unique minimizer $x^*\in\argmin_{x\in\cC}f(x)$, and suppose that there exist $\beta, \kappa > 0$ such that Assumptions~\ref{ass:strict_comp} and~\ref{ass:opt_in_face} are satisfied.
Let $M = \max_{x\in \cC} \|\nabla f(x)\|_2$, 
$\fwt = \max\left\{ \left\lceil (16L \delta^2) / (\alpha_f\beta^2) \right\rceil, \left\lceil  (64 L^3\delta^4)/(\alpha_f\kappa^2) \right\rceil \right\}$,
$T\in\N$, and $\eta_t = \frac{4}{t+4}$ for all $t\in\Z$.
Then, for the iterates of Algorithm~\ref{algo:fw} with step-size $\eta_t$, it holds that
\begin{align}\label{eq:sol_poly}
    h_t \leq \eta_{t-2}^2 \max \left\{ \frac{h_\fwt}{\eta_{\fwt-1}^2},  \frac{ B^2}{A^2} + B, \frac{D}{\eta_\fwt^2} + E\right\}
\end{align}
for all $t\in\{\fwt, \ldots, T\}$,
where
\begin{align}\label{eq:big_letters}
    A & =\frac{\sqrt{\alpha_f}\beta}{2\sqrt{2}}, \qquad B =  \frac{L \delta^2}{2}
    + \frac{\beta\sqrt{\alpha_f \beta M}}{\eta_\fwt 2\sqrt{2}} + \frac{L\beta^2}{\eta_\fwt 2}, \qquad D   = \beta M , \qquad  E =\frac{L \delta^2}{2}.
\end{align}
\end{theorem}
\begin{proof}
Let $t\in\{\fwt, \ldots, T-1\}$ and suppose that $h_t \geq \frac{\eta_t^4}{\eta_{\fwt}^4} \beta M$.
Combine \eqref{eq:always_combine_with_this} and \eqref{eq:start_progress_bound} to obtain
$h_{t+1}  \leq (1 - \frac{\eta_t}{2}) h_t - \frac{\eta_t}{2} \langle\nabla f(x_t), x_t - p_t\rangle + \eta_t^2\frac{ L \|x_t-p_t\|^2_2}{2}$.
Plugging \eqref{eq:scaling_bor} and \eqref{eq:scaling_cvx} into this inequality results in
$h_{t+1} \leq (1 - \frac{\eta_t}{2})h_t - \frac{\eta_t \beta}{2} (\sqrt{\frac{\alpha_f}{2}} \sqrt{h_t} -  \frac{\eta_t^2}{\eta_{\fwt}^2}\sqrt{\frac{\alpha_f\beta M}{2}}-\frac{\eta_t^4}{\eta_\fwt^4}L \beta) + \frac{\eta_t^2 L \delta^2}{2}$.
Since $\eta_t / \eta_\fwt \leq 1$ for all $t\in\{\fwt, \ldots, T-1\}$, it holds that
\begin{align}\label{eq:ready_for_lemma_sequences_border}
    h_{t+1} & \leq \left(1 - \frac{\eta_t}{2}\right)h_t - \eta_t \frac{\sqrt{\alpha_f}\beta}{2\sqrt{2}}\sqrt{h_t}
    + \eta_t^2 \left(\frac{L \delta^2}{2}
    + \frac{\beta\sqrt{\alpha_f \beta M}}{\eta_\fwt 2\sqrt{2}} + \frac{L\beta^2}{\eta_\fwt 2}\right).
\end{align}
Let $A, B, C$ as in \eqref{eq:big_letters}, $C_t=1$
for all $t\in\{\fwt, \ldots, T-1\}$, and $\psi = 1/2$.
Ideally, we could now apply Lemma~\ref{lemma:sequences}. However, Inequality~\eqref{eq:ready_for_lemma_sequences_border} is only guaranteed to hold in case that $h_t \geq \frac{\eta_t^4}{\eta_{\fwt}^4} \beta M$. Thus, we have to extend the proof of Lemma~\ref{lemma:sequences} for the case that $h_t \leq \frac{\eta_t^4}{\eta_{\fwt}^4} \beta M$. In case $h_t \leq \frac{\eta_t^4}{\eta_{\fwt}^4} \beta M$, \eqref{eq:always_combine_with_this} implies that
$h_{t+1}\leq (1-\eta_t)h_t + \eta_t^2 \frac{L\|x_t - p_t\|_2^2}{2} \leq h_t + \eta_t^2 \frac{L\delta^2}{2}\leq \eta_{t-1}\eta_t( \frac{\beta M}{\eta_{\fwt}^2} +  \frac{L\delta^2}{2}) = \eta_{t-1}\eta_t( \frac{D}{\eta_{\fwt}^2} +  E)$,
where $D = \beta M $ and $E=\frac{L \delta^2}{2}$.
Thus, in the proof of Lemma~\ref{lemma:sequences}, the induction assumption \eqref{eq:cd} has to be replaced by
$h_t \leq \max \left\{  \frac{\eta_{t-2}\eta_{t-1}}{\eta_{\fwt-2}\eta_{\fwt-1}}h_\fwt, \frac{\eta_{t-2}\eta_{t-1} B^2}{A^2} + \eta_{t-2 }\eta_{t-1} BC, \eta_{t-2}\eta_{t-1}( \frac{D}{\eta_{\fwt}^2} +  E)\right\}$.
Then, using the same analysis as in Lemma~\ref{lemma:sequences}, extended by the case that $h_t \leq \frac{\eta_t^4}{\eta_{\fwt}^4} \beta M$, proves that
\eqref{eq:sol_poly} holds
for all $t\in\{\fwt, \ldots, T\}$.
\end{proof}
In the following remark to Theorem~\ref{thm:polytope}, we discuss how to relax strict complementarity.

\begin{remark}[Relaxation of strict complementarity]\label{rem:relaxation_of_strict_complementarity}
In the proof of Theorem~\ref{thm:polytope}, strict complementarity is only needed to guarantee that after a specific iteration $\fwt \in \{1,\ldots, T-1\}$, for all $t\in\{\fwt,\ldots, T-1\}$, it holds that $p_t\in \vertices(\cC^*)$, that is, only vertices that lie in the optimal face get returned by FW's LMO. However, strict complementarity is only a sufficient but not necessary criterion to guarantee that only vertices in the optimal face are obtained from the LMO for iterations $t\in\{\fwt,\ldots, T-1\}$: Consider, for example, the minimization of $f(x) = \frac{1}{2} \|x  - b \|_2^2$ for $b = (0, 1/2, 1/2)^\intercal\in \R^3$ over the probability simplex $\cC=\conv\left(\{e^{(1)}, e^{(2)}, e^{(3)}\}\right)$. Note that $\cC^* = \conv\left(\{ e^{(2)}, e^{(3)}\}\right)$.
It holds that $x^* = b $ and $\nabla f(x^*) = (0, 0 , 0)^\intercal \in \R^3$. Thus, strict complementarity is violated. However, for any $x_t = (u, v , w)^\intercal\in\R^3$ with $u + v + w = 1$ and $u,v,w \geq 0$, it holds, by case distinction, that either $\langle \nabla f(x_t), e^{(1)} -x_t\rangle >  \min\{\langle \nabla f(x_t), e^{(2)}-x_t \rangle, \langle \nabla f(x_t), e^{(3)}-x_t \rangle\}$, or $x^* = x_t$. Thus, $p_t \in \cC^*$ for all $t\geq 0$ without strict complementarity being satisfied.
\end{remark}

\begin{figure}[t]
\captionsetup[subfigure]{justification=centering}
\centering
\begin{tabular}{c c}
     \begin{subfigure}{.3\textwidth}
    \centering
        \includegraphics[width=1\textwidth]{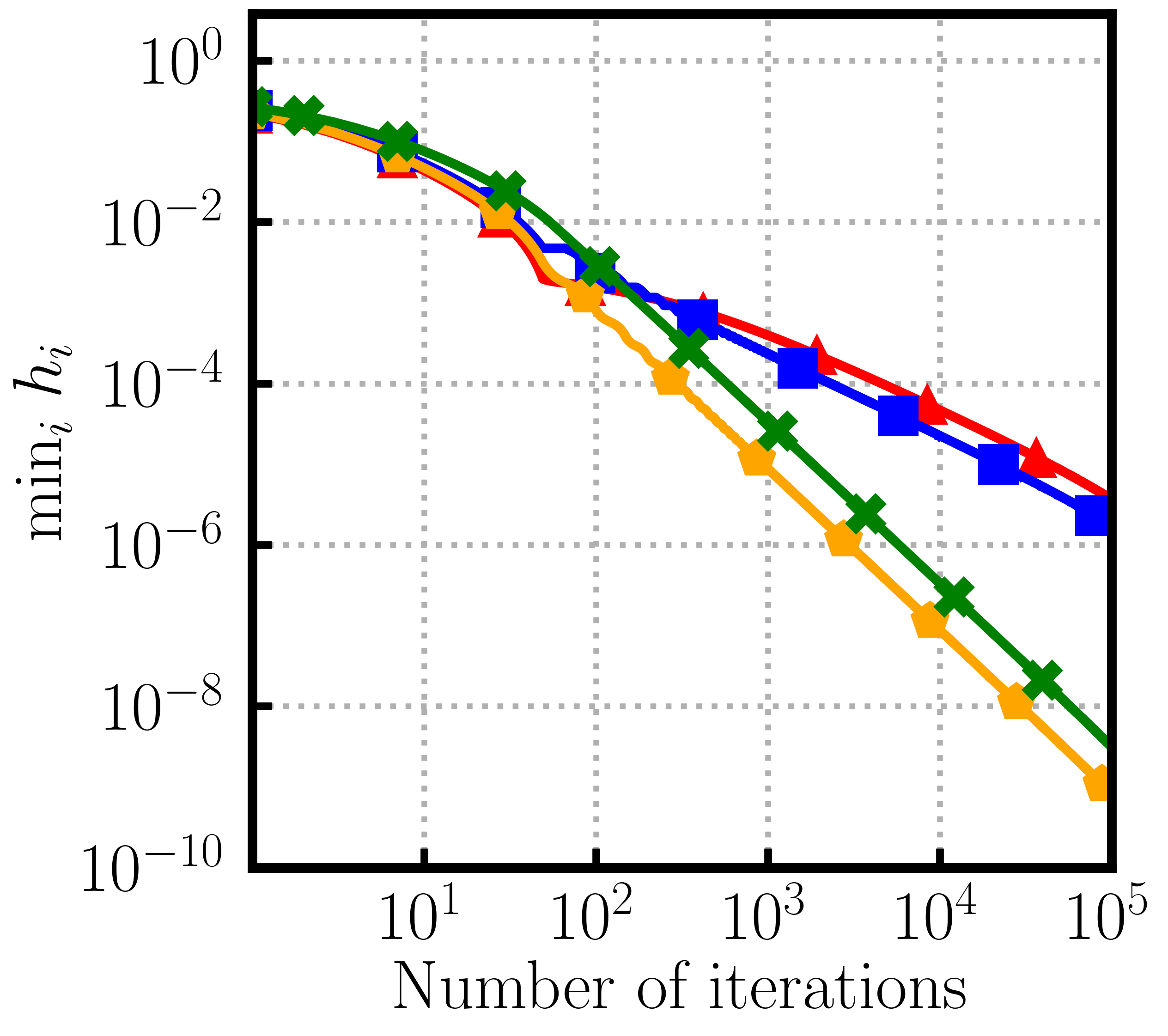}
        \caption{$\rho=\frac{1}{4}$.}
        \label{fig:ls_sublinear}
    \end{subfigure}& 
    \begin{subfigure}{.3\textwidth}
    \centering
        \includegraphics[width=1\textwidth]{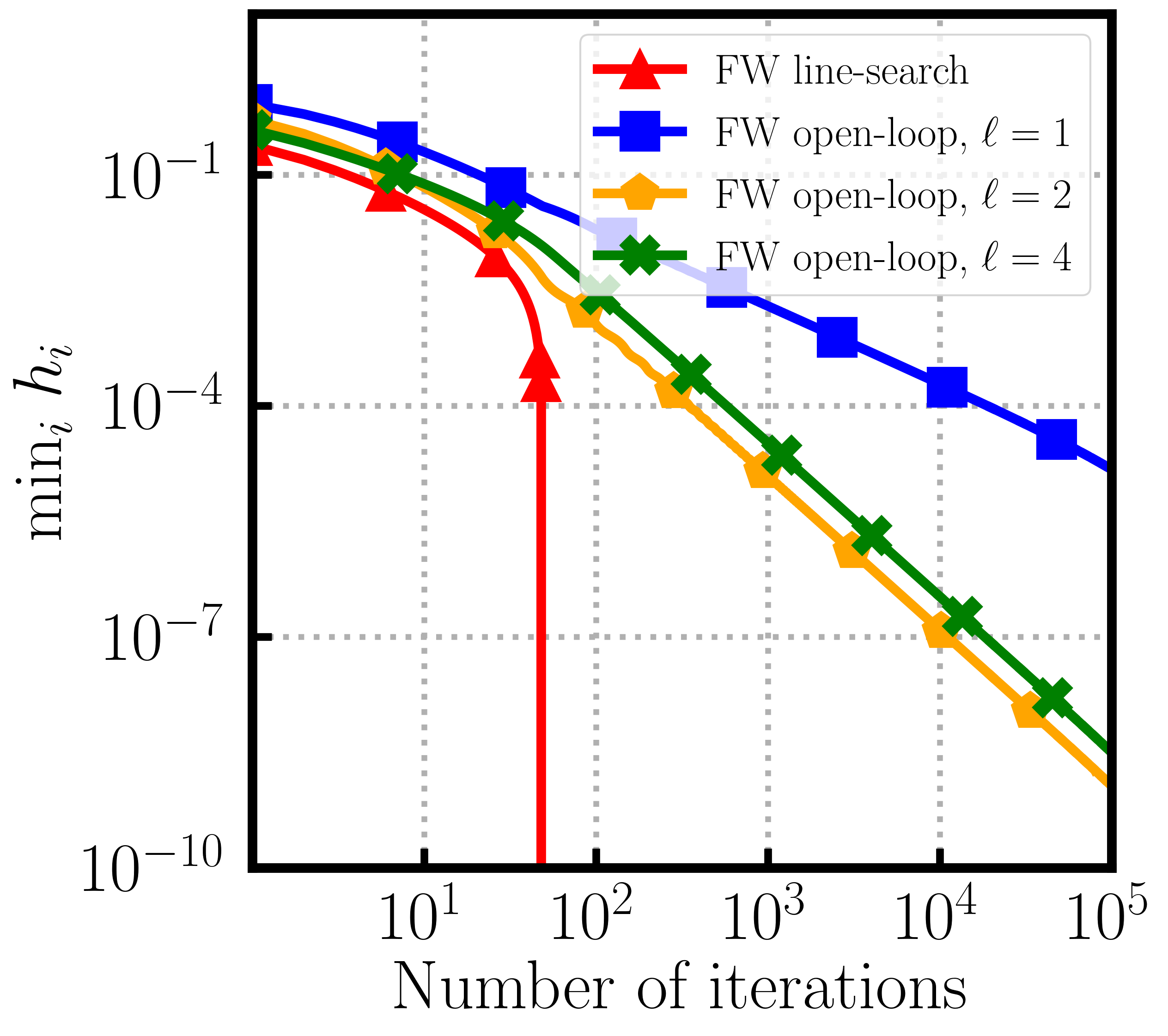}
        \caption{$\rho=2$.}
        \label{fig:ls_linear}
    \end{subfigure} 
\end{tabular}
\caption{
Comparison of FW with different step-sizes when the feasible region $\cC\subseteq\R^{100}$ is the probability simplex, the objective $f(x) = \frac{1}{2}\|x-\rho \bar{\oneterm}\|_2^2$, where $\rho \in\{ \frac{1}{4}, 2\}$, is strongly convex, and the optimal solution $x^*\in\argmin_{x\in\cC}f(x)$ lies in the relative interior of an at least one-dimensional face of $\cC$. 
The $y$-axis represents the minimum primal gap. 
For both settings, FW with open-loop step-sizes $\eta_t = \frac{\ell}{t+\ell}$ converges at a rate of order $\cO(1/t^2)$ when $\ell\in\N_{\geq 2}$ and at a rate of order $\cO(1/t)$ when $\ell=1$.
FW with line-search converges at a rate of order $\cO(1/t)$ when $\rho = \frac{1}{4}$ and linearly when $\rho = 2$. In the latter setting, FW with line-search solves the problem exactly after $|\supp(x^*)|$ iterations.
}\label{fig:experiments_polytope}
\end{figure}

The results in Figure~\ref{fig:experiments_polytope}, see Section~\ref{sec:experiment_polytope} for details, show that when the feasible region $\cC$ is a polytope, $f=\frac{1}{2}\|x-\rho \bar{\oneterm}\|_2^2$, where $\rho\in\{\frac{1}{4},2\}$, is strongly convex, the constrained optimal solution $x^*\in\argmin_{x\in\cC} f(x)$ lies in the relative interior of an at least one-dimensional face of $\cC$, FW with open-loop step-sizes $\eta_t=\frac{\ell}{t+\ell}$, where $\ell\in\N_{\geq 2}$, converges at a rate of order $\cO(1/t^2)$ and FW with open-loop step-size $\eta_t=\frac{1}{t+1}$ converges at a rate of order $\cO(1/t)$. For the same setting, FW with line-search either converges at a rate of order $\cO(1/t)$ when $\rho=\frac{1}{4}$ or linearly when $\rho=2$. We have thus demonstrated both theoretically and in practice that there exist settings for which FW with open-loop step-sizes converges non-asymptotically faster than FW with line-search or short-step.

\section{Algorithmic variants}\label{sec:fw_variants}

In Section~\ref{sec:ol_faster_than_ls_ss}, we established that when the feasible region $\cC$ is a polytope, the objective $f$ is strongly convex, and the unique minimizer $x^*\in\argmin_{x\in\cC}f(x)$ lies in the relative interior of an at least one-dimensional face $\cC^*$ of $\cC$, FW with open-loop step-size $\eta_t = \frac{4}{t+4}$ converges at a rate of order $\cO(1/t^2)$. Combined with the convergence-rate lower bound of $\Omega(1/t^{1+\epsilon})$ for any $\epsilon > 0$ for FW with line-search or short-step by \citet{wolfe1970convergence}, this characterizes a problem setting for which FW with open-loop step-sizes converges non-asymptotically faster than FW with line-search or short-step. However, our accelerated convergence rate only holds when strict complementarity or similar assumptions, see Remark~\ref{rem:relaxation_of_strict_complementarity}, hold. Similarly, the accelerated convergence rate of MFW  \citep{li2021momentum} in the described setting also relies on the assumption of strict complementarity.

Here, we address this gap in the literature and present two FW variants employing open-loop step-sizes that admit convergence rates of order $\cO(1/t^2)$ in the setting of the lower bound due to \citet{wolfe1970convergence} without relying on the assumption of strict complementarity.

\subsection{Decomposition-invariant pairwise Frank-Wolfe algorithm}\label{sec:difw}

Using the proof blueprint from Section~\ref{sec:blueprint}, we derive accelerated convergence rates for the decomposition-invariant pairwise Frank-Wolfe algorithm (DIFW) \citep{garber2016linear} in the setting of the lower bound due to \citet{wolfe1970convergence}.
DIFW with line-search or step-size as in Option 1 in \citet[Algorithm~3]{garber2016linear} converges linearly when the feasible region is a specific type of polytope and the objective function is strongly convex.
Benefits of DIFW are that the convergence rate does not depend on the dimension of the problem but the sparsity of the optimal solution $x^* \in \argmin_{x\in \cC} f(x)$, that is, $|\supp(x^*)| = |\{x^*_i \neq 0 \mid i \in \{1, \ldots, d\}\}| \ll d$, and it is not necessary to maintain a convex combination of the iterate $x_t$ throughout the algorithm's execution. The latter property leads to reduced memory overhead compared to other variants of FW that admit linear convergence rates in the setting of \citet{wolfe1970convergence}.
The main drawback of DIFW is that the method is not applicable to general polytopes, but only feasible regions that are similar to the simplex, that is, of the form described below.
\begin{definition}[Simplex-like polytope (SLP)]\label{def:difw}
Let $\cC\subseteq \R^d$ be a polytope such that $\cC$ can be described as $\cC = \{x\in \R^d \mid x\geq 0, Ax=b\}$ for $A\in\R^{m \times d}$ and $b\in \R^m$ for some $m\in \N$ and all vertices of $\cC$ lie on the Boolean hypercube $\{0, 1\}^d$. Then, we refer to $\cC$ as a \emph{simplex-like polytope} (SLP).
\end{definition}
Examples of SLPs are the probability simplex and the flow, perfect matchings, and marginal polytopes, see \citet{garber2016linear} and references therein for more details.
In this section, we show that DIFW with open-loop step-size $\eta_t = \frac{8}{t+8}$ admits a convergence rate of order up to $\cO(1/t^2)$ when optimizing a function satisfying \eqref{eq:heb} over a SLP.

\begin{algorithm}[th!]
\SetKwInput{Input}{Input} \SetKwInput{Output}{Output}
\SetKwComment{Comment}{$\triangleright$\ }{}
  \caption{Decomposition-invariant pairwise Frank-Wolfe algorithm (DIFW) \citep{garber2016linear}}\label{algo:difw}
  \Input{$x_0\in\cC$, step-sizes $\eta_t\in [0, 1]$ for $t\in\{0,\ldots, T-1\}$.}
  \hrulealg
  {$x_1 \in \argmin_{p\in \cC} \langle \nabla f(x_0), p - x_0 \rangle$}\\
  \For{$t= 0, \ldots, T-1 $}{
        {$p_t^+ \in \argmin_{p\in \cC} \langle \nabla f(x_t), p - x_t \rangle$\label{line:fw_vertex_difw}}\\
        {Define the vector $\tilde{\nabla}f(x_t) \in \R^d$ entry-wise for all $i\in\{1, \ldots, d\}$:
        \begin{equation*}
            (\tilde{\nabla} f(x_t))_i = \begin{cases}
            (\nabla f(x_t))_i, & \text{if} \ (x_t)_i > 0 \\
            -\infty , & \text{if} \ (x_t)_i = 0.
        \end{cases}\label{line:defining_gradient_difw}
        \end{equation*}}\\
        {$p_t^- \in \argmin_{p\in \cC} \langle -\tilde{\nabla}f(x_t), p - x_t\rangle$}\\
        {Let $\delta_t$ be the smallest natural number such that $2^{-\delta_t} \leq \eta_t$, and define the new step-size $\gamma_t \gets 2^{-\delta_t}$.}\\
        {$x_{t+1} \gets x_t + \gamma_t (p_t^+ - p_t^-)$}
    }
\end{algorithm}

\subsubsection{Algorithm overview}
We refer to $p_t^+$ and $p_t^-$ as the FW vertex and away vertex, respectively. At iteration $t\in\{0,\ldots, T\}$, consider the representation of $x_t$ as a convex combination of vertices of $\cC$, that is, $x_t = \sum_{i=0}^{t-1} \lambda_{p_i,t} p_i$, where $p_i \in \vertices (\cC)$  and $\lambda_{p_i, t} \geq 0$ for all $i\in\{0,\ldots, t-1\}$ and $\sum_{i=0}^{t-1}\lambda_{p_i, t} = 1$. 
DIFW takes a step in the direction
$\frac{p_t^+-p_t^-}{\|{p_t^+-p_t^-}\|_2}$, which
moves weight from the away vertex $p_t^-$ to the FW vertex $p_t^+$. Note that DIFW does not need to actively maintain a convex combination of $x_t$ because of the assumption that the feasible region is a SLP.

\subsubsection{Convergence rate of order $\cO(1/t)$}
We first derive a baseline convergence rate of order $\cO(1/t)$ for DIFW with open-loop step-size $\eta_t = \frac{8}{t+8}$.
\begin{proposition}[Convergence rate of order $\cO(1/t)$]\label{prop:baseline_difw}
Let $\cC\subseteq \R^d$ be a SLP of diameter $\delta > 0$ and let $f\colon \cC \to \R$ be a convex and $L$-smooth function with unique minimizer $x^*\in\argmin_{x\in\cC}f(x)$. Let $T\in\N$ and $\eta_t = \frac{8}{t+8}$ for all $t\in\Z$.
Then, for the iterates of Algorithm~\ref{algo:difw} with open-loop step-size $\eta_t$, it holds that
$h_t \leq \frac{32L \delta^2}{t+7} = \eta_{t-1} 4 L \delta^2  $
for all $t\in\{1,\ldots, T\}$.
\end{proposition}
\begin{proof}
Let $t\in\{0, \ldots, T-1\}$.
Feasibility of $x_t$ follows from Lemma~1 in \citet{garber2016linear}. Further, in the proof of Lemma $3$ in \citet{garber2016linear}, it is shown that
\begin{align}\label{eq:difw_basic_bound}
    h_{t+1} & \leq h_t + \frac{\eta_t \langle \nabla f(x_t), p_t^+ - p_t^-\rangle}{2} + \frac{\eta_t^2 L \delta^2}{2}.
\end{align}
Consider an irreducible representation of $x_t$ as a convex sum of vertices of $\cC$, that is, $x_t = \sum_{i=0}^{k}\lambda_{p_i, t} p_i$ such that $p_i\in\vertices(\cC)$ and $\lambda_{p_i, t} > 0$ for all $i \in \{ 0, \ldots , k\}$, where $k\in \N$.
By Observation 1 in \citet{garber2016linear}, it holds that 
$\langle \nabla f(x_t), p_i\rangle \leq \langle \nabla f(x_t), p_t^-\rangle$
for all $i\in\{0, \ldots, k\}$.
Thus, $\langle \nabla f(x_t), x_t - p_t^- \rangle  \leq \langle \nabla f(x_t), x_t - \sum_{i=0}^{k}\lambda_{p_i, t} p_i \rangle \leq \langle \nabla f(x_t), x_t - x_t \rangle  = 0$.
Plugging this inequality into \eqref{eq:difw_basic_bound}, using $\langle \nabla f(x_t), p_t^+ -x_t \rangle \leq -h_t$, and using $h_1 \leq \frac{L\delta^2}{2} $, which is derived in the proof of Theorem~1 in \citet{garber2016linear}, we obtain
\begin{align}
    h_{t+1} & \leq h_t + \frac{\eta_t \langle \nabla f(x_t), p_t^+ -x_t \rangle}{2} + \frac{\eta_t \langle \nabla f(x_t),  x_t - p_t^-\rangle}{2} + \eta_t^2\frac{ L \delta^2}{2} \nonumber\\
    & \leq (1 - \frac{\eta_t}{2}) h_t + \eta_t^2\frac{ L \delta^2}{2} \label{eq:need_for_difw}\\
    & \leq \prod_{i = 1}^t (1 - \frac{\eta_i}{2})h_1 + \frac{L\delta^2}{2} \sum_{i = 1}^t \eta_i^2 \prod_{j = i+1}^t (1 - \frac{\eta_j}{2}) \nonumber\\
    & = \frac{5\cdot 6 \cdot 7 \cdot 8}{(t+5)(t+6)(t+7)(t+8)} h_1 + \frac{L \delta^2}{2} \sum_{i = 1}^t\frac{8^2}{(i+8)^2} \frac{(i+5)(i+6)(i+7)(i+8)}{(t+5) (t+6)(t+7)(t+8)}\nonumber\\
    & \leq \frac{64L\delta^2}{2}(\frac{1}{(t+7) (t+8)} + \frac{t}{(t+7) (t+8)}) \nonumber\\
    & \leq \frac{32L\delta^2}{t+8}.\nonumber
\end{align}
\end{proof}

\subsubsection{{Convergence rate of order up to $\cO(1/t^2)$}}
Then, acceleration follows almost immediately from the analysis performed in \citet{garber2016linear}.
\begin{theorem}[Convergence rate of order up to $\cO(1/t^2)$]\label{thm:difw_slp}
Let $\cC\subseteq\R^d$ be a SLP of diameter $\delta > 0$ and let $f\colon \cC \to \R$ be a convex and $L$-smooth function satisfying a $(\mu, \theta)$-\eqref{eq:heb} for some $\mu > 0$ and $\theta \in [0, 1/2]$.
Let $T\in\N$ and $\eta_t = \frac{8}{t+8}$ for all $t\in\Z$.
Then, for the iterates of Algorithm~\ref{algo:difw} with open-loop step-size $\eta_t$, it holds that
\begin{align}\label{eq:difw_pg}
    h_t  & \leq \max \left\{ \eta_{t-2}^{1/(1-\theta)}\frac{L\delta^2}{2}, \left(\eta_{t-2} 2\mu L \delta^2\sqrt{|\supp(x^*)|} \right)^{1/(1-\theta)} + \eta_{t-2}^2\frac{ L \delta^2}{2} \right\}.
\end{align}
for all $t\in\{1,\ldots, T\}$.
\end{theorem}
\begin{proof}
Let $t\in\{1,\ldots, T-1\}$.
We can extend Lemma~$3$ in \citet{garber2016linear} from $\alpha_f$-strongly convex functions to convex functions satisfying \eqref{eq:heb}. Strong convexity is only used to show that $\Delta_t := \sqrt{\frac{2 |\supp(x^*)| h_t}{\alpha_f}}$ satisfies
$\Delta_t \geq \sqrt{|\supp(x^*)|} \|x_t - x^*\|_2$.
Here, we instead define $\Delta_t:= \sqrt{|\supp(x^*)|}\mu h_t^\theta$ for a function $f$ satisfying a $(\mu, \theta)$-\eqref{eq:heb}. Then, $\Delta_t \geq \sqrt{|\supp(x^*)|} \|x_t - x^*\|_2$. By Lemma~$3$ in \citet{garber2016linear}, we have
$h_{t+1} \leq h_t -\frac{\eta_th_t^{1-\theta}}{2\mu \sqrt{|\supp(x^*)|}} + \eta_t^2\frac{ L \delta^2}{2}$.
Combined with \eqref{eq:need_for_difw},
\begin{align}\label{eq:reason_for_new_ss}
    h_{t+1} & \leq \left(1 - \frac{\eta_t}{4}\right)h_t -\frac{\eta_th_t^{1-\theta}}{4\mu \sqrt{|\supp(x^*)|}} + \eta_t^2\frac{ L \delta^2}{2}.
\end{align}
Using the same proof technique as in Lemma~\ref{lemma:sequences}, we prove that
\begin{align}\label{eq:difw_cd}
    h_t & \leq \max \left\{ \left(\eta_{t-2}\eta_{t-1}\right)^{1/(2(1-\theta))}\frac{L\delta^2}{2}, \left(\eta_{t-2}\eta_{t-1} \left(2\mu L \delta^2\sqrt{|\supp(x^*)|}\right)^2 \right)^{1/(2(1-\theta))} + \eta_{t-2}\eta_{t-1}\frac{ L \delta^2}{2} \right\}
\end{align}
for all $t\in\{1,\ldots, T\}$, which then implies \eqref{eq:difw_pg}.
For $t = 1$, $h_1 \leq \frac{L\delta^2}{2}$ and \eqref{eq:difw_cd} holds. Suppose that \eqref{eq:difw_cd} is satisfied for a specific iteration $t\in\{1,\ldots, T-1\}$. We distinguish between two cases.
First, suppose that
$h_t \leq (\eta_t 2\mu L \delta^2\sqrt{|\supp(x^*)|} )^{1/(1-\theta)}
$.
Plugging this bound on $h_t$ into \eqref{eq:reason_for_new_ss} yields
$h_{t+1} \leq (\eta_t 2\mu L \delta^2\sqrt{|\supp(x^*)|} )^{1/(1-\theta)} + \frac{ \eta_t^2L \delta^2}{2} \leq (\eta_{t-1}\eta_t (2\mu L \delta^2\sqrt{|\supp(x^*)|})^2 )^{1/(2(1-\theta))} + \eta_{t-1}\eta_t\frac{ L \delta^2}{2}$.
Next, suppose that
$h_t \geq (\eta_t 2\mu L \delta^2\sqrt{|\supp(x^*)|} )^{1/(1-\theta)}$.
Plugging this bound on $h_t$ into \eqref{eq:reason_for_new_ss} and using the induction assumption yields 
{\footnotesize
\begin{align}\label{eq:need_small_steps}
    h_{t+1} &\leq  (1 - \frac{\eta_t}{4})h_t  + 0 \\
    &= \frac{t+6}{t+8} h_t\nonumber\\
    &\leq \frac{\eta_{t}}{\eta_{t-2}} h_t\nonumber\\
    &\leq  \frac{\eta_{t}}{\eta_{t-2}} \max \left\{ \left(\eta_{t-2}\eta_{t-1}\right)^{1/(2(1-\theta))}\frac{L\delta^2}{2}, \left(\eta_{t-2}\eta_{t-1} \left(2\mu L \delta^2\sqrt{|\supp(x^*)|}\right)^2 \right)^{1/(2(1-\theta))} + \eta_{t-2}\eta_{t-1}\frac{ L \delta^2}{2} \right\}\nonumber\\
    & \leq \max \left\{ \left(\eta_{t-1}\eta_{t}\right)^{1/(2(1-\theta))}\frac{L\delta^2}{2}, \left(\eta_{t-1}\eta_{t} \left(2\mu L \delta^2\sqrt{|\supp(x^*)|}\right)^2 \right)^{1/(2(1-\theta))} + \eta_{t-1}\eta_{t}\frac{ L \delta^2}{2} \right\}, \nonumber
\end{align}
}where the last inequality holds due to $\frac{\eta_t}{\eta_{t-2}}(\eta_{t-2}\eta_{t-1})^{1/(2(1-\theta))} \leq (\eta_{t-1}\eta_{t})^{1/(2(1-\theta))}$ for $\frac{\eta_t}{\eta_{t-2}}\in [0,1]$ and $1/(2(1-\theta)) \in [1/2,1]$.
In either case, \eqref{eq:difw_cd} is satisfied for $t+1$. By induction, the theorem follows.
\end{proof}

Below, we discuss the technical necessity for $\eta_t = \frac{8}{t+8}$ instead of $\eta_t = \frac{4}{t+4}$ in Theorem~\ref{thm:difw_slp}.

\begin{remark}[Necessity of $\eta_t = \frac{8}{t+8}$]\label{rem:necessity_for_switch_difw}
    Note that Inequality~\eqref{eq:reason_for_new_ss} is responsible for making our usual proof with $\eta_t = \frac{4}{t+4}$, $t\in\Z$, impossible. Indeed, for $\eta_t = \frac{4}{t+4}$,
    $(1- \frac{\eta_t}{4}) = \frac{t+3}{t+4}$,
    which is not enough progress in, for example, \eqref{eq:need_small_steps} assuming that $\theta = \frac{1}{2}$, to obtain a convergence rate of order $\cO(1/t^2)$.
\end{remark}

\subsection{Away-step Frank-Wolfe algorithm}\label{sec:afw}
\begin{algorithm}[h]
\SetKwInput{Input}{Input} \SetKwInput{Output}{Output}
\SetKwComment{Comment}{$\triangleright$\ }{}
\caption{Away-step Frank-Wolfe algorithm (AFW) with open-loop step-sizes}\label{algo:afw}
\Input{$x_0\in \vertices (\cC)$, step-sizes $\eta_t\in [0, 1]$ for $t\in\{0, \ldots, T-1\}$.}
\hrulealg
{$\cS_0 \gets \{x_0\}$}\\
{$\lambda_{p,0} \gets \begin{cases}
1, & \text{if} \ p = x_0\\
0, & \text{if} \ p \in\vertices(\cC) \setminus \{x_0\}
\end{cases}$}\\
{$\ell_0 \gets 0$ \Comment*[f]{$\ell_t:$ number of progress steps performed before iteration $t$}\label{line:l}}\\
\For{$t= 0,\ldots, T-1 $}{
    {$p_{t}^{FW} \in \argmin_{p \in \cC} \langle\nabla f(x_{t}), p- x_{t}\rangle$\label{line:p_fw}}\\
        {$p_{t}^{A} \in \argmax_{p \in \cS_t} \langle\nabla f(x_{t}), p- x_{t}\rangle$\label{line:p_a}}\\
        \uIf{$\langle\nabla f (x_t) ,p_t^{FW} - x_t\rangle \leq \langle \nabla f(x_t), x_t -p_t^A\rangle$\label{eq: which_direction}}{
            {$d_t \gets p_t^{FW} - x_t$; $\eta_{t, \max} \gets 1$ \label{eq:AFW_line_p_FW}}}
        \Else{
            {$d_t \gets x_t - p_t^A$; $\eta_{t, \max} \gets \frac{\lambda_{{p_t^A},t}}{1-\lambda_{{p_t^A},t}}$
            \label{eq:AFW_line_p_A}}
            }
        {$\gamma_t \gets \min\left\{\eta_{\ell_t}, \eta_{t, \max}\right\}$\label{line:gamma_t}}\\
        {$x_{t+1} \gets x_t + \gamma_t d_t$}\\
        \uIf{$\langle\nabla f (x_t) ,p_t^{FW} - x_t\rangle \leq \langle \nabla f(x_t), x_t -p_t^A\rangle$}{
            {$\lambda_{p,t+1} \gets
            \begin{cases}
            (1- \gamma_t) \lambda_{p, t} + \gamma_t, & \text{if} \ p=p_t^{FW}\\
            (1-\gamma_t) \lambda_{p,t}, & \text{if} \ p\in\vertices(\cC) \setminus\{p_t^{FW}\}
            \end{cases}$}
            }
        \Else{
            {$\lambda_{p,t+1} \gets
            \begin{cases}
            (1+ \gamma_t) \lambda_{p, t} - \gamma_t, & \text{if} \ p=p_t^{A}\\
            (1+\gamma_t) \lambda_{p,t}, & \text{if} \ p\in\vertices(\cC) \setminus\{p_t^{A}\}
            \end{cases}$}
            }
        {$\cS_{t+1} \gets \{p\in\vertices(\cC) \mid \lambda_{p, t+1} >0\}$}\\
        \uIf{$(\eta_{\ell_t} - \gamma_t) \langle \nabla f(x_t), p_t^A - p_t^{FW}\rangle \leq (\eta_{\ell_t}^2 - \gamma_t^2)L\delta^2 $\label{line:no_ds1}}{
            {$\ell_{t+1} \gets \ell_t + 1$\label{line:no_ds4} \Comment*[f]{progress step}}
            }
        \Else{\label{line:no_ds3}
            {$\ell_{t+1} \gets \ell_t$\Comment*[f]{non-progress step}\label{line:no_ds2}}
            }\label{line:no_ds5}
    }
\end{algorithm}
\begin{algorithm}[h]
  \caption{Away-step Frank-Wolfe algorithm (AFW) with line-search \citep{guelat1986some}}\label{algo:afw_ls_ss}
  {Identical to Algorithm~\ref{algo:afw}, except that Lines~\ref{line:l}, \ref{line:no_ds1}, \ref{line:no_ds4},  \ref{line:no_ds3}, \ref{line:no_ds2}, and \ref{line:no_ds5} have to be deleted and Line~\ref{line:gamma_t} has to be replaced by $\gamma_t \in \argmin_{\gamma\in [0, \eta_{t, \max}]}f(x_t + \gamma d_t)$.}
\end{algorithm}
In this section, we derive a version of the away-step Frank-Wolfe algorithm (AFW) \citep{guelat1986some, lacoste2015global} with step-size $\eta_t = \frac{4}{t+4}$ that admits a convergence rate of order up to $\cO(1/t^2)$ when optimizing a function satisfying \eqref{eq:heb} over a polytope.

\subsubsection{Algorithm overview} 
For better understanding, we first discuss AFW with line-search, which is presented in Algorithm~\ref{algo:afw_ls_ss}.
At iteration $t\in\{0,\ldots, T\}$, we can write
$x_t = \sum_{i=0}^{t-1} \lambda_{p_i,t} p_i$, where $p_i \in \vertices (\cC)$  and $\lambda_{p_i, t} \geq 0$ for all $i\in\{0,\ldots, t-1\}$ and $\sum_{i=0}^{t-1}\lambda_{p_i, t} = 1$.
We refer to $\cS_t := \{p_i \mid \lambda_{p_i, t} > 0\}$ as the active set at iteration $t$. Note that maintaining the active set can incur a significant memory overhead. However, with AFW, instead of being limited to taking a step in the direction of a vertex $p_t^{FW}\in \vertices (\cC)$ as in Line~\ref{line:p_t_det} of vanilla FW, we are also able to take an away step: Compute $p_t^{A} \in \argmax_{p\in \cS_t} \langle \nabla f (x_t), p - x_t \rangle$ and take a step away from vertex $p_t^{A}$, removing weight from vertex $p_t^{A}$ and adding it to all other vertices in the active set. Away steps facilitate the option of takin drop steps. A drop step occurs when a vertex gets removed from the active set. In case $x^*$ lies in the relative interior of an at least one-dimensional face $\cC^*$ of  $\cC$, drop steps allow AFW to get rid of bad vertices in the convex combination representing $x_t$, that is, vertices not in $\cC^*$. As soon as the optimal face is reached, that is, $x_t \in \cC^*$, the problem becomes that of having the optimal solution in the relative interior of $\cC^*$, for which FW with line-search admits linear convergence rates.

We next explain AFW with step-size $\eta_t = \frac{4}{t+4}$, presented in Algorithm~\ref{algo:afw}, which requires a slight modification of the version presented in \citet{lacoste2015global}. The main idea is to replace line-search with the open-loop step-size $\eta_t = \frac{4}{t+4}$. However, as we motivate in detail below, at iteration $t\in\{0,\ldots, T-1\}$, AFW's step-length is $\eta_{\ell_t}$, where $0 = \ell_0 \leq \ell_1 \leq \ldots \leq \ell_{T-1}\leq T-1$, that is, AFW may perform multiple steps of the same length.
Let $t\in\{0,\ldots, T-1\}$.
Note that for $d_t$ obtained from either Line~\eqref{eq:AFW_line_p_FW} or Line~\eqref{eq:AFW_line_p_A} in Algorithm~\ref{algo:afw}, it holds that $\langle \nabla f (x_t), d_t \rangle \leq \langle \nabla f(x_t),  p_t^{FW} -p_t^A  \rangle /2$.
By $L$-smoothness, 
\begin{align}\label{eq:contract_afw_poly_without_scaling}
    h_{t+1} & \leq h_t - \frac{\gamma_t  \langle \nabla f(x_t), p_t^A - p_t^{FW}\rangle}{2} +  \frac{\gamma_t^2 L\delta^2}{2}.
\end{align}
Working towards a convergence rate of order up to $\cO(1/t^2)$, we need to characterize a subsequence of steps for which an inequality of the form \eqref{eq:gotta_derive_this} holds. To do so, let
\begin{align*}
    g_t(\gamma) := - \frac{\gamma  \langle \nabla f(x_t), p_t^A - p_t^{FW}\rangle}{2} +  \frac{\gamma^2 L\delta^2}{2} \qquad \text{for} \ \gamma \in [0,1].
\end{align*}
We refer to all iterations $t\in\{0,\ldots, T-1\}$ such that $g_t(\gamma_t) \leq g_t(\eta_{\ell_t})$ as \emph{progress steps} and denote the number of progress steps performed before iteration $t\in\{0,\ldots, T\}$ by $\ell_t$, see Lines~\ref{line:l}, \ref{line:gamma_t}, and~\ref{line:no_ds1}-\ref{line:no_ds5} of Algorithm~\ref{algo:afw}. Thus, a progress step occurs during iteration $t$ if and only if the inequality in Line~\ref{line:no_ds1} is satisfied, which necessitates the computation of the smoothness constant $L$ of $f$ prior to the execution of the algorithm.
A non-drop step is always a progress step as $\gamma_t = \eta_{\ell_t}$ and the following lemma shows that drop steps which are non-progress steps do not increase the primal gap.
\begin{lemma}[Drop-step characterization]\label{lemma:contraction}
Let $g\colon [0,1] \to \R$ be defined via $g(\eta) := - \eta A + \eta^2 B$, where $A,B > 0$. For $t\in\N$, let $\eta_t = \frac{4}{t+4}$ and $\gamma_t \in [0, \eta_t]$. Then, $g(\gamma_t) \leq g(0)$ or $g(\gamma_t) \leq g(\eta_t)$.
\end{lemma}
\begin{proof}
By case distinction. Let $t\in\N$.
Case 1: $g(\eta_t) \leq g(0)$. By convexity,
$g(\gamma_t) = g(\lambda \eta_t + (1-\lambda) 0) \leq \lambda g(\eta_t) + (1-\lambda) g(0) \leq g(0) = 0$
where $\lambda \in [0,1]$.
Case 2: $g(\eta_t) > g(0)$. Then, $\eta_t > \eta^* \in \argmin_{\eta \in [0, \eta_t]} g(\eta)$, as $g$ is monotonously decreasing in the interval $[0, \eta^*]$. If $\eta^* \leq \gamma_t$, then $g(\gamma_t) \leq g(\eta_t)$ due to $g$ being monotonously increasing in $[\eta^*, \eta_t]$. If $\eta^* \geq \gamma_t$, then $g(\gamma_t) \leq g(0)$, as $g$ is monotonously decreasing in $[0, \eta^*]$. 
\end{proof}

Thus, a drop step is either a progress step and $h_{t+1} \leq h_t + g_t(\eta_{\ell_t})$, or $h_{t+1} \leq h_t$.

\begin{lemma}[Number of progress steps]\label{lemma:n_progress}
Let $\cC \subseteq \R^d$ be a compact convex set of diameter $\delta > 0$, let $f\colon \cC \to \R$ be a convex and $L$-smooth function. Let $T\in\N$ and $\eta_t = \frac{4}{t+4}$ for all $t\in\Z$.
Then, for all iterations $t\in\{0,\ldots, T\}$ of Algorithm~\ref{algo:afw} with step-size $\eta_t$, it holds that $\ell_t\geq \lceil t/2\rceil \geq t/2$.
\end{lemma}
\begin{proof}
Since all non-drop steps are progress steps and $\cS_t$, where $t\in\{0, \ldots, T\}$, has to contain at least one vertex of $\cC$, there cannot occur more drop steps than non-drop steps. Thus, $\ell_t\geq \lceil t/2\rceil \geq t/2$.
\end{proof}

\subsubsection{Convergence rate of order $\cO(1/t)$}
We first derive a baseline convergence rate of order $\cO(1/t)$ for AFW with step-size $\eta_t = \frac{4}{t+4}$.

\begin{proposition}[Convergence rate of order $\cO(1/t)$]\label{prop:baseline_afw}
Let $\cC \subseteq \R^d$ be a compact convex set of diameter $\delta > 0$, let $f\colon \cC \to \R$ be a convex and $L$-smooth function. Let $T\in\N$ and $\eta_t = \frac{4}{t+4}$ for all $t\in\Z$.
Then, for the iterates of Algorithm~\ref{algo:afw} with step-size $\eta_t$, it holds that
$h_t \leq \frac{ 16 L \delta^2}{t+6} = \eta_{t+2} 4 L \delta^2$
for all $t\in\{1,\ldots, T\}$.
\end{proposition}
\begin{proof}
Let $t\in\{0,\ldots, T-1\}$ and suppose that during iteration $t$, we perform a progress step.
Either $d_t = p_t^{FW}-x_t$, or $d_t = x_t - p_t^A$ and by Line \ref{eq: which_direction} of Algorithm~\ref{algo:afw},
$\langle \nabla f (x_t), x_t - p_t^A  \rangle \leq \langle \nabla f(x_t), p_t^{FW} - x_t \rangle$.
In either case, by $L$-smoothness,
\begin{align}\label{eq:basic_bf_bound_afw}
    h_{t+1} & \leq h_t - \gamma_{t} \langle \nabla f (x_t), x_t - p_t^{FW} \rangle + \frac{\gamma_{t}^2 L \delta^2}{2}  \leq (1 - \gamma_{t})h_t  + \frac{\gamma_{t}^2 L \delta^2}{2}.
\end{align}
By Lemma~\ref{lemma:contraction}, since non-progress steps do not increase the primal gap, we can limit our analysis to the subsequence of iterations corresponding to progress steps, $\{t^{(k)}\}_{k\in\{0,\ldots, \ell_T\}}$, for which, by \eqref{eq:basic_bf_bound_afw}, it holds that
\begin{align}\label{eq:necessary_for_acceleration}
    h_{t^{(k+1)}} & \leq (1 - \eta_{\ell_{t^{(k)}}}) h_{t^{(k)}} + \frac{\eta_{\ell_{t^{(k)}}}^2L\delta^2}{2} = (1 - \eta_k) h_{t^{(k)}} + \frac{\eta_k^2L\delta^2}{2}
\end{align}
for all $k\in\{0, \ldots, \ell_T-1\}$.
Since the first step is a non-drop step and thus a progress step, $h_{t^{(1)}} \leq h_1 \leq \frac{L\delta^2}{2}$. By similar arguments as in the proof of Proposition~\ref{prop:generalization_jaggi} starting with \eqref{eq:always_combine_with_this}, we obtain the bound $h_{t^{(k)}} \leq \frac{8L \delta^2}{k + 3}$ for all $k\in\{1,\ldots, \ell_T\}$. 
Since non-progress steps do not increase the primal gap and by Lemma~\ref{lemma:n_progress}, $h_t\leq h_{t^{(\ell_t)}}  \leq  \frac{8L \delta^2}{\ell_t+3} \leq  \frac{16L \delta^2}{t + 6} = \eta_{t+2} 4 L \delta^2$ for all $t\in\{1,\ldots, T\}$.
\end{proof}

\subsubsection{Convergence rate of order up to $\cO(1/t^2)$}
The introduction of away steps introduces another type of scaling inequality based on the \emph{pyramidal width}, a constant depending on the feasible region, see \citet{lacoste2015global} for more details.

\begin{lemma}[\citealp{lacoste2015global}]\label{lemma:away_step_scaling}
Let $\cC\subseteq \R^d$ be a polytope with pyramidal width $\omega > 0$ and let $f\colon \cC \to \R$ be a convex function with unique minimizer $x^*\in\argmin_{x\in\cC}f(x)$. Let $p^{FW} \in \argmin_{p\in \cC} \langle \nabla f (x), p \rangle$ and $p^A \in \argmax_{p\in \cS} \langle \nabla f(x),  p \rangle$ for some $\cS \subseteq \vertices (\cC)$ such that $x\in \conv(\cS)$. Then, it holds that
\begin{align}\tag{Scaling-A}\label{eq:scaling_a}
    \frac{\langle \nabla f(x), p^A - p^{FW}\rangle}{\omega} \geq \frac{\langle \nabla f (x), x - x^*\rangle}{\|x-x^*\|_2}.
\end{align}
\end{lemma}
For example, the pyramidal width of the unit cube in $\R^d$ satisfies $\omega\geq 2/\sqrt{d}$ \citep{lacoste2015global} and the pyramidal width of the $\ell_1$-ball in $\R^d$ satisfies $\omega \geq {1}/{\sqrt{d-1}}$ \citep{wirth2023approximate}.
Combining \eqref{eq:scaling_a} and \eqref{eq:scaling_heb} leads to a subsequence of primal gaps of the form \eqref{eq:gotta_derive_this} and a convergence rate of order up to $\cO(1/t^2)$ for Algorithm~\ref{algo:afw}.
\begin{theorem}[Convergence rate of order up to $\cO(1/t^2)$]\label{theorem:afw_polytope}
Let $\cC\subseteq \R^d$ be a polytope of diameter $\delta >0$ and pyramidal width $\omega >0$ and let $f\colon \cC \to \R$ be a convex and $L$-smooth function satisfying a $(\mu, \theta)$-\eqref{eq:heb} for some $\mu > 0 $ and $\theta \in [0, 1/2]$ with unique minimizer $x^*\in\argmin_{x\in\cC} f(x)$. Let $T\in\N$ and $\eta_t=\frac{4}{t+4}$ for all $t\in\Z$. Then, for the iterates of Algorithm~\ref{algo:afw} with step-size $\eta_t$, it holds that
\begin{align}\label{eq:to_derive_acc_afw}
    h_{t} & \leq \max \left\{ \eta_{\lceil t/2 -2 \rceil}^{1/(1-\theta)} \frac{L\delta^2}{2}, \left(\frac{\eta_{\lceil t/2 -2 \rceil} 2 \mu L \delta^2}{\omega}\right)^{1/(1-\theta)} + \eta_{\lceil t/2 -2 \rceil}^2 \frac{L\delta^2}{2}\right\}
\end{align}
for all $t\in\{1,\ldots, T\}$.
\end{theorem}
\begin{proof}
Let $t\in\{0,\ldots, T-1\}$. By \eqref{eq:contract_afw_poly_without_scaling}, \eqref{eq:scaling_a}, convexity of $f$, and \eqref{eq:scaling_heb}, it holds that
$h_{t+1} \leq h_t - \frac{\gamma_t \omega  \langle \nabla f(x_t), x_t-x^*\rangle}{2\|x_t-x^*\|_2} +  \frac{\gamma_t^2L\delta^2}{2} \leq h_t - \frac{\gamma_t  \omega}{2 \mu} h_t^{1-\theta} +  \frac{\gamma_t^2L\delta^2}{2}$. Thus, by Lemma~\ref{lemma:contraction}, non-progress steps satisfy $h_{t+1} \leq h_t$ and progress steps satisfy
\begin{align}\label{eq:contract_afw_actual_step_size}
    h_{t+1}\leq h_t - \frac{\eta_{\ell_t}  \omega}{2 \mu} h_t^{1-\theta} +  \frac{\eta_{\ell_t}^2L\delta^2}{2}.
\end{align}
Since non-progress steps do not increase the primal gap, we can limit our analysis to the subsequence of iterations corresponding to progress steps, $\{t^{(k)}\}_{k\in\{0,\ldots, \ell_T\}}$, for which, by \eqref{eq:contract_afw_actual_step_size}, it holds that
\begin{align*}
h_{t^{(k+1)}} \leq h_{t^{(k)}} - \frac{\eta_{\ell_{t^{(k)}}}\omega}{2\mu} h_{t^{(k)}}^{1-\theta} +  \frac{\eta_{\ell_{t^{(k)}}}^2L\delta^2}{2} = h_{t^{(k)}} - \frac{\eta_k\omega}{2\mu} h_{t^{(k)}}^{1-\theta} +  \frac{\eta_k^2L\delta^2}{2}.
\end{align*}
Combined with \eqref{eq:necessary_for_acceleration}, it thus holds that
\begin{align}\label{eq:afw_apply_sequence_lemma}
    h_{t^{(k+1)}} \leq (1 - \frac{\eta_k}{2})h_{t^{(k)}} - \frac{\eta_k\omega}{4\mu} h_{t^{(k)}}^{1-\theta} +  \frac{\eta_k^2L\delta^2}{2}.
\end{align}
for all $k \in \{1,\ldots, \ell_T-1\}$.
Since the first step is a non-drop step and thus a progress step, $h_{t^{(1)}} \leq h_1 \leq \frac{L\delta^2}{2}$.
Inequality~\ref{eq:afw_apply_sequence_lemma} allows us to apply Lemma~\ref{lemma:sequences} with $A = \frac{\omega}{4 \mu}$, $B = \frac{L\delta^2}{2}$, $C= 1$, $C_{t^{(k)}} = 1$ for all $k \in \{1,\ldots, \ell_T-1\}$, $\psi = \theta$, and $\fwt =1$, resulting in
$h_{t^{(k)}} \leq \max \left\{ \eta_{k-2}^{1/(1-\theta)} \frac{L\delta^2}{2}, \left(\frac{\eta_{k-2} 2 \mu L \delta^2}{\omega}\right)^{1/(1-\theta)} + \eta_{k-2}^2 \frac{L\delta^2}{2}\right\}
$
for all $k \in \{1,\ldots, \ell_T\}$, where we used that $\eta_{-1} \geq \eta_0  = 1$.
Since non-progress steps do not increase the primal gap and by Lemma~\ref{lemma:n_progress}, \eqref{eq:to_derive_acc_afw} holds for all $t\in\{1,\ldots, T\}$.
\end{proof}

\section{{Kernel herding}}\label{sec:kernel_herding}
In this section, we explain why FW with open-loop step-sizes converges at a rate of order $\cO(1/t^2)$ in the kernel-herding setting of \citet[Section~5.1 and Figure~3, right]{bach2012equivalence}.

\subsection{{Kernel herding and the Frank-Wolfe algorithm}}

Kernel herding is equivalent to solving a quadratic optimization problem in a \emph{reproducing kernel Hilbert space} (RKHS) with FW.
To describe this application of FW, we use the following notation:
Let $\cY\subseteq \R$ be an observation space, $\cH$ a RKHS with inner product $\langle \cdot, \cdot\rangle_\cH$, and $\Phi\colon \cY \to \cH$ the feature map associating a real function on $\cY$ to any element of $\cH$ via $x(y) = \langle x, \Phi(y) \rangle_\cH$ for  $x\in \cH$ and $y\in \cY$.
The positive-definite kernel associated with $\Phi$ is denoted by $k\colon (y,z) \mapsto k(y,z) = \langle \Phi(y), \Phi(z)\rangle_\cH$ for $y, z \in \cY$. In kernel herding, the feasible region is usually the \emph{marginal polytope} $\cC$, the convex hull of all functions $\Phi(y)$ for $y\in \cY$, that is, $\cC = \conv \left( \left\{\Phi(y) \mid y \in \cY\right\} \right)\subseteq \cH$. 
We consider a fixed probability distribution $p$ over $\cY$ and denote the associated mean element by
$\mu = \E_{p(y)}\Phi(y) \in \cC$,
where $\mu \in \cC$ follows from the fact that the support of $p$ is contained in $\cY$.
In \citet{bach2012equivalence}, kernel herding was shown to be equivalent to solving the following optimization problem with FW and step-size $\eta_t = \frac{1}{t+1}$:
\begin{equation}\tag{OPT-KH}\label{eq:kh}
    \min_{x\in \cC}  f(x),
\end{equation}
where $f(x):=\frac{1}{2}\|x - \mu\|_\cH^2$.
This equivalence led to the study of FW (variants) with other step-sizes to solve \eqref{eq:kh} \citep{chen2012super,lacoste2015sequential,tsuji2022pairwise}.
Under the assumption that $\|\Phi(y)\|_\cH = R$ for some constant $R > 0$ and all $y\in \cY$, the herding procedure is well-defined and all extreme points of $\cC$ are of the form $\Phi(y)$ for $y\in \cY$ \citep{bach2012equivalence}.
Thus, the linear minimization oracle (LMO) in FW always returns an element of the form $\Phi(y) \in \cC$ for $y\in \cY$.
Furthermore, FW constructs iterates of the form $x_t = \sum_{i=1}^t v_i \Phi(y_i)$, where $v = (v_1, \ldots, v_t)^\intercal$ is a weight vector, that is, $\sum_{i=1}^tv_i = 1$ and $v_i \geq 0$ for all $i \in \{1, \ldots, t\}$, and $x_t$ corresponds to an empirical distribution $\tilde{p}_t$ over $\cY$ with empirical mean
$\tilde{\mu}_t = \E_{\tilde{p}_t(y)}\Phi(y) = \sum_{i=1}^t v_i \Phi(y_i)= x_t \in \cC$.
Then, according to \citet{bach2012equivalence},
$\sup_{x\in \cH, \|x\|_\cH = 1}|\E_{p(y)}x(y) - \E_{\tilde{p}_t(y)}x(y)| = \|\mu - \tilde{\mu}_t\|_\cH$.
Thus, a bound on $\|\mu - \tilde{\mu}_t\|_\cH$ implies control on the error in computing the expectation for all $x\in \cH$ such that $\|x\|_\cH=1$.
In kernel herding, since the objective function is a quadratic, line-search and short-step are identical.

\subsection{{Explaining the phenomenon in} \citet{bach2012equivalence}}\label{sec:kernel_whaba}

We briefly recall the infinite-dimensional kernel-herding setting of \citet[Section~5.1 and Figure~3, right]{bach2012equivalence}, see also \citet[Section~2.1]{wahba1990spline}.
Let $\cY = [0,1]$ and
\begin{align}\label{eq:hs}
     \cH  = \{& x \colon {[0,1]} \to \R \mid x'(y) \in L^2({[0,1]}),  x(y)= \sum_{j = 1}^{\infty}(a_j \cos(2\pi j y) + b_j \sin(2\pi j y)),  a_j, b_j \in \R\}.
\end{align}
For $w, x\in \cH$, 
$\langle w, x\rangle_\cH:= \int_{[0,1]} w'(y)x'(y) dy$
defines an inner product and $(\cH, \langle \cdot, \cdot \rangle_\cH)$ is a Hilbert space.
Moreover, $\cH$ is also a RKHS and for $y,z\in [0,1]$, $\cH$ has the reproducing kernel
\begin{align}\label{eq:whaba_kernel}
 k(y,z) & =  \sum_{j = 1}^\infty \frac{2}{(2\pi j)^{2}}\cos(2 \pi j (y-z))  = \frac{1}{2}B_{2}(y-z-\lfloor y - z\rfloor)  = \frac{1}{2}B_{2}([y-z]),
 \tag{Bernoulli-kernel}
\end{align}
where for $y \in\R$, $[y] := y-\lfloor y \rfloor$, and
$B_2(y) = y^2-y + \frac{1}{6}$
is a \emph{Bernoulli polynomial}.
In the right plot of Figure~$3$ in \citet{bach2012equivalence}, kernel herding on $[0, 1]$ and Hilbert space $\cH$ is considered for the uniform density $p(y) := 1$ for all $y\in {[0,1]}$.
Then, for all $z\in [0, 1]$, we have
$\mu (z)  = \int_{[0,1]} k(z,y)p(y) dy  = \int_{[0,1]} \sum_{j = 1}^\infty \frac{2}{(2\pi j)^{2}}\cos(2 \pi j (z-y))\cdot 1 dy  = \sum_{j = 1}^\infty 0 = 0$,
where the integral and the sum can be interchanged due to the theorem of Fubini, see, for example, \citet{royden1988real}. For the remainder of this section, we assume that $p(y) = 1$ and, thus, $\mu(y) = 0$ for all $y\in{[0,1]}$. Thus, $f(x) = \frac{1}{2}\|x\|_\cH^2$. 
For this setting, \citet{bach2012equivalence} observed empirically that FW with open-loop step-size $\eta_t = \frac{1}{t+1}$ converges at a rate of order $\cO(1/t^2)$, whereas FW with line-search converges at a rate of order $\cO(1/t)$, see the reproduced plot in Figure~\ref{fig:kernel_herding_uniform}. The theorem below explains the accelerated convergence rate for FW with step-size $\eta_t = \frac{1}{t+1}$.

\begin{theorem}[Kernel herding]\label{thm:answering_bach}
Let $\cH$ be the Hilbert space defined in \eqref{eq:hs}, let $k \colon \R \times \R \to \cH$ be the kernel defined in \eqref{eq:whaba_kernel}, let $\Phi\colon[0,1] \to\cH$ be the feature map associated with $k$ restricted to $[0,1]\times [0,1]$, let $\cC=\conv(\{\Phi(y)\mid y\in[0,1]\})$ be the marginal polytope, and let $\mu = 0$ such that $f(x) = \frac{1}{2}\|x\|_\cH^2$.
Let $T\in\N$ and $\eta_t = \frac{1}{t+1}$ for all $t\in\Z$.
Then, for the iterates of Algorithm~\ref{algo:fw} with step-size $\eta_t$ and the LMO satisfying Assumption~\ref{ass:argmin} (a tie-breaking rule), it holds that $f(x_t)  = 1/(24 t^2)$ for all $t \in\{1,\ldots, T\}$ such that $t=2^m$ for some $m\in\N$.
\end{theorem}

We first provide a proof sketch for Theorem~\ref{thm:answering_bach} and subsequently prove the theorem in detail.

\begin{proof}[Sketch of proof for Theorem~\ref{thm:answering_bach}]
The main idea behind the proof is that FW with $\eta_t = \frac{1}{t+1}$ leads to iterates $x_t = \frac{1}{t}\sum_{i = 1}^t \Phi(y_i)$ with $\{y_1, \ldots, y_t\} = \{\frac{i-1}{t} \mid i = 1, \ldots, t\}$ for all $t = 2^m$, where $m \in \N$. Then, the proof follows by a series of calculations. We make several introductory observations. Note that Line~\ref{line:p_t_det} of Algorithm~\ref{algo:fw} becomes
$p_t \in \argmin_{p\in \cC} Df(x_t) (p - x_t) = \argmin_{p\in \cC} Df(x_t)(p)$,
where, for $w, x\in \cH$, $D f(w)(x) = \langle w,x \rangle_\cH$ denotes the first derivative of $f$ at $w$.
For $x\in \cC$ and $x_t\in \cC$ of the form $x_t = \frac{1}{t}\sum_{i=1}^{t} \Phi(y_i)$ for $y_1,\ldots, y_t\in {[0,1]}$, it holds that
$Df(x_t)(x) = \langle \frac{1}{t}\sum_{i=1}^{t} \Phi(y_i), x\rangle_\cH$.
Then, for $y\in [0,1]$, let
\begin{align}\label{eq:def_g_t}
    g_t(y) := \langle \frac{1}{t}\sum_{i=1}^t \Phi(y_i), \Phi(y)\rangle_\cH
    =\frac{1}{t}\sum_{i=1}^t k(y_i, y).
\end{align}
Since the LMO of FW always returns a vertex of $\cC$ of the form $\Phi(y)$ for $y\in [0, 1]$ \citep{bach2012equivalence}, it holds that
$\min_{p\in \cC} Df(x_t)(p) = \min_{y\in {[0,1]}} g_t(y)$
and the vertex returned by the LMO during iteration $t$ is contained in the set
$\{\Phi(z) \mid z \in \argmin_{y\in {[0,1]}} g_t(y) \}$.
Thus, instead of considering the LMO directly over $\cC$, we can perform the computations over $[0,1]$.
To simplify the proof, we make the following assumption on the $\argmin$ operation in the LMO of FW, a tie-breaking rule in case $|\argmin_{p\in \cC}Df(x_t)(p)| \geq 2$.
\begin{assumption}\label{ass:argmin}
The LMO of FW always returns $p_t \in \argmin_{p\in \cC}Df(x_t)(p)$ such that $p_t = \Phi(z)$ for $ z = \min (\argmin_{y\in {[0,1]}} g_t(y))$.
\end{assumption}
Recall that FW starts at iterate $x_0$, but since $\eta_0 = 1$, it holds that $x_1 = \Phi(y_1)$. As we will prove in Lemma~\ref{lemma:second}, without loss of generality, we can assume that FW starts at iterate $x_1 = \Phi(y_1)$, where $y_1 = 0$.
\end{proof}

To rigorously prove Theorem~\ref{thm:answering_bach}, we require the following four technical lemmas. In the lemma below, we prove several technical properties of kernel $k$ as in \eqref{eq:whaba_kernel}.

\begin{lemma}\label{lemma:cos_is_symmetric}
    Let $\cH$ be the Hilbert space defined in \eqref{eq:hs} and let $k \colon \R \times \R \to \cH$ be the kernel defined in \eqref{eq:whaba_kernel}.
    For $y, z \in [0, 1]$ and $n\in\Z$, it holds that
    $k(y,z) = k(z,y) = k(|y-z|,0) = \frac{1}{2}B_2(|y-z|)$ and $k(y,z) = k(y, z+n)$.
\end{lemma}
\begin{proof}
We first prove that for $y, z \in [0, 1]$, it holds that $k(y,z) = k(z,y)$.
    Let $a\in[0, 1[$. Then,
    \begin{align}\label{eq:squarea}
        [a] &= a,  &  [-a]  &= 1 - a,  &  B_2([a])  & = a^2 - a + \frac{1}{6} = (1-a)^2 - (1-a) + \frac{1}{6} = B_2[-a],\\
        [1]  &=  0, &  [-1]  &=  0, &  B_2([1])  & = B_2([-1]).\label{eq:square1}
    \end{align}
    By \eqref{eq:squarea} and \eqref{eq:square1}, for any $y, z \in [0, 1]$, it holds that $k(y,z) = \frac{1}{2}B_2([y-z]) = \frac{1}{2}B_2([z-y]) = k(z,y)$.
    
Next, we prove that for $y, z \in [0, 1]$, it holds that $k(y,z) = k(|y-z|,0) = \frac{1}{2}B_2(|y-z|)$.
Let $y,z\in[0,1]$ such that $|y-z|=a\in[0,1[$.
Then, by \eqref{eq:squarea}, $k(y,z) = \frac{1}{2}B_2([y-z]) = \frac{1}{2}B_2([|y-z|]) = \frac{1}{2}B_2(|y-z|)$.
Furthermore,
$k(y,z) = \frac{1}{2}B_2([y-z]) = \frac{1}{2}B_2([|y-z|]) = k(|y-z|,0)$.
Next, let $y,z\in[0,1]$ such that $|y-z|=1$.
Then, by \eqref{eq:square1},
$k(y,z) = \frac{1}{2}B_2([y-z]) = \frac{1}{2}B_2([|y-z|]) =  \frac{1}{2}B_2([1]) = \frac{1}{12} = \frac{1}{2}\left(1^2 - 1 +\frac{1}{6}\right) = \frac{1}{2}B_2(1) = \frac{1}{2}B_2(|y-z|)$.
Furthermore,
$k(y,z) = \frac{1}{2}B_2([y-z]) = \frac{1}{2}B_2([|y-z|]) =  \frac{1}{2}B_2([1]) = k(|y-z|, 0)$.
    
Finally, we prove that for $y,z\in [0,1]$ and $n\in\Z$, it holds that $k(y,z)=k(y,z+n)$.
Indeed,
$k(y,z) = \frac{1}{2}B_2(y-z- \lfloor y-z\rfloor)= \frac{1}{2}B_2(y-z-n - \lfloor y-z-n\rfloor)= k(y,z+n)$.
\end{proof}
In the two lemmas below, we characterize $\argmin_{y\in [0,1]}g_t(y)$, where $g_t$ is defined as in \eqref{eq:def_g_t}.
\begin{lemma}\label{lemma:first}
Let $\cH$ be the Hilbert space defined in \eqref{eq:hs}, let $k \colon \R \times \R \to \cH$ be the kernel defined in \eqref{eq:whaba_kernel}, let $\Phi\colon[0,1] \to\cH$ be the feature map associated with $k$ restricted to $[0,1]\times [0,1]$,
let $t\in \N$, let $\{y_1 , \ldots, y_t\} = \{\frac{i-1}{t} \mid i \in \{1, \ldots, t\}\}$, and let $g_t$ be defined as in \eqref{eq:def_g_t}, that is, $g_t (y) = \frac{1}{t}\sum_{i=1}^tk(y_i,y)$. Then, it holds that $\argmin_{y\in{[0,1]}} g_t(y) = \{ y_i + \frac{1}{2t} \mid i \in \{1, \ldots, t\} \}$.
\end{lemma}
\begin{proof}
Let $t\in \N$ and $\{y_1, \ldots, y_t\} = \{\frac{i - 1}{t} \mid i \in \{1, \ldots, t\}\}$. We stress that this does not imply that for all $i\in\{1, \ldots, t\}$, $y_i = \frac{i-1}{t}$.
By Lemma~\ref{lemma:cos_is_symmetric}, for all $y\in [0, 1]$, it holds that
$g_t (y)  = \langle \frac{1}{t} \sum_{i = 1}^t \Phi(y_i), \Phi(y) \rangle_\cH =\frac{1}{t}\sum_{i=1}^t k(y_i, y) = \frac{1}{2t}\sum_{i = 1}^t(|y_i -y |^2 - |y_i -y| + \frac{1}{6})$.
Then, for $y \in [0, 1] \setminus \{y_1, \ldots, y_t\}$, it holds that
$g_t'(y) = \frac{1}{2t}\sum_{i=1}^t (2 (y - y_i) - \frac{y - y_i}{|y - y_i|})$
and since $\sum_{i=1}^{t}{y_i}= (t-1)/2$, we have
\begin{align*}
    g_t'(y) = \frac{1}{2}(2y - \frac{t-1}{t} - \frac{1}{t} \sabs{\{y_i < y \colon i \in \{1, \ldots, t\}\}} + \frac{1}{t}\sabs{\{y_i > y \colon i \in \{1, \ldots,t\}\}}).
\end{align*}
For $y\in \left]\frac{i-1}{t}, \frac{i}{t}\right[$, where $i \in \{1, \ldots, t\}$, it holds that
$g'_t (y)  = \frac{1}{2}(2y -  \frac{t-1}{t} - \frac{i}{t}  + \frac{t - i}{t}) = \frac{1}{2}(2y +\frac{1}{t} - \frac{2i}{t})$
and 
$g_t'(y) = 0$ if and only if
$y = \frac{i - \frac{1}{2}}{t}$.
Since $g_t$ is strongly convex on $]\frac{i-1}{t}, \frac{i}{t}[$ for $i \in \{1, \ldots, t \}$ and continuous on $[0, 1]$, it holds that $y_i = \frac{i-1}{t}$ cannot be a minimizer of $g_t$ on $[0,1]$ for any $i\in\{1, \ldots, t\}$. Since $g_t(0) =g_t(1)$ by Lemma~\ref{lemma:cos_is_symmetric}, $1$ cannot be a minimizer either.
Thus, only elements in $\{ y_i + \frac{1}{2t} \mid i \in \{1, \ldots, t\} \}$ can be minimizers of $g_t$ on $[0,1]$.
By Lemma \ref{lemma:cos_is_symmetric}, 
\begin{align*}
    \sum_{i = 1}^tk(\frac{i-1}{t}, \frac{j-1}{t}+ \frac{1}{2t}) - \sum_{i = 1}^t k(\frac{i-1}{t}, \frac{j}{t} + \frac{1}{2t}) & = \sum_{i = 1}^tk(\frac{i}{t}, \frac{j}{t}+ \frac{1}{2t}) - \sum_{i = 1}^t k(\frac{i-1}{t}, \frac{j}{t} + \frac{1}{2t}) \\
    & = k(\frac{t}{t}, \frac{j}{t}+ \frac{1}{2t}) -  k(\frac{0}{t}, \frac{j}{t}+ \frac{1}{2t})\\
    & = 0
\end{align*}
for all $j\in\{1, \ldots, t-1\}$.
Thus, $g_t(\frac{j-1}{t}+\frac{1}{2t}) =g_t(\frac{j}{t}+\frac{1}{2t})$ for all $j \in \{1, \ldots, t-1\}$. Thus, $g_t(\frac{i-1}{t} + \frac{1}{2t}) = g_t(\frac{j-1}{t} + \frac{1}{2t})$ for all $i,j \in \{1, \ldots, t\}$, proving the lemma.
\end{proof}

\begin{lemma}\label{lemma:second}
Let $\cH$ be the Hilbert space defined in \eqref{eq:hs}, let $k \colon \R \times \R \to \cH$ be the kernel defined in \eqref{eq:whaba_kernel}, let $\Phi\colon[0,1] \to\cH$ be the feature map associated with $k$ restricted to $[0,1]\times [0,1]$,
let $t\in \N$, let $y_1, \ldots, y_t \in [0, 1]$, and let $g_t$ be defined as in \eqref{eq:def_g_t}, that is,
$g_t (y) = \frac{1}{t}\sum_{i=1}^tk(y_i,y)$.
Suppose that $\argmin_{y\in {[0,1]}} g_t(y) = \{z_1, \ldots, z_k\}\subseteq [0, 1]$ for some $k\in \N$.
Let $c\in \R $, let $\tilde{y}_i = [y_i + c]$ for all $i \in \{1, \ldots, t\}$, and let
$\tilde{g}_t (y) = \frac{1}{t}\sum_{i=1}^tk(\tilde{y}_i,y)$.
Then, $\argmin_{z\in {[0,1]}} \tilde{g}_{t}(z) = \{[z_1 + c], \ldots, [z_k + c]\}$.
\end{lemma}
\begin{proof}
It holds that
\begin{align*}
    \argmin_{z\in[0,1]}\tilde{g}_t(z) & = \argmin_{z=[y+c], y\in\R} \tilde{g}_t (z)\\
    & =\argmin_{z=[y+c], y\in\R}\frac{1}{2t}\sum_{i=1}^t B_2([[y_i+c]-[y+c]])\\
    & =\argmin_{z=[y+c], y\in\R}\frac{1}{2t}\sum_{i=1}^t B_2([y_i+c - \lfloor y_i + c\rfloor - (y+c) -  (-\lfloor y + c\rfloor)])\\
    & =\argmin_{z=[y+c], y\in\R}\frac{1}{2t}\sum_{i=1}^t B_2([y_i - y - \lfloor y_i + c\rfloor +\lfloor y + c\rfloor])\\
    & =\argmin_{z=[y+c], y\in\R}\frac{1}{2t}\sum_{i=1}^t B_2([y_i - y ])\\
    & = \{[z_1 + c], \ldots, [z_k + c]\},
\end{align*}
where the second-to-last equality is due to Lemma~\ref{lemma:cos_is_symmetric}.
\end{proof}

In the lemma below, we leverage the previous lemmas to prove that FW with step-size $\eta_t = \frac{1}{t+1}$ leads to iterates $x_t = \frac{1}{t}\sum_{i = 1}^t \Phi(y_i)$ with $\{y_1, \ldots, y_t\} = \{\frac{i-1}{t} \mid i = 1, \ldots, t\}$ for all $t = 2^m$, where $m \in \N$.

\begin{lemma}\label{lemma:third}
Let $\cH$ be the Hilbert space defined in \eqref{eq:hs}, let $k \colon \R \times \R \to \cH$ be the kernel defined in \eqref{eq:whaba_kernel}, let $\Phi\colon[0,1] \to\cH$ be the feature map associated with $k$ restricted to $[0,1]\times [0,1]$, let $\cC=\conv(\{\Phi(y)\mid y\in[0,1]\})$ be the marginal polytope, and let $\mu = 0$ such that $f(x) = \frac{1}{2}\|x\|_\cH^2$.
Let $T\in\N$ and $\eta_t = \frac{1}{t+1}$ for all $t\in\Z$.
Then, for the iterates of Algorithm~\ref{algo:fw} with step-size $\eta_t$ and the LMO satisfying Assumption~\ref{ass:argmin} it holds that $x_t = \frac{1}{t}\sum_{i = 1}^t \Phi (y_i)$ with $\{y_1, \ldots, y_t\} = \{ \frac{i -1}{t} \mid i \in \{1,\ldots,t\}\}$ for all $t \in\{1,\ldots, T\}$ such that $t=2^m$ for some $m\in\N$,.
\end{lemma}
\begin{proof}
Since $\eta_0 = 1$, it holds that $x_1 = \Phi(y_1)$.
By Lemma~\ref{lemma:second}, without loss of generality, we can assume that FW starts with iterate $x_1 = \Phi(y_1)$, where $y_1 = 0$.
Let $t\in\{1,\ldots, T\}$. Since we use the step-size $\eta_t = \frac{1}{t+1}$, we obtain uniform weights, that is, $x_t = \frac{1}{t}\sum_{i=1}^t\Phi(y_i)$, where $y_i \in [0,1]$ for all $i \in \{ 1,\ldots, t\}$. 
Suppose that $t=2^m$ for some $m\in \N$. The proof that it holds that $\{y_1, \ldots, y_t\} = \{ \frac{i -1}{t} \mid i \in \{1,\ldots,t\}\}$ is by induction on $m\in \N$.
The base case, $m = 0$, follows from $x_1 = \Phi(y_1)$, where $y_1 = 0$.
Suppose that for $t=2^m$ for some $m\in\N$, it holds that $\{y_1, \ldots, y_t\} = \{ \frac{i -1}{t} \mid i \in \{1,\ldots,t\}\}$. If we show that
\begin{align}\label{eq:kh_to_prove}
    \{y_1, \ldots, y_{2t}\} = \{ \frac{i -1}{2t} \mid i \in \{1,\ldots,2t\}\},
\end{align}
the statement of the lemma follows from induction. \eqref{eq:kh_to_prove} is subsumed by the stronger statement that $y_{t+j} = y_j + \frac{1}{2t}$ for all $j \in\{ 1, \ldots, t\}$, and we prove the latter for the remainder of this proof.
By Lemma~\ref{lemma:first} and Assumption~\ref{ass:argmin}, it holds that $y_{t +1} = \frac{1}{2t}$. Suppose that for some $\ell \in\{1, \ldots, t-1\}$, it holds that $y_{t + j} = y_{j} + \frac{1}{2t}$ for all $j\in \{1, \ldots, \ell\}$.
We decompose the function $g_{t+\ell}(y)$ into $g_t(y)$ and
$\tilde{g}_{\ell}(y) = \langle \frac{1}{\ell} \sum_{i=1}^\ell \Phi(y_i + \frac{1}{2t}), \Phi(y)  \rangle_\cH$,
that is, we consider the decomposition
$g_{t+\ell}(y) = \frac{t}{t+\ell} g_t(y) + \frac{\ell}{t+\ell}\tilde{g}_{\ell}(y)$.
By Lemma~\ref{lemma:first}, 
$\argmin_{y\in[0,1]} g_{t}(y) = \left\{y_i + \frac{1}{2t} \mid i \in \{1, \ldots, t\}\right\}\subseteq [0,1]$
and by Assumption~\ref{ass:argmin}, $y_{\ell+1} = \min(\argmin_{y\in [0,1]}g_\ell(y))$. Thus,
by Lemma~\ref{lemma:second}, it holds that
$\min\argmin_{y\in [0,1]}\tilde{g}_\ell(y) = \min(\argmin_{y\in [0,1]}g_\ell(y) + \frac{1}{2t}) = y_{\ell+1} + \frac{1}{2t}\in \{y_i + \frac{1}{2t} \mid i \in \{1, \ldots, t\}\}$.
Thus, 
$\min\argmin_{y\in [0,1]}\tilde{g}_\ell(y) \in \argmin_{y\in[0,1]} g_{t}(y)$
and
\begin{align*}
    y_{t+\ell+1} = \min\argmin_{y\in [0,1]}g_{t+\ell}(y) = \min\argmin_{y\in [0,1]} \tilde{g}_\ell(y) =  y_{\ell+1} + \frac{1}{2t}.
\end{align*}
By induction, $y_{t+j} = y_j + \frac{1}{2t}$ for all $j \in\{ 1, \ldots, t\}$,
as required to conclude the proof.
\end{proof}

Finally, we prove Theorem~\ref{thm:answering_bach}.

\begin{proof}[Proof of Theorem~\ref{thm:answering_bach}]
By Lemma~\ref{lemma:third}, $x_t = \frac{1}{t}\sum_{i=1}^t \Phi(\frac{i - 1}{t})$ and, since $\mu = 0$, we have
$f(x_t) = \frac{1}{2}\|x_t\|_\cH^2 = \frac{1}{2t^2} \sum_{j=1}^{t} \sum_{i = 1}^{t} k(\frac{i-1}{t}, \frac{j-1}{t}) = \frac{1}{2t} \sum_{i=1}^{t}  k(\frac{i-1}{t}, 1)$,
where the last equality follows from repeatedly applying
\begin{align}\label{eq:proof_in_kh}
    \sum_{i=1}^tk(\frac{i-1}{t}, \frac{j-1}{t}) & = \sum_{i=1}^tk(\frac{i-1}{t}, \frac{j}{t}),
\end{align}
where $j\in\{1,\ldots,t\}$.
To see that \eqref{eq:proof_in_kh} holds, recall that by Lemma \ref{lemma:cos_is_symmetric}, it holds that
\begin{align*}
\sum_{i=1}^tk(\frac{i-1}{t}, \frac{j-1}{t}) - \sum_{i=1}^tk(\frac{i-1}{t},\frac{j}{t})= \sum_{i=1}^tk(\frac{i}{t}, \frac{j}{t}) - \sum_{i=1}^tk(\frac{i-1}{t},\frac{j}{t}) = k(1, \frac{j}{t}) - k(0, \frac{j}{t})=0
\end{align*}
for all $j\in\{1, \ldots, t\}$.
Thus,
$f(x_t)  = \frac{1}{2t} \sum_{i = 1}^{t} k(\frac{i-1}{t}, 1) = \frac{1}{2t} \sum_{i = 1}^{t} k(\frac{i-1}{t}, 0) = \frac{1}{2t} \sum_{i = 1}^{t} k(\frac{i}{t},0) = \frac{1}{4t} \sum_{i = 1}^{t} ((\frac{i}{t})^2 - \frac{i}{t} + \frac{1}{6})$,
where the second, third, and fourth equalities are due to Lemma~\ref{lemma:cos_is_symmetric}.
Since $\sum_{i = 1}^t i= \frac{t (t +1)}{2}$ and $\sum_{i=1}^t i^2 = \frac{2t^3 + 3t^2+t}{6}$, it holds that
$f(x_t) = \frac{1}{4t} (\frac{2t + 3+\frac{1}{t}}{6} - \frac{t+1}{2} + \frac{t}{6} ) = \frac{1}{24t^2}$.
\end{proof}
The proof of Theorem~\ref{thm:answering_bach} implies that the iterates of FW with open-loop step-size $\eta_t = \frac{1}{t+1}$ are identical to the Sobol sequence at any iteration $t=2^m$, where $m\in \N$. The Sobol sequence is known to converge at the optimal rate of order $\cO(1/t^2)$ \citep{bach2012equivalence} in this infinite-dimensional kernel-herding setting.
Here, the equivalence of FW with kernel herding leads to the study and discovery of new convergence rates for FW. This is in contrast to other papers \citep{chen2012super, bach2012equivalence, tsuji2022pairwise} in which FW is exploited to improve kernel-herding methods.
\begin{figure}[t]
\captionsetup[subfigure]{justification=centering}
\centering
\begin{tabular}{c c}
     \begin{subfigure}{.3\textwidth}
    \centering
        \includegraphics[width=1\textwidth]{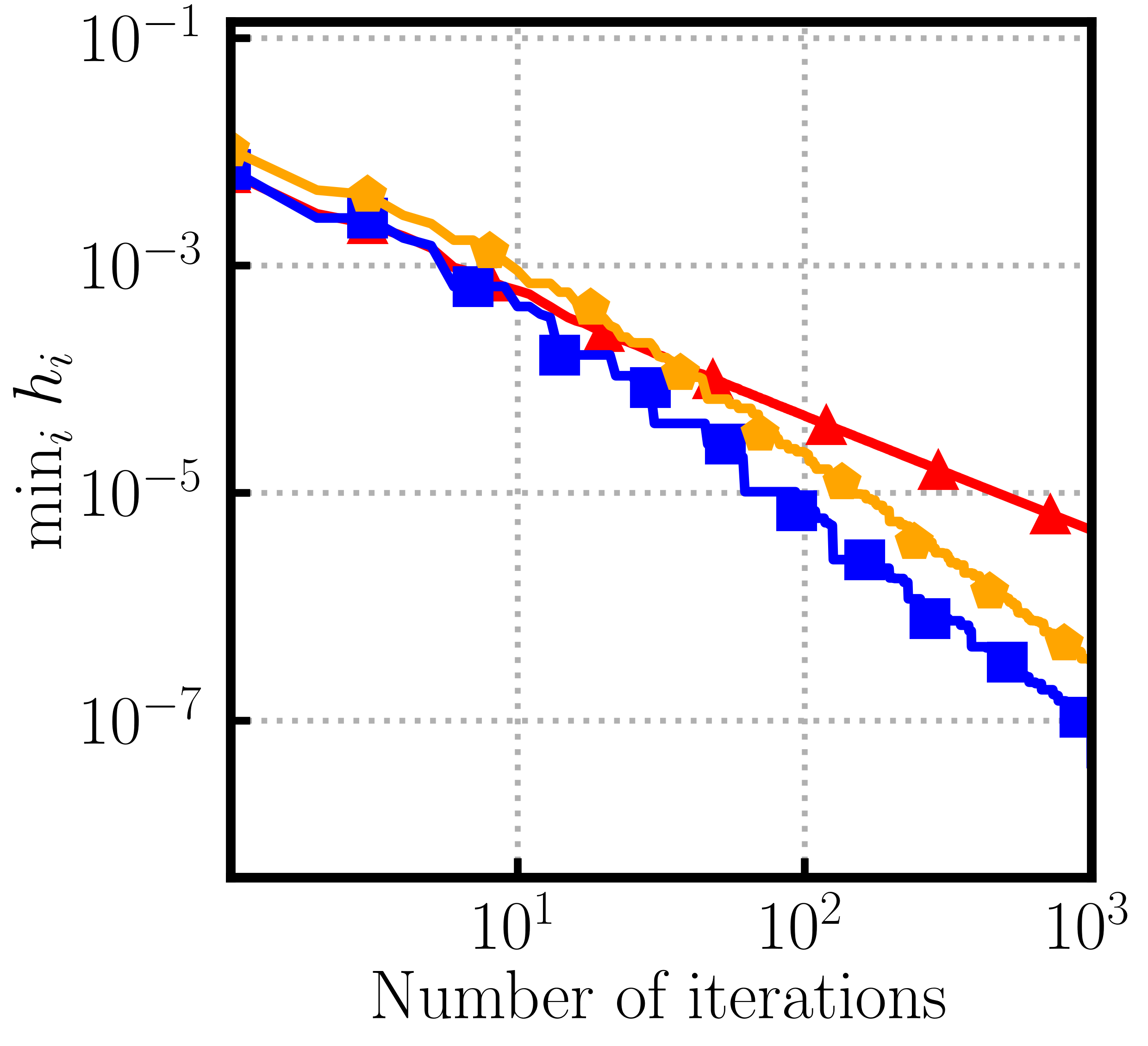}
        \subcaption{Uniform density.}
        \label{fig:kernel_herding_uniform}
    \end{subfigure}& 
     \begin{subfigure}{.3\textwidth}
    \centering
        \includegraphics[width=1\textwidth]{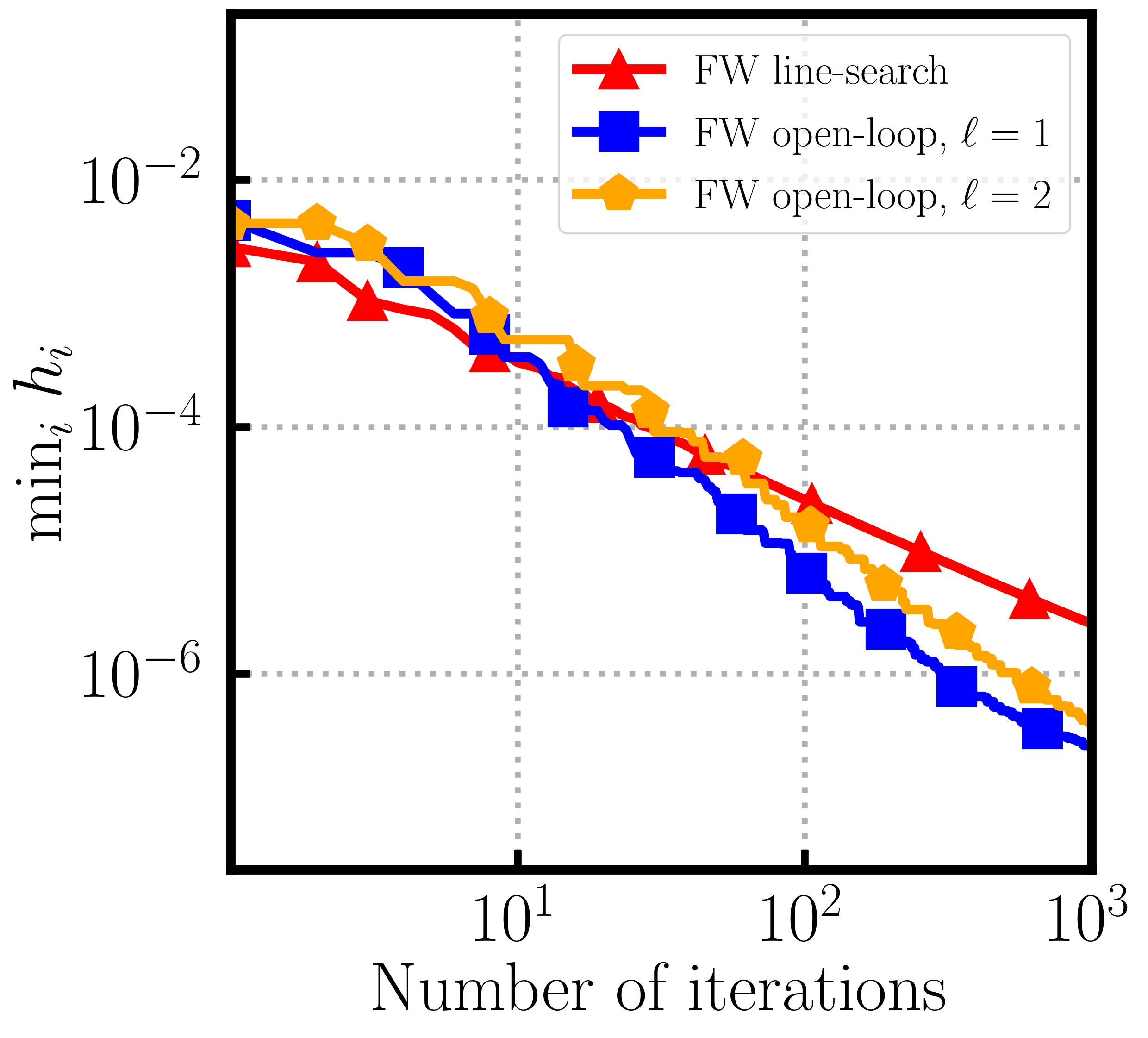}
        \subcaption{Non-uniform density.}
        \label{fig:kernel_herding_non_uniform}
    \end{subfigure}
\end{tabular}
\caption{
Comparison of FW with different step-sizes for the kernel-herding problem \eqref{eq:kh} as specified in Section~\ref{sec:kernel_herding} for RKHS $\cH$ as in \eqref{eq:hs}, kernel $k$ as in \eqref{eq:whaba_kernel}, and both uniform and non-uniform densities.
The $y$-axis represents the minimum primal gap. 
In both settings, FW with open-loop step-sizes converges at a rate of order $\cO(1/t^2)$ whereas FW with line-search converges at a rate of order $\cO(1/t)$.
}\label{fig:kernel_herding}
\end{figure}

The results in Figure~\ref{fig:kernel_herding}, see Section~\ref{sec:experiment_kernel_herding} for details, show that
in the kernel-herding setting of Section~\ref{sec:kernel_whaba}, for RKHS $\cH$ as in \eqref{eq:hs}, kernel $k$ as in \eqref{eq:whaba_kernel}, and both uniform and non-uniform densities over $ \cY = [0, 1]$, FW with open-loop step-sizes $\eta_t = \frac{\ell}{t+\ell}$, where $\ell\in\N_{\geq 1}$, converges at a rate of order $\cO(1/t^2)$ and FW with line-search converges at a rate of order $\cO(1/t)$. It remains an open problem to extend Theorem~\ref{thm:answering_bach} to non-uniform densities.
\section{{Numerical experiments}}\label{sec:numerical_experiments_main}
In this section, we present the numerical experiments. Numerical experiments corroborating our results in Sections~\ref{sec:blueprint}, \ref{sec:unconstrained}, and~\ref{sec:fw_variants} are omitted since the studies do not provide new insights or highlight unexplained convergence rates. All of our numerical experiments are implemented in \textsc{Python} and performed on an Nvidia GeForce RTX 3080 GPU with 10GB RAM and an Intel Core i7 11700K 8x CPU at 3.60GHz with 64 GB RAM. Our code is publicly available on 
\href{https://github.com/ZIB-IOL/open_loop_fw}{GitHub}. 
For all numerical experiments, to avoid the oscillating behavior of the primal gap, the $y$-axis represents $\min_{i\in\{1, \ldots, t\}} h_i$, where $t$ denotes the number of iterations and $h_i$ the primal gap.
\subsection{Detailed setups for the numerical experiments in Figures~\ref{fig:exterior}, \ref{fig:experiments_polytope}, and~\ref{fig:kernel_herding}}
Throughout the paper, we present several toy examples in Figures~\ref{fig:exterior}, \ref{fig:experiments_polytope}, and~\ref{fig:kernel_herding} to illustrate results and raise open questions. For completeness, we present the detailed setups for these experiments below.
\subsubsection{Detailed setup for numerical experiments in Figure~\ref{fig:exterior}}\label{sec:experiment_exterior}
For $d = 100$, we address \eqref{eq:opt} with FW for $\cC\subseteq \R^{d}$ the $\ell_p$-ball, $f(x) = \frac{1}{2}\|Ax-b\|_2^2$, where $A\subseteq \R^{100\times 100}$ and $b\in \R^{100}$ are a random matrix and vector, respectively, such that $f$ is not strongly convex, the unconstrained optimal solution $\argmin_{x\in \R^d}f(x)$ lies in the exterior of the feasible region and, thus, $\|\nabla f(x)\|_2 \geq \lambda > 0$ for all $x\in\cC$ and some $\lambda > 0$.
For $p\in\{2,3,5\}$, we compare FW with open-loop step-sizes $\eta_t = \frac{\ell}{t+\ell}$, where $\ell\in\{1, 2, 4, 6\}$, and the constant step-size introduced in Remark~\ref{rem:ol_linear}, starting with $x_0 = e^{(1)}$. We plot the results of the experiments in log-log plots in Figure~\ref{fig:exterior}.
\subsubsection{Detailed setup for numerical experiments in Figure~\ref{fig:experiments_polytope}}\label{sec:experiment_polytope}
For $d = 100$, we address \eqref{eq:opt} with FW for $\cC\subseteq \R^{d}$ the probability simplex and $f(x) = \frac{1}{2}\|x-\rho \bar{\oneterm}\|_2^2$, where $\rho \geq \frac{2}{d}$ and $\bar{\oneterm}$ is the vector with zeros for the first $\lceil d/2\rceil$ entries and ones for the remaining entries.
Then, $\frac{2}{d}\bar{\oneterm} = x^*\in\argmin_{x\in \cC} f(x)$ is the unique minimizer of $f$.
For $\rho \in \{\frac{1}{4}, 2\}$, we compare FW with line-search and open-loop step-sizes $\eta_t = \frac{\ell}{t+\ell}$, where $\ell\in\{1, 2, 4\}$, starting with $x_0 = e^{(1)}$. Here, short-step is identical to line-search and, thus, omitted. We plot the results of the experiments in log-log plots in Figure~\ref{fig:experiments_polytope}. 
\subsubsection{Detailed setup for numerical experiments in Figure~\ref{fig:kernel_herding}}\label{sec:experiment_kernel_herding}
We consider the kernel-herding setting of Section~\ref{sec:kernel_whaba} over $ \cY = [0, 1]$, that is, $\cH$ is the RKHS as in \eqref{eq:hs} and $k$ is the kernel as in \eqref{eq:whaba_kernel}. Given either the uniform density or a random non-uniform density of the form
$
    p(y) \backsim \left(\sum_{i=1}^n(a_i \cos(2  \pi i y) + b_i \sin (2 \pi i y))\right)^2
$
with $n\leq 5$ and $a_i,b_i\in \R$ for all $i\in\{1,\ldots, n\}$ such that $\int_{[0,1]} p(y) dy = 1$, we address \eqref{eq:kh} with FW with line-search and open-loop step-sizes $\eta_t = \frac{\ell}{t+\ell}$, where $\ell \in \{1, 2\}$. The LMO is implemented as an exhaustive search over $[0,1]$ and run for 1,000 iterations. We plot the results of the experiments in log-log plots in Figure~\ref{fig:kernel_herding}. 

\subsection{Logistic regression}\label{sec:logistic_regression}
\begin{figure}[t]
\captionsetup[subfigure]{justification=centering}
\begin{tabular}{c c c}
    \begin{subfigure}{.3\textwidth}
    \centering
        \includegraphics[width=1\textwidth]{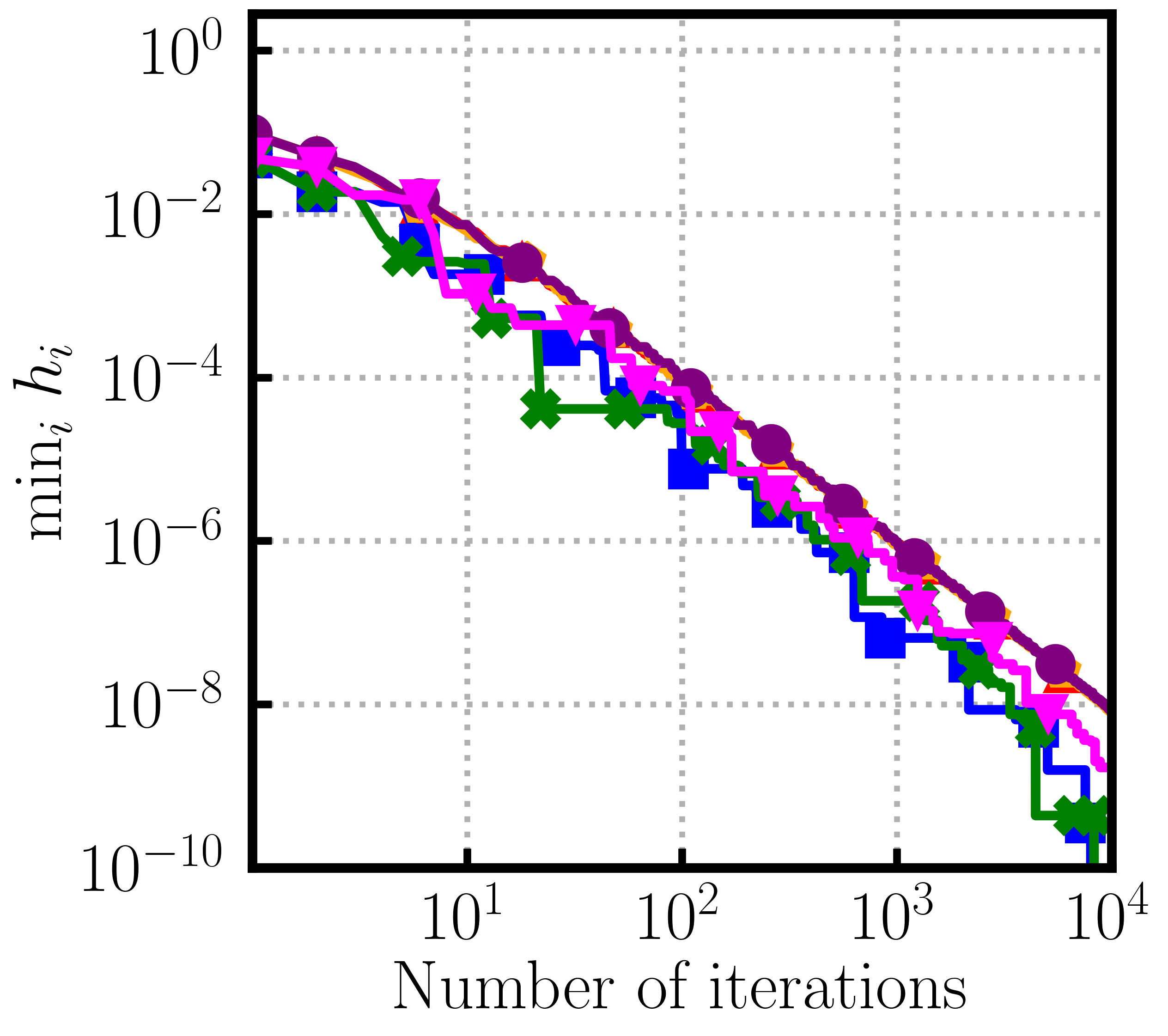}
        \caption{$\ell_1$-ball.}\label{fig:logistic_regression_1}
    \end{subfigure}& 
    \begin{subfigure}{.3\textwidth}
    \centering
        \includegraphics[width=1\textwidth]{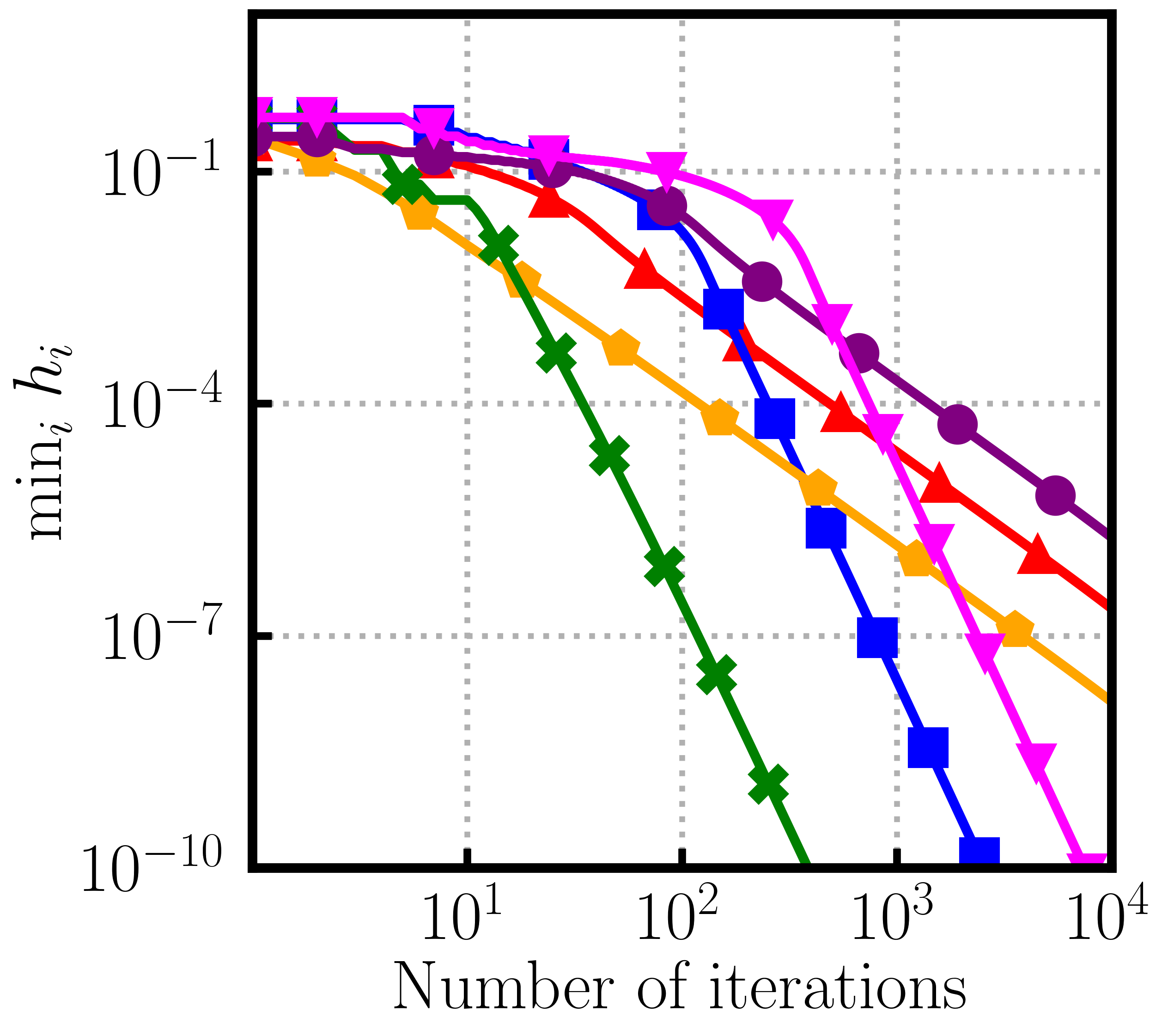}
        \caption{$\ell_2$-ball.}\label{fig:logistic_regression_2}
    \end{subfigure} & 
    \begin{subfigure}{.3\textwidth}
    \centering
        \includegraphics[width=1\textwidth]{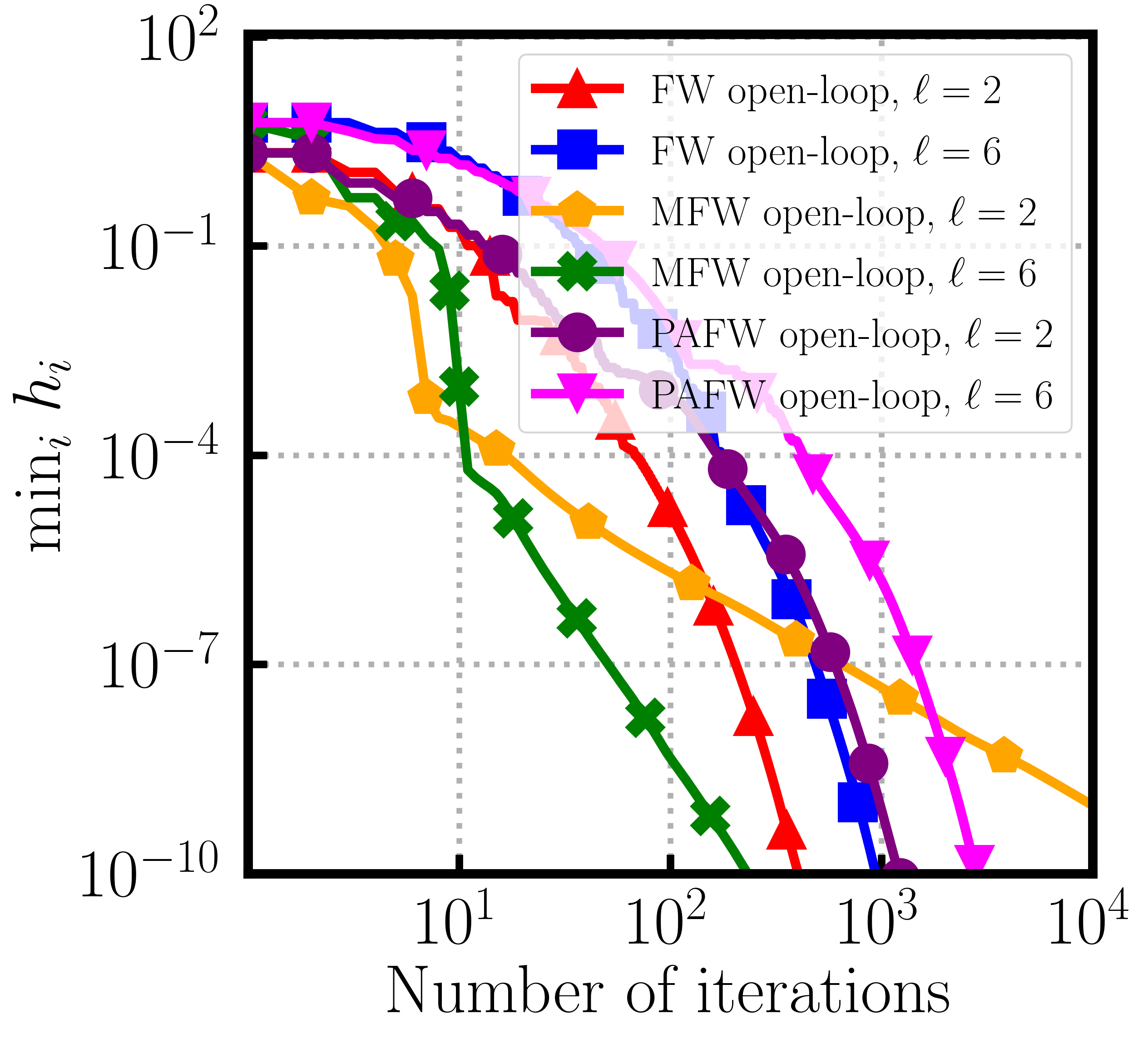}
        \caption{$\ell_5$-ball.}\label{fig:logistic_regression_5}
    \end{subfigure}\\
\end{tabular}
\caption{
Logistic regression for different $\ell_p$-balls.
}\label{fig:logistic_regression}
\end{figure}
We consider the problem of logistic regression, which for feature vectors $a_1,\ldots, a_m \in\R^d$, label vector $b\in\{-1,+1\}^m$, $p\in\R_{\geq 1}$, and radius $r > 0$, leads to the problem formulation
\begin{align*}
    \min_{x\in\R^d} & \frac{1}{m}\sum_{i=1}^m \log(1+\exp(-b_i a_i^\intercal x))\\
    \text{subject to} \ & \|x\|_p\leq r.
\end{align*}
Note that the feasible region is an $\ell_p$-ball and when $p=1$, the problem formulation is that of sparsity-constrained logistic regression, which induces sparsity in the iterates of FW variants.
For $p\in\{1,2,5\}$, we compare FW, PAFW, and MFW, with open-loop step-sizes $\eta_t=\frac{\ell}{t+\ell}$, where $\ell\in\{2,6\}$, on the Z-score normalized Gisette dataset\footnote{Available online at \href{https://archive.ics.uci.edu/ml/datasets/Gisette}{https://archive.ics.uci.edu/ml/datasets/Gisette}.} \citep{guyon2003introduction}. The number of features is $d=5,000$, we use $m=2,000$ samples of the dataset, and we set $r=1$. We plot the results of the experiments in log-log plots in Figure~\ref{fig:logistic_regression}.

PAFW and MFW seem to enjoy the same accelerated convergence rates as FW with step-sizes $\eta_t = \frac{\ell}{t+\ell}$, where $\ell\in\N_{\geq 1}$. This includes the rates of order $\cO(1/t^\ell)$ when $p\in\{2,5\}$, see also Remark~\ref{rem:ol_linear}. This raises the question whether PAFW \citep{lan2013complexity, kerdreux2021local} and MFW \citep{li2021momentum} admit accelerated convergence rates due to the exploitation of momentum, as indicated in the respective works, or due to the specific choice of open-loop step-size. 
Furthermore, MFW seems to converge at an accelerated rate earlier than FW, which converges at an accelerated rate earlier than PAFW. However, for $p = 5$, MFW converges quickly during early iterations but then converges at a slower rate than FW and PAFW, especially for step-size $\eta_t = \frac{2}{t+2}$.
For $p=1$, all methods converge at the same rate of order $\cO(1/t^2)$.

\subsection{Collaborative filtering}\label{sec:collaborative_filtering}
\begin{figure}[t]
\centering
\captionsetup[subfigure]{justification=centering}
\begin{tabular}{c}
    \begin{subfigure}{.3\textwidth}
    \centering
        \includegraphics[width=1\textwidth]{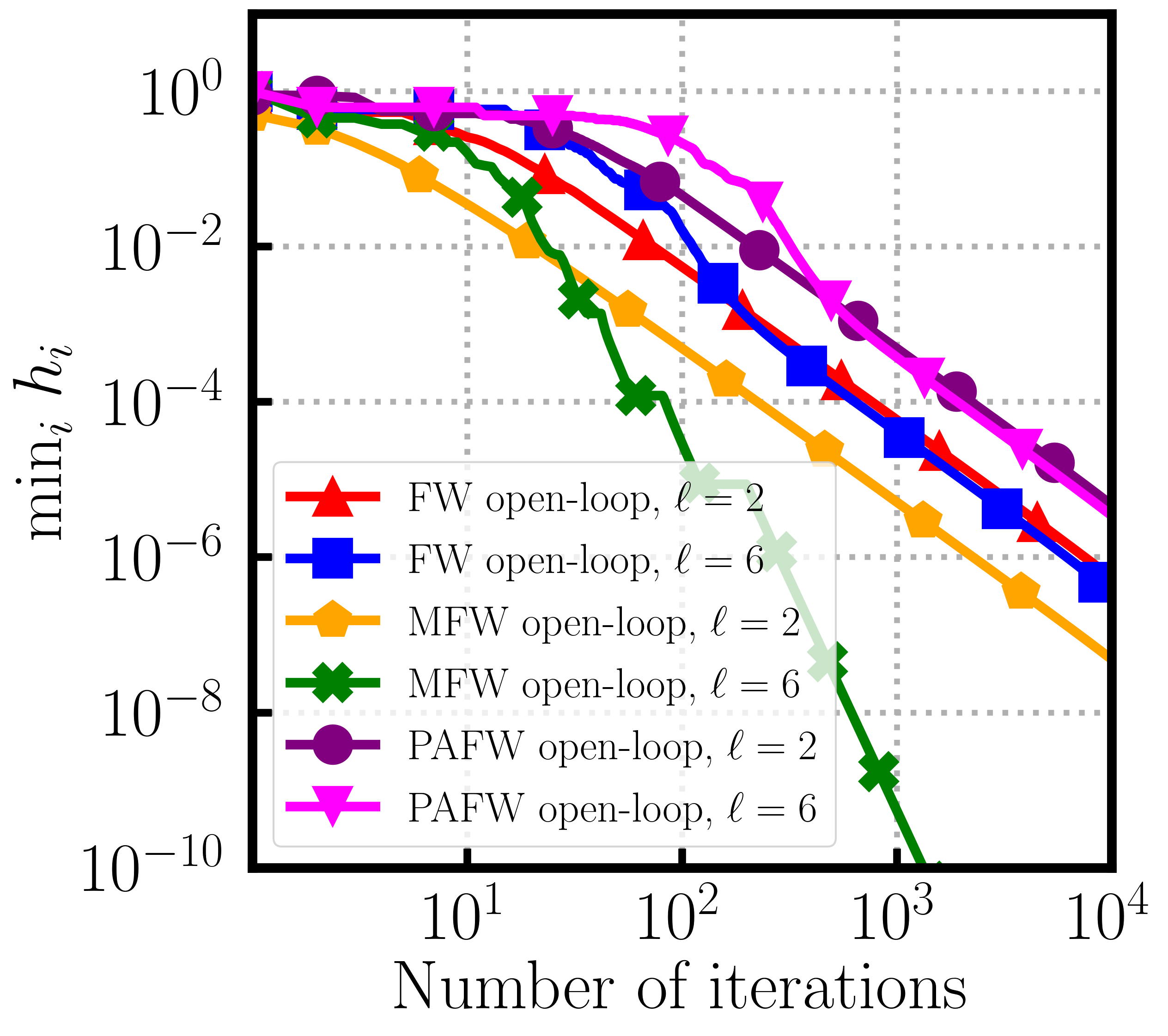}
    \end{subfigure}
\end{tabular}
\caption{
Collaborative filtering.
}\label{fig:collaborative_filtering}
\end{figure}
We consider the problem of collaborative filtering.
In particular, let $A\in\R^{m\times d}$ be a matrix with only partially observed entries, that is, there exists a subset of indices $\cI\subseteq \{1,\ldots,m\}\times\{1,\ldots,d\}$ such that only the entries $A_{i,j}$ with $(i,j)\in\cI$ are observed. The task is to predict the unobserved entries of $A$. Let $H_\rho$ be the Huber loss with parameter $\rho > 0$ \citep{huber1992robust}:
\begin{align*}
    H_\rho\colon x\in \R \mapsto  \begin{cases}
    \frac{x^2}{2}, & \text{if} \ |x| \leq \rho\\
    \rho(|x| - \frac{\rho}{2}), & \text{if} \ |x| > \rho,
    \end{cases}
\end{align*}
$\|\cdot\|_{\nuc}\colon X\in\R^{m\times d} \mapsto \trace(\sqrt{X^\intercal X})$ the nuclear norm, and $r>0$ the radius of the nuclear norm ball.
Since we assume the solution to be low rank, the approach of \citet{mehta2007robust} leads to the problem formulation
\begin{align*}
    \min_{X\in\R^{m\times d}} & \frac{1}{|\cI|} \sum_{(i,j)\in\cI} H_\rho(A_{i,j} - X_{i,j})\\
    \text{subject to} \ & \|X\|_{\nuc}\leq r.
\end{align*}
We compare FW, PAFW, and MFW, with open-loop step-sizes $\eta_t = \frac{\ell}{t+\ell}$, where $\ell\in\{2,6\}$, on the MovieLens 100k dataset\footnote{Available online at \href{https://grouplens.org/datasets/movielens/100k/}{https://grouplens.org/datasets/movielens/100k/}.} \citep{harper2015movielens} with $m=943$, $d=1682$, and $|\cI| = 10,000$, and we set $\rho = 1$ and $r = 2,000$.
We plot the results of the experiments in a log-log plot in Figure~\ref{fig:collaborative_filtering}. 

All algorithms with any step-size ultimately converge at a rate of order $\cO(1/t^2)$, except for MFW with step-size $\eta_t = \frac{6}{t+6}$, which appears to converge at a rate of order $\cO(1/t^6)$. The latter phenomenon is not currently motivated by results in this paper or \citet{li2021momentum}.
Among the different methods, MFW admits the fastest rate of convergence, followed by FW.
\section{{Discussion and open questions}}\label{sec:discussion}
We investigated settings in which FW with open-loop step-sizes achieves accelerated convergence rates. Specifically, we observed in Figures~\ref{fig:exterior} and~\ref{fig:logistic_regression} that FW with step-size $\eta_t=\frac{\ell}{t+\ell}$, where $\ell\in\N_{\geq 1}$, converges at a rate of order $\cO(1/t^\ell)$ when the feasible region $\cC$ is strongly convex and the norm of the gradient of $f$ is bounded from below by a nonnegative constant. These rates are better than the rates of order $\cO(1/t^{\ell/2})$ derived in Remark~\ref{rem:ol_linear}, which raises the question whether this gap between theory and practice can be closed. Furthermore, it remains to investigate the accelerated rates of order up to $\cO(1/t^\ell)$ when $\cC$ is only uniformly convex instead of strongly convex, see Figures~\ref{fig:exterior_3} and~\ref{fig:exterior_5}.
Furthermore, these convergence guarantees of order $\cO(1/t^{\ell/2})$ are significantly better than the convergence guarantees of order up to $\cO(1/t^2)$ of FW variants PAFW \citep{lan2013complexity, kerdreux2021local} and MFW \citep{li2021momentum}, which are designed to perform well in this setting.
We thus conducted numerical experiments to investigate whether PAFW and MFW also achieve accelerated rates depending on the choice of open-loop step-size. According to the logistic-regression experiments in Figure~\ref{fig:logistic_regression}, it appears that they do, which raises the question whether the accelerated convergence rates of PAFW and MFW stem from exploitation of momentum, as suggested in the respective works, or are in fact due to the choice of the open-loop step-size. 
The latter explanation is further supported by the unexplained convergence rate of order $\cO(1/t^6)$ of MFW with step-size $\eta_t = \frac{6}{t+6}$ in the collaborative filtering experiment in Figure~\ref{fig:collaborative_filtering}.
Further, we proved that FW with open-loop step-sizes achieves faster convergence rates than FW with line-search or short-step in the setting of the lower bound due to \citet{wolfe1970convergence}, assuming strict complementarity is satisfied. 
In case strict complementarity or similar assumptions are not satisfied, we proved that DIFW and AFW with open-loop step-sizes always converge at accelerated rates. 
We also answered the open question in \citet{bach2012equivalence} by demonstrating that FW with open-loop step-size $\eta_t = \frac{1}{t+1}$ achieves accelerated convergence rates in the setting of Section~\ref{sec:kernel_whaba} for the uniform density in Theorem~\ref{thm:answering_bach}. Numerical experiments in Figure~\ref{fig:kernel_herding_non_uniform} indicate that acceleration also holds for non-uniform densities, an observation which is currently not backed by theoretical results.
Finally, an important limitation of our study is that the proofs rely on norms, which are affine variant, whereas FW is known to be affine invariant. We plan to address this limitation in future work.

\subsubsection*{{Acknowledgements}}
This research was partially funded by the Deutsche Forschungsgemeinschaft (DFG, German Research Foundation) under Germany´s Excellence Strategy – The Berlin Mathematics Research Center MATH$^+$ (EXC-2046/1, project ID 390685689, BMS Stipend).

\bibliography{utils/biblio_infinite_dimensional_spaces.bib}


\end{document}

%% file: prologue.tex
\usepackage[utf8]{inputenc} 
\usepackage[T1]{fontenc}    
\usepackage{url}            
\usepackage{booktabs}       
\usepackage{amsfonts}       
\usepackage{microtype}      
\usepackage{bbold}
\usepackage{nicefrac} 
\usepackage{color, colortbl}
\usepackage[first=0,last=9]{lcg}
\definecolor{LightCyan}{rgb}{0.92,1,1}
\usepackage{subcaption}
\usepackage{amsmath}
\usepackage{amsthm}
\usepackage{amssymb}
\usepackage{thmtools,thm-restate}
\usepackage{thm-restate}
\usepackage{paralist}
\usepackage{comment}
\usepackage{dsfont}

\usepackage{bbm}
\usepackage{float}
\usepackage{balance}
\usepackage{stmaryrd}
\usepackage{multicol}
\usepackage[export]{adjustbox}
\usepackage{multirow}
\usepackage{tabularx}
\usepackage{mathtools}
\usepackage{mdframed}

\usepackage{enumitem}
\setitemize{noitemsep,topsep=0pt,parsep=0pt,partopsep=0pt}

\usepackage{chemmacros}

\usepackage[ruled, linesnumbered]{algorithm2e} %
\usepackage{setspace}
\usepackage{xcolor} %
\usepackage{tikz}
\usetikzlibrary{fit,calc}

\colorlet{pink}{red!40}
\colorlet{lightblue}{blue!30}
\colorlet{lightgreen}{green!30}

\DontPrintSemicolon

\usepackage{jmlr2e}
\usepackage[margin=1in]{geometry}

\usepackage{natbib}
\setcitestyle{authoryear,round,citesep={;},aysep={,},yysep={;}}
\renewcommand{\cite}[1]{\citep{#1}}

\def\fwt{{S}}

\DeclareMathOperator{\Iota}{I}

\newcommand{\E}{\mathbb{E}}

\newcommand{\N}{\mathbb{N}}

\newcommand{\R}{\mathbb{R}}

\newcommand{\Z}{\mathbb{Z}}

\newcommand\cC{{\ensuremath{\mathcal{C}}}\xspace}

\newcommand\cH{{\ensuremath{\mathcal{H}}}\xspace}
\newcommand\cI{{\ensuremath{\mathcal{I}}}\xspace}

\newcommand\cO{{\ensuremath{\mathcal{O}}}\xspace}

\newcommand\cS{{\ensuremath{\mathcal{S}}}\xspace}

\newcommand\cY{{\ensuremath{\mathcal{Y}}}\xspace}

\newcommand{\eps}{\varepsilon}

\DeclareMathOperator{\conv}{conv}
\DeclareMathOperator{\vertices}{vert}
\DeclareMathOperator{\aff}{aff}
\DeclareMathOperator{\supp}{supp}

\DeclareMathOperator{\mathspan}{span}

\DeclareMathOperator{\trace}{tr}
\DeclareMathOperator{\nuc}{nuc}

\DeclareMathOperator{\argmin}{argmin}
\DeclareMathOperator{\argmax}{argmax}

\newcommand{\zeroterm}{\ensuremath{\mathbb{0}}}
\newcommand{\oneterm}{\ensuremath{\mathbb{1}}}

\newcommand{\sabs}[1]{\left\lvert #1 \right\rvert}

\newcommand*{\vsepfbox}[1]{%
  \begingroup
    \sbox0{\fbox{#1}}%
    \setlength{\fboxrule}{0pt}%
    \mbox{\kern-\fboxsep\fbox{\unhbox0}\kern-\fboxsep}%
  \endgroup
}

\theoremstyle{plain} \numberwithin{equation}{section}
\newtheorem{theorem}{Theorem}[section]
\numberwithin{theorem}{section}
\newtheorem{lemma}[theorem]{Lemma}

\newtheorem{proposition}[theorem]{Proposition}

\theoremstyle{definition}
\newtheorem{definition}[theorem]{Definition}

\newtheorem{remark}[theorem]{Remark}

\theoremstyle{plain}
\newtheorem{assumption}{Assumption}

\graphicspath{{imgs/}}
\makeatletter
\def\mathcolor#1#{\@mathcolor{#1}}
\def\@mathcolor#1#2#3{%
  \protect\leavevmode
  \begingroup
    \color#1{#2}#3%
  \endgroup
}
\makeatother
\newcommand{\hrulealg}[0]{\vspace{1mm} \hrule \vspace{1mm}}